\documentclass{amsart}
\usepackage[english]{babel}
\usepackage{amssymb}
\usepackage{mathtools}
\usepackage{enumitem}
\usepackage{mathrsfs}
\usepackage[utopia]{mathdesign}
\usepackage[breaklinks]{hyperref}
\usepackage{xcolor}
\usepackage[bbgreekl]{mathbbol}
\usepackage{fullpage}
\usepackage{geometry}
\newcommand{\loc}{\textrm{loc}}
\newcommand{\N}{\mathbb{N}}
\newcommand{\R}{\mathbb{R}}
\newcommand{\tot}{\textrm{total}}
\newcommand{\eff}{\text{eff}}
\newcommand{\nh}{\text{nh}}

\renewcommand{\vec}[1]{{\bf #1 } }
\newcommand{\citep}[1]{\cite{#1}}
\newcommand{\sys}[1]{\left \{ \begin{aligned} #1 \end{aligned} \right. }
\newcommand{\saut}{{\color{white} h} \\}

\newcommand{\ini}{\mathrm{in}}
\newcommand{\eq}{\mathrm{eq}}

\newcommand{\nouveau}[1]{{#1}}

\newcommand{\etab}{\underline{\eta}}

\newcommand{\Vb}{\underline{V}}
\newcommand{\Ub}{ { \bf \underline{U}}}
\newcommand{\Pb}{\underline{P}}
\newcommand{\E}{\mathcal{E}}
\newcommand{\Elow}{\mathcal{E}_0}

\newcommand{\rhob}{\underline{\rho}}

\newcommand{\Mb}{\underline{\text{M}}}
\newcommand{\nablamu}{\nabla_{\mu}}
\newcommand{\jac}{J}
\newcommand{\gradphi}[1][]{\nabla_{#1}^{\varphi}}
\newcommand{\eul}[1]{{#1}^{\mathrm{eul}}}
\newcommand{\iso}[1]{{#1}^{\mathrm{iso}}}

\newcommand{\Ral}[2]{R_{#1}^{\mathbb{\Lambda}} (#2)}
\newcommand{\Run}{R_1}
\newcommand{\Rdeux}{R_2}
\newcommand{\Rtrois}{R_3}
\newcommand{\Rquatre}{R_4}
\newcommand{\Rcinq}{R_5}
\newcommand{\Rsix}{R_6}
\newcommand{\NLun}{ \left( \Ub + \epsilon \vec{U} \right) \cdot \nabla^{\varphi} \left( \Ub + \epsilon \vec{U}\right)}
\newcommand{\NLdeux}{\partial_i^{\varphi}\left( \Ub + \epsilon \vec{U}\right)_j \partial_j^{\varphi}\left( \Ub + \epsilon \vec{U}\right)_i}

\newtheorem{theorem}{Theorem}[section]
\newtheorem{lemma}[theorem]{Lemma}
\newtheorem*{lemma*}{Lemma}
\newtheorem{corollary}[theorem]{Corollary}
\newtheorem{proposition}[theorem]{Proposition}

\theoremstyle{definition}
\newtheorem{definition}[theorem]{Definition}

\theoremstyle{remark}
\newtheorem{remark}[theorem]{Remark}

\numberwithin{equation}{section}

\newcounter{hyp}

\newenvironment{hyp}{\begin{equation} \stepcounter{hyp} \tag{H.\thehyp}}{\end{equation}}

\AtBeginDocument{
   \def\MR#1{}
}

\begin{document}
\title[The Euler equations in a stably stratified ocean]{Well-posedness of the Euler equations in a stably stratified ocean in isopycnal coordinates}
\date{\today}
\author{Théo Fradin}
\address{Univ. Bordeaux, CNRS, Bordeaux INP, IMB, UMR 5251, F-33400 Talence, France }
\email{theo.fradin@math.u-bordeaux.fr}
\maketitle
\begin{abstract}
This article is concerned with the well-posedness of the incompressible Euler equations describing a stably stratified ocean, reformulated in isopycnal coordinates. Our motivation for using this reformulation is twofold: first, its quasi-2D structure renders some parts of the analysis easier. Second, it closes a gap between the analysis performed in the paper by Bianchini and Duchêne in 2022 in isopycnal coordinates, with shear velocity but with a regularizing term, and the analysis performed in the paper by Desjardins, Lannes, Saut in 2020 in Eulerian coordinates, without any regularizing term but without shear velocity. Our main result is a local well-posedness result in Sobolev spaces on the system in isopycnal coordinates, with shear velocity, without any regularizing term. The time of existence that we obtain is uniform with respect to the size $\epsilon$ of the perturbation, and boils down to the large time $1/\epsilon$ with the assumptions of the paper by Desjardins, Lannes, Saut in 2020. With additional assumptions, it is also uniform in the shallow-water parameter. The main difficulty consists in transposing to the isopycnal reformulation the symmetric structure of the system which is more straightforward in Eulerian coordinates. 
\end{abstract}
\section{Introduction}
\subsection{General setting}
The aim of this paper is to study equations describing the evolution of an ocean, that is in our context a strip of water in which we consider density variations. We start from the incompressible Euler equations in Eulerian coordinates, on a time interval $[0,T)$ where $T>0$, on the strip $S_z := \R^d \times [-H,0]$, with $d \in \{1,2\}$ being the horizontal dimension and $H > 0$ the height of the strip. This set of equations on the velocity $\eul{\vec{U}} := \begin{pmatrix} \eul{V} \\ \eul{w}  \end{pmatrix} $ is completed by the mass conservation on the density $\rho : \R^d \times[-H,0] \to [\rho^*,\rho_*]$ with $0 < \rho_* < \rho^*$ two constants. The Euler equations in this setting read
\begin{equation}
\tag{\ensuremath{\eul{E}}}
\label{eqn:euler_euleriennes}
 \sys{\partial_t \eul{\vec{U}} + (\eul{\vec{U}} \cdot \nabla)\eul{\vec{U}} + \frac{1}{\rho} \nabla \eul{P} - \vec{g} & = 0 ,\\
		\partial_t \rho + \eul{\vec{U}} \cdot \nabla \rho &= 0 , & \text{ \ in } [0,T) \times S_z.\\
		\nabla \cdot \eul{\vec{U}} &= 0 ,\\ }
\end{equation}
Here, $\nabla := (\partial_{x_1}, \dots, \partial_{x_d}, \partial_z )^T$ denotes the gradient, $\vec{g} = - g \vec{e_{\nouveau{d+1}}}$ with $g > 0$ is the acceleration of gravity and $\vec{e_{\nouveau{d+1}}}$ is the upward vertical unit vector. We assume a flat bottom and a rigid lid, so that the impermeability conditions yield the first boundary conditions below
\begin{equation}
\label{eqn:euler_euleriennes:bc}
\sys{\eul{w}_{|z=0} &= \eul{w}_{|z=-H} = 0,\\
	\rho_{|z=0} &= \rho_*, \\
	\rho_{|z=-H}  &= \rho^*,
} \text{ \ in } [0,T) \times \R^d;
\end{equation}
the last two boundary conditions state that the density is constant at the flat rigid lid and at the bottom. Note that this condition is propagated by the equations thanks to the transport equation on the density $\rho$ and the fact that the velocity field is tangent \nouveau{to} the boundary. \\
This set of equations is completed by the initial data
\begin{equation}
\label{eqn:euler_euleriennes:ci}
\left\{ \begin{aligned}
\eul{\vec{U}}_{|t=0} &= \eul{\vec{U}}_{\ini},\\
\rho_{|t=0} &= \rho_{\ini}, \end{aligned} \right. \qquad \text{ in } S_z,
\end{equation}
that should satisfy the divergence-free condition in \eqref{eqn:euler_euleriennes} and boundary conditions \eqref{eqn:euler_euleriennes:bc}. The well-posedness of the inhomogeneous (that is, when $\rho$ is not constant) incompressible Euler equations is well known,  see for instance \cite{Danchin2010} for a study in the whole space or \cite{Shigeharu1999} in a more general domain.\\

In this study we assume that the ocean is stratified, meaning that we can find a $C^1$ function $\eta: \R^d \times [0,1] \to [-H,0]$ with $ - \partial_r \eta \geq h_*$ for some $h_* > 0$ and with $\eta_{|r=0}=0$ and $\eta_{|r=1} = -H$ (and as such, $\eta$ is a surjection), such that the density is constant along the level sets of $\eta$ with respect to the third variable, i.e. for any $(t,x,r) \in [0,T] \times \R^d \times [0,1]$,
\begin{equation}
\label{eqn:rho_eta}
 \rho(t,x,\eta(t,x,r)) = \varrho(r),
\end{equation}
where $\varrho:  [0,1] \to [\rho_*, \rho^*]$ is thus the function mapping the level set $r$ to the density of the fluid on this level set, see Figure \ref{fig:eul}. Because of \eqref{eqn:rho_eta}, the assumption that $0 < \rho_* \leq \rho \leq \rho^*$ is now stated as
\begin{hyp}
\label{hyp:non_cavitation}
\rho_* \leq \varrho \leq \rho^*.
\end{hyp}
\begin{figure}
\includegraphics[scale=1]{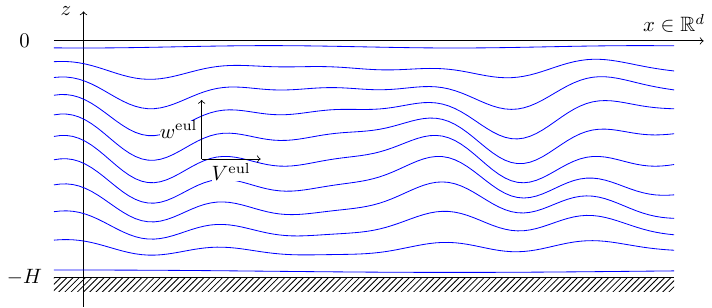}
\caption{Setting in Eulerian coordinates. Isopycnals are sets where the density is constant.}\label{fig:eul}
\end{figure}This yields a diffeomorphism
$$ \varphi: \left| \begin{aligned} [0,T) \times S_{r} & \to [0,T) \times  S_z \\
								(t,x,r) & \mapsto \varphi(t,x,r) := (t,x,\eta(t,x,r)). \end{aligned} \right.$$
Here, $S_{r}$ is the domain $ \R^d \times [0,1]$ of the so-called isopycnal coordinates, see Figure \ref{fig:iso}.\\
\begin{figure}
\includegraphics[scale=1]{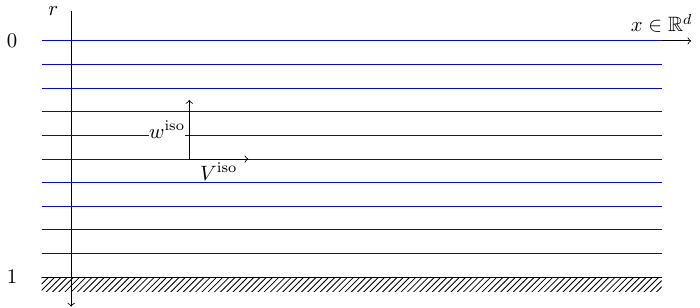}
\caption{Setting in isopycnal coordinates. Isopycnals are flattened.}\label{fig:iso}
\end{figure}Writing the Euler equations in isopycnal coordinates means applying the change of coordinates $\varphi$ to the system \eqref{eqn:euler_euleriennes}. To this end, we define the new set of unknowns 
$$\begin{aligned} \iso{\vec{U}}&:= \eul{\vec{U}}\circ \varphi, & \qquad 
		 \iso{V}&:= \eul{V}\circ \varphi, \\
		  \iso{w}&:= \eul{w}\circ \varphi, & \qquad 
		   \iso{P}&:= \eul{P}\circ \varphi, \\
\end{aligned} 
$$ 
as well as the operators
$$ \begin{aligned}
\partial_t^{\varphi}\iso{f} &:= (\partial_t \eul{f}) \circ \varphi, & \qquad 
\gradphi[x] \iso{f} &:= (\nabla_x \eul{f}) \circ \varphi, \\
\partial_r^{\varphi}\iso{f} &:= (\partial_r \eul{f}) \circ \varphi, & \qquad 
\gradphi &= \begin{pmatrix} \gradphi[x] \\ \partial_r^{\varphi} \end{pmatrix},
\end{aligned}$$
for $\eul{f}$ an unknown in Eulerian coordinates and $\iso{f} := \eul{f} \circ \varphi $ its expression in isopycnal coordinates. We thus obtain the first reformulation of the momentum equation and divergence-free condition in isopycnal coordinates 

\begin{equation*}
\sys{ \partial_t^{\varphi} \iso{\vec{U}} + (\iso{\vec{U}} \cdot \gradphi)\iso{\vec{U}} + \frac{1}{\varrho} \gradphi \iso{P} - \vec{g} & = 0, \\
		\gradphi \cdot \iso{\vec{U}} &= 0, \\ } \qquad  \text{ \ in } [0,T) \times S_r.
\end{equation*}
This set of equations is completed with the boundary conditions
\begin{equation}
\label{eqn:euler:bc}
\sys{\iso{w}_{|r = 0} &= \iso{w}_{|r = 1} = 0 \\ 
	\eta_{|r=0} &= 0,\\
	\eta_{|r=1} &= -H,}\text{ \ in } [0,T) \times \R^d,
\end{equation}
where the conditions on $\eta$ come from its definition. They imply the conditions on $\rho$ in \eqref{eqn:euler_euleriennes:bc} as they state that the bottom and top boundaries are level-lines of the density. This system is completed with the initial data
\begin{equation*}
\sys{		\iso{\vec{U}}_{|t=0} &= \iso{\vec{U}}_{\ini}, \\
		\eta_{|t=0} &= \eta_{\ini},} \qquad \text{ \ in } S_r,
\end{equation*}
that should satisfy \eqref{eqn:euler:bc} and the divergence-free condition in isopycnal coordinates. Using the chain rule, we can compute
$$\begin{aligned} \gradphi[t,x] \iso{f} &= \nabla_{t,x} \iso{f} - \frac{\nabla_{t,x} \eta}{\partial_r \eta} \partial_r \iso{f}, \\
\partial_r^{\varphi} \iso{f} &= \frac{1}{\partial_r \eta} \partial_r \iso{f}, \end{aligned} $$
for any function $\iso{f}$ defined in isopycnal coordinates. \\
Thus, we need to complement the momentum equation in isopycnal coordinates with an equation on $\eta$. This is done by differentiating \eqref{eqn:rho_eta} with respect to $x$ and $t$, namely
$$ \sys{ \nabla_x \rho(t,x,\eta(t,x,r)) + \nabla_x \eta\nouveau{(t,x,r)} \partial_z \rho(t,x,\eta(t,x,r)) = 0 , \\
	\partial_t \rho(t,x,\eta(t,x,r)) + \partial_t  \eta\nouveau{(t,x,r)} \partial_z \rho(t,x,\eta(t,x,r)) = 0 .}$$
	Using the mass conservation equation in \eqref{eqn:euler_euleriennes}, this yields
	\begin{equation}
	\label{eqn:euler_isopycnales:eta}
	\left(\partial_t \eta + V^{\nouveau{\mathrm{iso}}} \cdot \nabla_x \eta - w^{\nouveau{\mathrm{iso}}}\right)\nouveau{\circ \varphi^{-1}}\partial_z \rho = 0.
	\end{equation}
Now restricting to the case of a stable stratification, that is, there exists $c > 0$ such that $-\partial_z \rho \geq c > 0$, the above equation \eqref{eqn:euler_isopycnales:eta} yields an evolution equation on $\eta$.
Note that differentiating \eqref{eqn:rho_eta} with respect to $r$ yields
$$  \partial_r \eta  \partial_z \rho = \varrho'.$$
Thus, as $- \partial_r \eta \geq h_* > 0$, the stable stratification assumption is equivalent to the existence of $c_* > 0$ such that $\varrho' \geq c_*$.\\
Using \eqref{eqn:euler_isopycnales:eta}, we can compute 
\begin{equation}
\label{eqn:slag}
\partial^{\varphi}_t + \iso{\vec{U}} \cdot \gradphi = \partial_t + \iso{V} \cdot \nabla_x \ .
\end{equation}
We can now write the system in isopycnal coordinates according to this remark, namely
\begin{equation}
\tag{\ensuremath{\iso{E}}}
\label{eqn:euler_isopycnal}
\sys{ \partial_t \iso{\vec{U}} + (\iso{V} \cdot \nabla_x) \iso{\vec{U}} + \frac{1}{\varrho} \gradphi \iso{P} - \vec{g} & = 0, \\
		\partial_t \eta + \iso{V} \cdot \nabla_x \eta  - \iso{w} &= 0, & \qquad \text{ in } [0,T) \times S_r.\\
		\gradphi \cdot \iso{\vec{U}} &= 0, \\}
\end{equation}
It is important to notice that there is no vertical advection in the formulation  \eqref{eqn:euler_isopycnal} of the system in isopycnal coordinates, as opposed to the formulation in Eulerian coordinates. For this reason, \eqref{eqn:euler_isopycnal} is sometimes said to be quasi two-dimensional \cite{Vallis2017} or semi-Lagrangian \cite{Duchene2022}. In other words, the flow is largely along the isopycnals \cite{Vallis2017}, hence the study of large-scale flows in the ocean is sometimes performed in the isopycnal coordinates \cite{Griffies2000}. From a technical viewpoint, we crucially use this semi-Lagrangian property in this study for the construction of a solution from a regularized system, which is less obvious in Eulerian coordinates, see Remark \ref{rk:slag_existence}.  \\ 
\subsection{Ansatz}
The system in Eulerian coordinates \eqref{eqn:euler_euleriennes} admits equilibria called shear flows, namely of the form
$$(\eul{V}, \eul{w}, \rho)(t,x,z) := (\eul{V}_{\eq}, 0, \rho_{\eq})(z).$$
The pressure $\eul{P}_{\eq}$ is then given by the equation on $w$ in \eqref{eqn:euler_euleriennes}, and is called the hydrostatic pressure
$$ \eul{P}_{\eq}(z) := \int_z^0 g \rho_{\eq}(z) dz,$$
where we assumed without loss of generality that the pressure at equilibrium at the rigid lid is zero.
In Eulerian coordinates, we make the natural ansatz of studying a perturbation of size $\epsilon \in [0,1]$ of a reference solution that is an equilibrium:   
$$ (\eul{ \Vb}(z), 0, \eul{ \Pb}(z),  \rhob(z)) := (\eul{V}_{\eq}(z), 0, \eul{P}_{\eq}(z), \rho_{\eq}(z)).$$
Namely, we look for solutions under the form
$$(\eul{V}_{\tot}, \eul{w}_{\tot}, \eul{P}_{\tot}, \rho_{\tot}) := (\eul{\Vb}, 0, \eul{\Pb}, \rhob) + \epsilon (\eul{\tilde{V}}, \eul{\tilde{w}}, \eul{\tilde{P}}, \tilde{\rho}).$$
The system \eqref{eqn:euler_euleriennes} then writes, in $[0,T) \times S_z$:
\begin{equation}
\label{eqn:euler_euleriennes:pert}
 \sys{\partial_t \eul{\tilde{V}} + (\eul{\Vb} + \epsilon \eul{\tilde{V}}) \cdot \nabla_x\eul{\tilde{V}} + \eul{\tilde{w}} \partial_z  \left( \eul{\Vb} + \epsilon \eul{\tilde{V}}\right)+  \frac{1}{\rhob + \epsilon \tilde{\rho}} \nabla_x \eul{\tilde{P}} & = 0, \\
\partial_t \eul{\tilde{w}} +  (\eul{\Vb} + \epsilon \eul{\tilde{V}}) \cdot \nabla_x\eul{\tilde{w}} + \epsilon \eul{\tilde{w}} \partial_z \eul{\tilde{w}}  +  \frac{1}{\rhob + \epsilon \tilde{\rho}} \partial_z \eul{\tilde{P}}  + \frac{\nouveau{g}\tilde{\rho}}{\rhob + \epsilon \tilde{\rho}} & = 0, \\
		\partial_t \tilde{\rho} + (\eul{\Vb} + \epsilon \eul{\tilde{V}})\cdot \nabla_x \tilde{\rho}  + \eul{\tilde{w}} \partial_z(\rhob + \epsilon \tilde{\rho}) &= 0, \\
		\nabla_x \cdot \eul{\tilde{V}} + \partial_z \eul{\tilde{w}} &= 0. \\}
\end{equation}
In \cite{Desjardins2019}, the authors study the system \eqref{eqn:euler_euleriennes:pert} in this particular setting, with the assumption of stable stratification, i.e. that the density is strictly decreasing with respect to the vertical coordinate. Under several technical assumptions, one of which being the absence of background shear-flow, they prove the well-posedness of the system \eqref{eqn:euler_euleriennes:pert}, in Sobolev spaces of finite regularity, on a time interval $[0,T)$ with $T$ of order $O(\frac{1}{\epsilon})$ (see \cite[Theorem 2]{Desjardins2019}). Regarding this issue, Theorem \ref{thm:wp} is an extension of \cite[Theorem 2]{Desjardins2019} to the case of a non-zero shear flow. Note however that Theorem \ref{thm:wp} yields a shorter time of existence than \cite[Theorem 2]{Desjardins2019}, see Remark \ref{rk:short_time}.\\

Let us now turn to the reformulation in isopycnal coordinates.  
The system in isopycnal coordinates \eqref{eqn:euler_isopycnal} also admits equilibrium solutions $(\iso{V}_{\eq}, 0,\iso{P}_{\eq}, \eta_{\eq})$ that only depend on the vertical variable $r$ and satisfying the hydrostatic balance
$$ \partial_r^{\varphi_{\eq}} \iso{P}_{\eq} = - g \varrho, $$
where $\varphi_{\eq}(t,x,r) := (t,x,\eta_{\eq}(r))$. We make the particular choice $\eta_{\eq}(r) = -rH$. Indeed, if $\rho_{\eq}$ is the density at equilibrium, we can write, for $r \in [0,1]$,
\begin{equation}
\label{eqn:intro:temp1}
\rho_{\eq}(\eta_{\eq}(r)) = \rho_{\eq}(-rH) =: \varrho(r),
\end{equation}
so that the change of coordinates $\varphi_{\eq}(t,x,r) = (t,x,\eta_{\eq}(r))$ does verify
$$\begin{aligned}
\nouveau{\eta_{\eq}(0)} &\nouveau{= 0}, \qquad & \nouveau{\eta_{\eq}(1) = - H},\\
h_{\eq}(r) &:= -\partial_r \eta_{\eq}(r) = H > 0.& 
\end{aligned}
$$
Note also that any other choice of $\eta_{\eq}$ would lead to a different function $\varrho$ through the relation \eqref{eqn:intro:temp1}, ultimately describing the same physical system, so that our choice for $\eta_{\eq}$ is not restrictive. We thus have
$$ (\iso{V}_{\eq}, \iso{P}_{\eq})(r) = (\eul{V}_{\eq}, \eul{P}_{\eq})(-rH).$$
In \cite{Duchene2022}, the authors study a perturbation of such an equilibrium. Let us make a slightly different ansatz, and look for solutions $(\iso{V}_{\tot}, \iso{w}_{\tot}, \iso{P}_{\tot}, \eta_{\tot})$ of the form 
$$ (\iso{V}_{\tot}, \iso{w}_{\tot}, \iso{P}_{\tot}, \eta_{\tot}) := (\iso{\Vb} + \epsilon \iso{\tilde{V}}, \epsilon \iso{\tilde{w}}, \iso{\Pb} + \epsilon \iso{\tilde{P}}, \etab + \epsilon \tilde{\eta}),$$
where $\epsilon \in [0,1]$ and 
$$\sys{
 \etab(r) &:= \eta_{\eq}(r) = -rH, \\
\iso{\Vb}(r) &:= \iso{V}_{\eq}(r), \\
\iso{ \Pb}(t,x,r) &:= \iso{P}_{\eq}\left(- \frac{1}{H} \left( \etab(r)  + \epsilon \tilde{\eta}(t,x,r) \right) \right). }$$
We can thus write the change of coordinates as
$$ \varphi(t,x,r) := (t,x,\etab(r) + \epsilon \tilde{\eta}(t,x,r)). $$
Remark that the ansatz on the pressure is such that 
$$\sys{\gradphi[x] \iso{\Pb} &= \left( \nabla_x \eul{P}_{\eq}\right)\circ (\etab  \nouveau{+} \epsilon \tilde{\eta}) = 0,\\
\nouveau{\partial^{\varphi}_r \iso{\Pb}} &= \nouveau{\left( \partial_z \eul{P}_{\eq} \right)\circ ( \etab + \epsilon \tilde{\eta}) = -g \rhob(\etab + \epsilon \eta)},
}
$$
from the definition of $\eul{P}_{\eq}$. These properties will simplify the analysis later on. Note that a Taylor expansion on $\iso{P}_{\eq}$ yields
$$ \iso{\Pb}(t,x,r) = \iso{P}_{\eq}(r) + O(\epsilon),$$
so that this ansatz and the one used in \cite{Duchene2022} consisting in studying directly a perturbation of the equilibrium correspond up to terms of order $\epsilon$. \\
Using the notation
\begin{equation*}
b := \frac{1}{\epsilon} \left( 1 - \frac{\nouveau{\rhob}\circ(\nouveau{\etab} + \epsilon \tilde{\eta})}{\nouveau{\rhob}\circ \nouveau{\etab}} \right) g
\end{equation*}
for the buoyancy term, this yields the system
\begin{equation}
\label{eqn:euler_slag_dim}
\sys{ \partial_t \iso{\tilde{V}} + (\iso{\Vb} + \epsilon \iso{\tilde{V}}) \cdot \nabla_x \iso{\tilde{V}} + \frac{1}{\varrho} \gradphi[x] \iso{\tilde{P}}  & = 0, \\
\left(\partial_t \iso{\tilde{w}} + (\iso{\Vb} + \epsilon \iso{\tilde{V}}) \cdot \nabla_x \iso{\tilde{w}} \right) + \frac{1}{\varrho} \partial_r^{\varphi} \iso{\tilde{P}}  + b  & = 0, \\
		\partial_t \tilde{\eta} + (\iso{\Vb} + \epsilon \iso{\tilde{V}}) \cdot \nabla_x \eta  - \iso{\tilde{w}} &= 0, \\
		\nabla_x^{\varphi} \cdot \iso{\tilde{V}}  + \frac{1}{\epsilon} \gradphi[x] \cdot \iso{\Vb} + \partial_r^{\varphi}\iso{\tilde{w}} &= 0, \\}  \qquad \text{ in } [0,T) \times S_r,
\end{equation}
with the boundary conditions
\begin{equation}
\sys{
		\iso{\tilde{w}}_{|r = 0} &= \iso{\tilde{w}}_{|r=1} = 0,\\
		\tilde{\eta}_{|r=0} &= \tilde{\eta}_{|r=1} = 0,} \qquad \text{ in } [0,T) \times \R^d.
\end{equation}
Note that the conditions on $\tilde{\eta}$ are such that $\etab + \epsilon \tilde{\eta}$ satisfy the conditions \eqref{eqn:euler:bc}. This system is completed with the initial conditions
\begin{equation}
\label{eqn:euler_slag_dim:ci}
\sys{	\iso{\tilde{V}}_{|t=0} &= \iso{\tilde{V}}_{\ini},\\
		\iso{\tilde{w}}_{|t=0} &= \iso{\tilde{w}}_{\ini}, & \qquad \text{ in } S_r.\\
		\tilde{\eta}_{|t=0} &= \tilde{\eta}_{\ini}, }
\end{equation}
Remark that a Taylor expansion of $\rhob$ around $\etab(r)$ yields the approximate expression
$$ b(t,x,r) = \nouveau{\frac{g}{H}}\frac{\varrho'(r)}{\varrho(r)} \tilde{\eta}(t,x,r) + O(\epsilon).$$
Recall indeed that $\rho_{\eq}(\eta_{\eq}(r)) = \varrho(r)$ by definition. This indicates that $b$ is of order $1$, and not of order $1/\epsilon$. It also makes appear the quantity 
\begin{equation}
\label{eqn:N2}
N^2(r) := \frac{\varrho'(r)}{\varrho(r)},
\end{equation}
which is the Brunt-Väisälä frequency. This quantity plays an important role in the dynamics, see for instance \cite{Vallis2017}, or \cite{Desjardins2019} for its role in the coupling of internal waves. This argument is made rigorous in Lemma \ref{lemma:NL} thanks to composition estimates. \\

A variation of the system \eqref{eqn:euler_slag_dim} has already been studied in \cite{Duchene2022}, though with the more natural ansatz of considering directly a perturbation of an equilibrium and in the free-surface setting. More importantly, the authors consider a diffusion term stemming from mesoscale turbulence introduced in \cite{Gent1990}, which regularizes the system. One striking fact is that they crucially use the regularizing properties of the diffusion term to prove the well-posedness of the system \eqref{eqn:euler_slag_dim}, whereas the analogous system \eqref{eqn:euler_euleriennes:pert} in Eulerian coordinates has been proved to be well-posed without this term in \cite{Desjardins2019}. \\
This discrepancy motivates the present investigation of the well-posedness of this system in isopycnal coordinates without the regularization used in \cite{Duchene2022}. Regarding this issue, Theorem \ref{thm:wp} is an extension of the existence result in \cite[Theorem \nouveau{2}.2]{Duchene2022}. 
\subsection{Non-dimensionalization}
We now turn to the non-dimensionalization of the equations, first in Eulerian coordinates. We write $\tilde{x} = \frac{x}{L}$, $\tilde{t} = \frac{\sqrt{gH}}{L}t$, $ \tilde{z} = \frac{z}{H}$, where $L$ and $H$ are characteristic horizontal and vertical scales, respectively.
We write $\mu := \frac{H^2}{L^2}$ for the shallowness parameter. In Eulerian coordinates, the system \eqref{eqn:euler_euleriennes:pert} becomes
\begin{equation}
\label{eqn:euler_euleriennes:mu}
 \sys{\partial_t \eul{\tilde{V}} + (\eul{\Vb} + \epsilon \eul{\tilde{V}}) \cdot \nabla_x\eul{\tilde{V}} + \eul{\tilde{w}} \partial_z  \left( \eul{\Vb} + \epsilon \eul{\tilde{V}}\right)+  \frac{1}{\rhob + \epsilon \tilde{\rho}} \nabla_x \eul{\tilde{P}} & = 0, \\
 \mu \left(\partial_t \eul{\tilde{w}} +  (\eul{\Vb} + \epsilon \eul{\tilde{V}}) \cdot \nabla_x\eul{\tilde{w}} + \epsilon \eul{\tilde{w}} \partial_z \eul{\tilde{w}} \right)  +  \frac{1}{\rhob + \epsilon \tilde{\rho}} \partial_z \eul{\tilde{P}}  + \frac{\tilde{\rho}}{\rhob + \epsilon \tilde{\rho}} & = 0, \\
		\partial_t \tilde{\rho} + (\eul{\Vb} + \epsilon \eul{\tilde{V}})\cdot \nabla_x \tilde{\rho}  + \eul{\tilde{w}} \partial_z(\rhob + \epsilon \tilde{\rho}) &= 0, \\
		\nabla_x \cdot \eul{\tilde{V}} + \partial_z \eul{\tilde{w}} &= 0. \\}
\end{equation}
We are interested in the regime $\mu << 1$, as for oceanic applications one can have $\mu \sim 10^{-6}$ (see \cite{Vallis2017}). Therefore, it is natural to look for a time of existence at least uniform with respect to $\mu$. The hydrostatic equations (that is, setting $\mu = 0$ in \eqref{eqn:euler_euleriennes:mu}) are very singular: in the homogeneous case (that is, $\rhob$ being constant and $\tilde{\rho}$ being zero in \eqref{eqn:euler_euleriennes:mu}), ill-posedness of the initial value problem of the linearized system (that is, setting $\epsilon = 0$ in \eqref{eqn:euler_euleriennes:mu}) was proved by Renardy in \cite{Renardy2009}. However, Masmoudi and Wong showed in \cite{Masmoudi2012} existence and uniqueness of solutions for the two-dimensional, homogeneous case, under the Rayleigh criterion
$$ \partial_z^2 \left( \Vb + \epsilon \eul{\tilde{V}} \right) \geq c_* > 0.$$
In the case of a stable stratification, the Miles and Howard criterion 
\begin{equation} 
\label{eqn:mh} \inf_{z \in [-H,0]} g \frac{-\rho'_{\eq}(z)}{\rho_{\eq}(z)|V'_{\eq}(z)|^2} \geq \frac14 
\end{equation}
(see \cite{Miles1961,Howard1961}) gives a sufficient condition to prevent unstable modes for the linearized equations around an equilibrium shear flow $(\rho_{\eq}(z), V_{\eq}(z))$. However, this is not enough to conclude well-posedness of the hydrostatic equations, which is still an open problem. Nonetheless, in \cite{Bianchini2024}, the authors found a stratified shear flow around which the hydrostatic system is ill-posed. \\

In \cite{Desjardins2019}, using the stable stratification assumption, but also restricting to the case of uniform stratification (that is, $\rho_{\eq}$ is an exponential), using the Boussinesq approximation (that is, the density is assumed constant in the momentum equations, except in the buoyancy term $b$), with well-prepared initial data and no shear flow, the authors prove the existence and stability of solutions on a time interval $[0,T)$ with $T$ of order $\frac{\sqrt{\mu}}{\epsilon}$. In Theorem \ref{thm:wp} below, we state a similar result in the case of a non-zero shear flow, without those additional assumptions. However our analysis provides a time of existence uniform in $\epsilon$ and $\mu$. A more thorough comparison with their result is provided in Remark \ref{rk:short_time}.   \\

In isopycnal coordinates, equations \eqref{eqn:euler_slag_dim} become, in non-dimensional form and {\it dropping the tildes and the superscripts `iso' for readability }
\begin{equation}
\tag{$E$}
\label{eqn:euler_slag}
\sys{ \partial_t V + (\Vb + \epsilon V) \cdot \nabla_x V + \frac{1}{\varrho} \gradphi[x] P  & = 0, \\
\mu \left(\partial_t w + (\Vb + \epsilon V) \cdot \nabla_x w \right) + \frac{1}{\varrho} \partial_r^{\varphi} P  + b  & = 0, \\
		\partial_t \eta + (\Vb + \epsilon V) \cdot \nabla_x \eta  - w &= 0, \\
		\nabla_x^{\varphi} \cdot V  + \frac{1}{\epsilon} \gradphi[x] \cdot \Vb + \partial_r^{\varphi}w &= 0, } \qquad \text{ in } [0,T) \times S_r.
\end{equation}
Here, the buoyancy term $b$ reads, in non-dimensional form:
\begin{equation}
\label{eqn:def:b}
b := \frac{1}{\epsilon} \left( 1 - \frac{\nouveau{\rhob}(\nouveau{\etab} + \epsilon \eta)}{\nouveau{\rhob}(\nouveau{\etab})} \right). 
\end{equation}
The system \eqref{eqn:euler_slag} is completed with the initial conditions (stemming from the non-dimensionalization of) \eqref{eqn:euler_slag_dim:ci} as well as the boundary conditions 
\begin{equation}
\label{eqn:euler_slag:bc}
\sys{ w_{|r=0} &= w_{|r=1} = 0,\\
		\eta_{|r=0}&=\eta_{|r=1}=0,}
\end{equation}
where the first two identities in \eqref{eqn:euler_slag:bc} are the impermeability conditions of the top and bottom boundary of the fluid domain, and the last two identities in \eqref{eqn:euler_slag:bc} denote the assumption that the density at the top and bottom boundaries is constant.\\
Note that the existence result in \cite[Theorem \nouveau{2}.2]{Duchene2022} gives a time of existence uniform in the shallow-water parameter $\mu$, \nouveau{in the free-surface setting}, though by using a regularizing term.

\subsection{Main result}
We now state our main result, that is the well-posedness of the system \eqref{eqn:euler_slag} in Sobolev spaces under the assumption of stable stratification, with a time of existence uniform in $\epsilon$. Assuming further a medium amplitude asymptotic regime ($\epsilon \leq \sqrt{\mu}$) and small shear velocity $\Vb'$, it is also uniform in $\mu$. \\
Let us state the classical assumptions used in Theorem \ref{thm:wp}. Let $ s \in \N$. First, let $\Mb > 0$. We assume that $\Vb, \varrho \in \nouveau{W^{s+1,\infty}([0,1])}$, with 
\begin{hyp}
\label{hyp:Mb}
| \Vb |_{W^{s+1,\infty}} + |\varrho|_{W^{s+1,\infty}} \leq \Mb,
\end{hyp}as well as the strict stability assumption, that is there exists $c_* > 0$ such that
\begin{hyp}
\label{hyp:stable}
\varrho' \geq c_* .
\end{hyp}Second, let $M_{\ini} > 0$. If $(V_{\ini},w_{\ini},\eta_{\ini}) \in H^{s}(S_r)^{d+2}$ correspond to initial data, then we assume
\begin{hyp}
\label{hyp:Min}
 \Vert V_{\ini} \Vert_{H^{s}}^2 + \mu \Vert w_{\ini} \Vert_{H^{s}}^2 + \Vert \eta_{\ini} \Vert_{H^{s}}^2 \leq M_{\ini}.
\end{hyp}Moreover, the initial data satisfy the boundary condition and divergence-free condition 
\begin{hyp}
\label{hyp:wp}
\begin{aligned} & \left. w_{\ini} \right|_{r=0} = \left. w_{\ini} \right|_{r=1} = 0, \\
&\left. \eta_{\ini} \right|_{r=0} = \left. \eta_{\ini} \right|_{r=1} = 0, \\
					& \nabla_x^{\varphi_{\ini}} \cdot \left( \Vb + \epsilon V_{\ini}\right)  + \epsilon \partial_r^{\varphi_{\ini}} w_{\ini} = 0, \end{aligned} 
\end{hyp}with the notation
	$$\begin{aligned}
		\nabla_x^{\varphi_{\ini}} &:= \nabla_x + \epsilon \frac{\nabla_x \eta_{\ini}}{1+ \epsilon h_{\ini}} \partial_r, \\
		\partial_r^{\varphi_{\ini}} &:= \frac{-1}{1+\epsilon h_{\ini}} \partial_r ,
	\end{aligned} $$
where 
$$h_{\ini} := - \partial_r \eta_{\ini}.$$
We assume that there exist $h^* \geq h_* > 0$ such that 
\begin{hyp}
\label{hyp:h}
h_* \leq 1+\epsilon h_{\ini} \leq h^*.
\end{hyp}
\begin{theorem}  
\label{thm:wp}
Let $d \in \N^*$, $s_0>\frac{d}{2}$, $s \in \N$ with $s \geq s_0+\frac52$. Let $\epsilon \in [0,1]$, $\mu \in ]0,1]$. Let $\Vb, \varrho \in W^{s+1,\infty}([0,1])$ , $(V_{\ini},w_{\ini},\eta_{\ini}) \in H^{s}(S_r)^{d+2}$ such that assumptions \eqref{hyp:non_cavitation}, \eqref{hyp:Mb}, \eqref{hyp:stable}, \eqref{hyp:Min},  \eqref{hyp:wp}, \eqref{hyp:h} are satisfied. Assume further that the shear flow is small with respect to $\mu$ and that we are in the weakly non-linear regime
$$\begin{aligned} |\Vb'|_{L^\infty} &\leq \sqrt{\mu}, \\
					 \epsilon &\leq \sqrt{\mu}. \end{aligned} $$
Then there exists $T > 0$ independent of $\epsilon,\mu$ such that there exists a unique solution $(V,w,\eta) \in C^0([0,T),H^{s}(S_r)^{d+2})$ to \eqref{eqn:euler_slag} with the initial data $(V_{\ini},w_{\ini},\eta_{\ini})$ and boundary conditions \eqref{eqn:euler_slag:bc}.
\end{theorem}
\begin{remark}
More precisely, the time of existence $T$ in Theorem \ref{thm:wp} admits a lower bound that depends only on $M_{\ini}$, $\Mb$, $c_*$, $h_*$, $h^*$, $d$, $s_0$, $s$, $\rho_*$, $\rho^*$.
\end{remark}
\begin{remark}
This result is to be compared with \cite[Theorem \nouveau{2}.2]{Duchene2022}, where the authors prove in particular the well-posedness of a system similar to \eqref{eqn:euler_slag}, but with a regularizing term and a slightly different ansatz. Here the use of Alinhac's good unknown, presented in Appendix \ref{appendixB}, allows us to study the system \eqref{eqn:euler_slag} without this regularizing term. The ansatz that we use here allows us to keep the symmetry of the system \eqref{eqn:euler_euleriennes:pert} (see the term $(iii)$ in the proof of Proposition \ref{lemma:energy}).\\ 
\end{remark}
\begin{remark}
\label{rk:short_time}
Our result is to be compared with \cite[Theorem 2]{Desjardins2019}, where the well-posedness is established in Eulerian coordinates without any regularizing term, in the case of no background current $\Vb$. Our analysis yields a more precise time of existence of the solution than the one stated in Theorem \ref{thm:wp}, namely
\begin{equation}
\label{eqn:time}
T = O\left(\left(1 +\frac{\epsilon}{\sqrt{\mu}} + \frac{|\Vb'|_{L^{\infty}}}{\sqrt{\mu}}\right)^{-1} \right).
\end{equation}
In \cite[Theorem 2]{Desjardins2019}, the authors also assume that the Brünt-Väisälä frequency $N^2$ defined through \eqref{eqn:N2} is constant and $\Vb = 0$, as well as a well-preparation of the initial data and the Boussinesq assumption, that is the density $\varrho$ is constant except in the buoyancy term $b$. In this case, they obtain the time of existence $T = O(\frac{\sqrt{\mu}}{\epsilon})$. Such a result is qualified as large time well-posedness, as the time of existence grows as $\epsilon$ goes to $0$.\\
In the case of $N^2$ non-constant and $\Vb \neq 0$, the time of existence in \eqref{eqn:time} does not grow as $\epsilon \to 0$, and such a result is qualified as short-time well-posedness. Let us mention that such a short-time well-posedness is not uncommon among the study of perturbations of non-constant equilibria  of non-linear systems, or more generally, of non-linear systems with non-constant coefficients. For example, the non-linear shallow water equations with non-flat bottom suffer from the same difficulty, although it has been resolved in \cite{BreschMetivier10}. Another example is the water-waves system, for which the large time well-posedness is known for the flat bottom case, but is still an open problem for the non-flat bottom case (see \cite{Lannes2013},\cite{Mesognon17}, \cite[Sections 2.5, 8.7]{Duchene2022a}).\\
In the spirit of weakening the assumptions made in \cite{Desjardins2019}, we allow $N^2$ to depend on $r$ and a non-zero shear flow $\Vb$, and we do not impose the Boussinesq approximation nor the well-preparation of the initial data stated above. Doing so, other contributions to the time of existence in \eqref{eqn:time} appear, and we do not keep track of them here, but only prove that they contribute to the time of existence through the term $1$ in \eqref{eqn:time}. In the statement of Theorem \ref{thm:wp}, we rather state that $T$ is independent of $\epsilon$ and $\mu$, under the assumptions of Theorem \ref{thm:wp}, for the sake of conciseness. 
A more detailed comparison is provided in Remark \ref{rk:trick}.\\
\end{remark}
\begin{remark}
Let us discuss the smallness assumption on the shear velocity now. First notice that this implies that the Miles and Howard criterion \eqref{eqn:mh} is verified for $\mu$ small enough. Then, notice that the total vorticity, when $d=2$, is defined as
$$ \omega_{\tot} := \begin{pmatrix}
 \partial_r V_{\tot}^{\perp} - \sqrt{\mu}\nabla_x^{\perp} w_{\tot} \\ - \nabla_x^{\perp} \cdot V_{\tot} 
\end{pmatrix}. $$
Here, $\nabla_x^{\perp} := \begin{pmatrix} - \partial_{x_2} \\ \partial_{x_1} \end{pmatrix} . $\\
Thus, the assumption of small shear velocity $|\Vb'|_{L^{\infty}} \leq \sqrt{\mu}$ yields that the vorticity at equilibrium (that is, when $\epsilon = 0$) is of order $\sqrt{\mu}$. Together with the assumption $\epsilon \leq \sqrt{\mu}$, this implies that the horizontal component of the total vorticity is of order $\sqrt{\mu}$. This is in accordance with the result in \cite{Castro2014a}, where the authors study the well-posedness of the water waves equations with vorticity under this assumption.
\end{remark}
\begin{remark}
Let us finally comment on the regularity $s \geq s_0+\frac52$ needed here. Note that this is equivalent to $s > \frac{d+1}{2} +2$, where $d+1$ is the dimension of the strip $S_r$. The classical critical regularity  for similar results on the Euler equations is $\frac{d+1}{2} +1$, where $d+1$ denotes the dimension of the domain occupied by the fluid. This discrepancy is due to the non-linearity introduced by the operators $\gradphi$, namely recall the expression of the term $\gradphi[x] P$ in \eqref{eqn:euler_slag}
$$ \gradphi[x] P := \nabla_x P + \epsilon \frac{\nabla_x \eta}{1+\epsilon h} \partial_r P,$$
where $h := - \partial_r \eta$. From a technical viewpoint, this is the sharp regularity needed on $\eta$ so that $(\nabla_x \eta, h) \in W^{1,\infty}$ by Sobolev embeddings. 
\end{remark}

\subsection{Perspectives}
Let us mention a few perspectives. \\

As shown in this study, the energy method can be applied to the reformulation \eqref{eqn:euler_slag} of the Euler equations in isopycnal coordinates, without the regularization used in \cite{Duchene2022}. This offers the possibility of using the reformulation in isopycnal coordinates for the study of stratified flows. For instance, one advantage of the isopycnal coordinates, besides the quasi-2D structure, is the rectification of the isopycnals. In the free surface case, assuming that the surface is initially an isopycnal, this change of coordinates also rectifies the surface (see for instance \cite{Duchene2022}): in this setting the free-surface case is quite close to the rigid lid case, and we expect that our study can be extended to the free-surface case with minor adaptations, see Remarks \ref{rk:slag_existence} and \ref{rk:free-surf}. However, the extension to the case of a non-flat bottom seems to require adaptations. On one hand, the isopycnal coordinates are not adapted to this setting any more, as they do not rectify the bottom in general. On the other hand, and more importantly, the dependency of the time of existence of the solutions in the shallow-water parameter $\mu$ would need to be carefully investigated. \\

Recall now that under the assumptions of no shear flow, uniform stratification and strong Boussinesq assumption, Theorem \ref{thm:wp} yields a time of existence of order $1$ with respect to $\epsilon$. However, with the additional assumption of good preparation on the initial data, \cite[Theorem 2]{Desjardins2019} yields an existence time of order $\frac{1}{\epsilon}$ with respect to $\epsilon$. Therefore, it remains to study the phenomenon arising in the long-time dynamics when no assumption of good preparation is made. In \cite{Klein2023}, the authors perform a similar study, though in a different regime, and show the existence of a boundary layer at the bottom. \\

Eventually, it would be of utmost interest to study a specific stratification that is ubiquitous in oceanography, namely two layers of (almost) constant density, separated by a pycnocline, where the density varies sharply but in a continuous manner. In \cite{Desjardins2019}, the authors suggested to rely on the modal decomposition they introduced to obtain the two-layer shallow water equations as a limit when the thickness of the pycnocline and the shallow water parameter $\mu$ both tend to zero. The isopycnal coordinates seem a useful setting, as the isopycnals are straightened, therefore so is the pycnocline - which can be viewed as a blurred interface between the two layers. Relying on the regularization used in \cite{Duchene2022}, the authors of \cite{Adim2024} proved the convergence of this continuous stratification with thin pycnocline setting towards the two-layer setting, in the hydrostatic case.

\subsection{Structure of the proof and plan of the paper}
In Section \ref{section:pressure}, we show how the pressure is defined from the other unknowns through an elliptic equation. We state in Subsection \ref{subsection:pression:results} two estimates on the pressure term. In Subsection \ref{subsection:pression:proofs}, we prove these estimates.  \\
In Section \ref{section:energy}, we derive the energy estimate on the unknowns $V,w,\eta$, that we want uniform in $\epsilon$ and $\mu$. The approach in Subsection \ref{subsection:energy_low_reg} yields a loss of derivatives, but is used in Subsection \ref{subsection:energy_high_reg} together with Alinhac's good unknown to derive energy estimates without any loss of derivatives. \\
In Section \ref{section:conclusion}, we finish the proof of Theorem \ref{thm:wp}. In particular, we show in Proposition \ref{prop:existence} how the quasi-2D structure of \eqref{eqn:euler_slag} and the elliptic estimates for the pressure in Proposition \ref{lemma:elliptic_high_reg} allow us to use the classical energy method to prove the existence of solutions. We also briefly show how Theorem \ref{thm:wp} yields existence and uniqueness of solutions in Eulerian coordinates.\\
Appendix \ref{appendixA} is concerned with product, commutator and composition estimates, and Appendix \ref{appendixB} gathers definitions and results about Alinhac's good unknown. Appendix \ref{appendixC} states an estimate on the non-linear term and the buoyancy term in \eqref{eqn:euler_slag}. In this appendix we also show how to use the divergence-free condition to get a control on $\partial_r w$ that is uniform in $\mu$.
\subsection{Notations}
\begin{itemize}[label = \textbullet]
\item If $(a,b) \in \R^2$, then $a\vee b$ denotes the maximum between $a$ and $b$, and $a \wedge b$ denotes the minimum.
\item The letter $\mathbb{\Lambda}$ generically represents a differential operator. We denote by $\Lambda$ the Fourier multiplier on $\R^d$ of symbol $(1+|\xi|^2)^{\frac12}$ and $|D|$ the Fourier multiplier on $\R^d$ of symbol $|\xi|$. We mostly use $\Lambda^s$ or $|D|^2 \Lambda^{s-2}$ for $s \in \R$, as their symbols are smooth. The symbol of the later vanishes at $\xi = 0$, which is a property that we need in some cases (more precisely, see \eqref{eqn:Ldeux} in the proof of Proposition \ref{lemma:elliptic:2} as well as \eqref{eqn:alinhac:magie} in the proof of Lemma \ref{apdx:alinhac:def}).
\item We denote by $\varrho(r)$ the density of the layer $r$ for $r \in [0,1]$. The stable stratification assumption yields that $\varrho$ is strictly increasing, as $r=0$ denotes the surface, and we assume that $ 0 < \varrho(0) < \varrho(1)$. The precise dependence on these two constants $\rho_0 := \varrho(0)$ and $ \rho_1 := \varrho(1)$ is not specified in the rest of the study. 
\item We use an isopycnal change of coordinates
\begin{equation}
\label{eqn:def:phi}
\varphi(t,x,r) = (t,x,-r + \epsilon \eta(t,x,r)),
\end{equation}
where the function $(t,x,r) \mapsto \eta(t,x,r)$ is a $C^1$ function, such that $\eta_{|r=0} = \eta_{r=1}=0$ and $ h := -\partial_r \eta$ satisfies $0 < h_* \leq 1+\epsilon h \leq h^*$, and  
$$(\rhob + \epsilon \rho)(t,x,-r + \epsilon \eta(t,x,r)) = \varrho(r).$$ 
The Jacobian of $\varphi$ is thus given by 
\begin{equation}
\label{eqn:def:jacphi} \begin{aligned} \partial \varphi &= \begin{pmatrix} 1 & 0 & \epsilon \partial_t \eta \\
									0 & I_d & \epsilon \nabla_x \eta \\
									0 & 0 & -1 - \epsilon h \end{pmatrix} \\
	 \end{aligned} \end{equation}
and its determinant is $-(1+ \epsilon h)$. \\
The gradient in Eulerian coordinates expressed in isopycnal coordinates is defined as $ \gradphi f := (\nabla (f\circ \varphi^{-1}))\circ \varphi.$ The chain rule yields 
\begin{equation}
\label{eqn:def:gradphi}
\begin{aligned}
\partial_t^{\varphi} f &= \partial_t f + \epsilon \frac{\partial_t \eta}{1+\epsilon h} \partial_r f, & \qquad \partial_r^{\varphi} f &=- \frac{1}{1+\epsilon h} \partial_r f, \\
\gradphi[x] f &= \nabla_{x} f + \epsilon \frac{\nabla_{x} \eta}{1+\epsilon h} \partial_r f. \\
 \end{aligned}
\end{equation}
We also write $\gradphi := (\gradphi[x], \partial_r^{\varphi})^T$.
\item For $\mu \in (0,1]$, we define
\begin{equation}
\label{eqn:def:gradphimu}
\nouveau{\nabla_{\mu} := \begin{pmatrix}
 \sqrt{\mu} \nabla_x \\ \partial_r \end{pmatrix}, }
 \qquad 
 \gradphi[\mu] := \begin{pmatrix}
 \sqrt{\mu} \gradphi[x] \\ \partial_r^{\varphi}  \end{pmatrix}.
 \end{equation}
\item For $p \in [1,\infty]$, we use the standard $L^p-$based Sobolev spaces
\begin{equation}
\label{eqn:sobolev:anisotrope:def}
W^{s,p}(\R^d) := \{ \phi  \in L^p(\R^d), |\phi |_{W^{s,p}} < + \infty \},  \qquad |\phi |_{W^{s,p}} := |\Lambda^s \phi  |_{L^p}.
\end{equation}
We denote by simple bars $|\cdot |$ norms in dimensions $1$ or $d$.
We write $H^s := W^{s,2}.$ We use the functional spaces, for $s \in \R_+$, \nouveau{$k\in \N$ with }$0 \leq k \leq s$,$p \in [1,\infty]$:
$$ W^{s,k,p}(S_r) := \{\phi  \in L^p(S_r), \Vert \phi  \Vert_{W^{s,k,p}} < \infty \}, \qquad \Vert \phi  \Vert_{W^{s,k,p}} :=  \sum\limits_{l = 0}^{k} \Vert \Lambda^{s-l} \partial_{r}^l \phi  \Vert_{L^p},  $$
and write $H^{s,k} := W^{s,k,2}$. In the case $s=k$, we write $H^s(S_r) := H^{s,k}(S_r)$ \nouveau{ and $W^{s,\infty}(S_r) := W^{s,s,\infty}(S_r)$}. We only use $L^2$ and $L^{\infty}$ based Sobolev spaces.
\item We introduce the space of locally square-integrable functions $L^2_{\loc}(S_r)$ as the functions that are square integrable on any compact subset of $S_r$. We then define the homogeneous Sobolev space

$$ \nouveau{\dot{H}^1(S_r) := \{ \phi  \in L^2_{\loc}(S_r), \nabla \phi  \in L^2(S_r) \}.}$$
\item Let $I \subset \R$ an interval, $X$ a Banach space. The space $C^k(I,X)$ denotes the space of $k$ times differentiable functions from $I$ to $X$, with continuous $k^{\text{th}}$-differential.
\item Let $\mathbb{\Lambda}$ be a differential operator, and $f,g$ two functions. We define the commutator $ [\mathbb{\Lambda},f]g := \mathbb{\Lambda}(fg) - f (\mathbb{\Lambda} g),$ as well as the symmetric commutator $[\mathbb{\Lambda};f,g] := \mathbb{\Lambda}(fg) - (\mathbb{\Lambda}f) g - f (\mathbb{\Lambda}g).$
\end{itemize}

\section{Pressure estimates}
\label{section:pressure}
In this section we show how the pressure is defined from the other unknowns, using the divergence-free condition. We also obtain an estimate on the pressure term thanks to the elliptic equation satisfied by the pressure. 
In Subsection \ref{subsection:pression:results}, we state the two main results of this section : Proposition \ref{lemma:elliptic_high_reg} states that the pressure is well defined and as regular as the unknowns $(\eta,V,w)$. However, the estimate \eqref{eqn:elliptic:high_reg} involves $\gradphi[\mu] P$ (defined through \eqref{eqn:def:gradphimu}). Proposition \ref{lemma:elliptic:high_reg:perte} states a similar result involving $\gradphi P$, however with a loss of derivatives, i.e. $\gradphi P$ is less regular than $(\eta,V,w)$. \\
In Subsection \ref{subsection:pression:proofs}, we first state two elliptic estimates, namely in Propositions \ref{lemma:elliptic:1} and \ref{lemma:elliptic:2}. Proposition \ref{lemma:elliptic:1} is classical, and we only prove Proposition \ref{lemma:elliptic:2}. The strategy for proving Proposition \ref{lemma:elliptic:2} relies on an adaptation of standard elliptic estimates, using Alinhac's good unknown. However, it is important to notice that alternative proofs that do not use Alinhac's good unknown do exist, see Remark \ref{rk:pas_euleriennes}. Then we state a bound on the difference of two solutions of elliptic equations of the form of those of Proposition \ref{lemma:elliptic:2} in Proposition \ref{lemma:elliptic:2:tilde}. The proof is an adaptation of the proof of Proposition \ref{lemma:elliptic:2} so that we only highlight the differences.  Eventually we use Propositions \ref{lemma:elliptic:1} and \ref{lemma:elliptic:2} to prove Propositions \ref{lemma:elliptic_high_reg} and \ref{lemma:elliptic:high_reg:perte}. 
\subsection{Statement of the results}
\label{subsection:pression:results}
To get the elliptic problem \eqref{eqn:elliptic} satisfied by the pressure, we mimic the method in Eulerian coordinates : we multiply the equation on the horizontal momentum in \eqref{eqn:euler_slag} by $\mu$, and want to apply $\gradphi[x] \cdot$ on this equation, $\partial_r^{\varphi}$ on the equation on the vertical momentum, then sum the two expressions. However, note that $\partial_t$ and $ \gradphi \cdot$ do not commute. To bypass this difficulty, note that \eqref{eqn:euler_slag} can still be written without the property of quasi-2D advection \eqref{eqn:slag}, and reads
\begin{equation}
\tag{$E'$}
\label{eqn:euler_phi}
\sys{ \partial_t^{\varphi} (\Vb + \epsilon V) + (\Ub + \epsilon \vec{U}) \cdot \gradphi (\Vb + \epsilon V) + \frac{\epsilon }{\varrho} \gradphi[x] P  & = 0, \\
\mu\left(\partial_t^{\varphi} w + (\Ub + \epsilon \vec{U}) \cdot \gradphi w \right) + \frac{1}{\varrho} \partial_r^{\varphi} P  + b  & = 0, \\
		\partial_t \eta + (\Vb + \epsilon V) \cdot \nabla_x \eta  - w &= 0,\\
		\nabla_x^{\varphi} \cdot (\Vb + \epsilon V) + \epsilon \partial_r^{\varphi}w &= 0. \\}
\end{equation}
Here, $\Ub := (\Vb,0)^T$ and recall the expression \eqref{eqn:def:b} of b.
The operators $\partial_t^{\varphi}$ and $\nabla^{\varphi}\cdot$ commute, as can be seen by change of coordinates. Therefore, we apply $\frac{\mu}{\epsilon} \gradphi[x] \cdot$ to the first equation, $\partial_r^{\varphi} $ to the second one. Summing these two expressions yields

\begin{equation}
\label{eqn:elliptic}
 \gradphi[\mu] \cdot \frac{1}{\varrho} \gradphi[\mu] P  = - \frac{\mu}{\epsilon}  \gradphi \cdot \left(\NLun\right) -  \partial_r^{\varphi} b.
\end{equation}
This is completed with Neumann conditions at the bottom and rigid lid, by taking the trace of the equation on the vertical momentum in \eqref{eqn:euler_phi}
\begin{equation}
\label{eqn:elliptic:bc}
\left(\frac{1}{\varrho} \partial_{r}^{\varphi} P\right)_{|r=0; r=1} = - \left. b  \right|_{r=0; r=1}.
\end{equation}
We also wrote 
$$ \gradphi[\mu] := \begin{pmatrix}
						\sqrt{\mu} \gradphi[x] \\ \partial_r^{\varphi} \end{pmatrix},\qquad \nabla_{\mu} := \begin{pmatrix}
						\sqrt{\mu} \nabla_x \\ \partial_r \end{pmatrix}.$$ 
Note that \eqref{eqn:elliptic} is valid at any given time $t$, it is not an evolution equation. Therefore, we omit the dependence in time in this section, to simplify the notations.\\
Using the divergence-free condition in \eqref{eqn:euler_phi}, we can write
\begin{equation}
\label{eqn:elliptic:ordre1}
 \gradphi[\mu] \cdot \frac{1}{\varrho} \gradphi[\mu] P  = - \frac{\mu}{\epsilon}  \NLdeux -  \partial_r^{\varphi} b,
\end{equation}
with the convention of summing over repeated indices. The equation \eqref{eqn:elliptic:ordre1} is completed with the boundary conditions \eqref{eqn:elliptic:bc}.
Let $\Mb,M > 0$, as well as $s_0 > d/2$, $s \geq s_0+\frac52$ with $s \in \N$. 
Throughout this section, we assume that there exists $h_* > 0$ such that\begin{hyp}
\label{hyp:h:elliptic}
1 + \epsilon h := - (1 + \epsilon \partial_r \eta) \geq h_*.
\end{hyp}
We also assume\begin{hyp}
 \label{hyp:M:elliptic}
 \Vert \eta \Vert_{H^{s}} + \Vert V \Vert_{H^{s}} + \sqrt{\mu}\Vert w \Vert_{H^{s}}  \leq M.
 \end{hyp}We also need to assume the particular asymptotic regime
\begin{hyp}
\label{hyp:epsilonmu}
\epsilon \leq \sqrt{\mu},
\end{hyp}
for the following estimates to be independent of $\mu$. We now state the main results of this section.
\begin{proposition}
	\label{lemma:elliptic_high_reg}
	Let $s_0 > d/2$, $s \in \N$ with $s \geq s_0 + \frac52  $. Let $(V,w,\eta) \in H^s(S_r)^{d+2}$ be such that the divergence-free condition in \eqref{eqn:euler_slag} and the boundary conditions \eqref{eqn:euler_slag:bc} are satisfied. We further assume \eqref{hyp:non_cavitation}, \eqref{hyp:Mb}, \eqref{hyp:h:elliptic}, \eqref{hyp:M:elliptic}, and \eqref{hyp:epsilonmu} for some constants $\Mb,M$. Then the pressure $P$ defined through \eqref{eqn:elliptic} and \eqref{eqn:elliptic:bc} is well-defined in $\dot{H}^1/\R$ and satisfies the following bound, with $C > 0$ a constant depending only on $M$, $\Mb$
\begin{equation}
\label{eqn:elliptic:high_reg}
\Vert \gradphi[\mu]P \Vert_{H^{s}} \leq C \left( \Vert \eta \Vert_{H^{s}} + \Vert V \Vert_{H^{s}} +\sqrt{\mu}\Vert w \Vert_{H^{s}} \right) .
\end{equation}
\end{proposition}
\nouveau{
\begin{remark}
On the right-hand side of \eqref{eqn:elliptic:high_reg}, we keep the factor $\Vert \eta \Vert_{H^{s}} + \Vert V \Vert_{H^{s}} +\sqrt{\mu}\Vert w \Vert_{H^{s}}$, instead of writing \eqref{eqn:elliptic:high_reg} as
$$\Vert \gradphi[\mu]P \Vert_{H^{s}} \leq CM.$$
The reason is that this factor is needed in order to write the energy estimate of Proposition \ref{lemma:energy}, which in turn is in a suitable form to use Grönwall's lemma.
\end{remark}
}
We now state a similar result, where the bound is uniform in $\mu$ but there is a loss of derivatives due to the use of the hydrostatic pressure in the proof.  Recall indeed that the estimate in Proposition \ref{lemma:elliptic_high_reg} involves $\gradphi[\mu] P := (\sqrt{\mu} \gradphi[x]P, \partial_r^{\varphi} P)^T$, whereas the following estimate is on $\gradphi P :=  (\gradphi[x]P, \partial_r^{\varphi} P)^T$.
\begin{proposition}
\label{lemma:elliptic:high_reg:perte}
With the same assumptions and notations as in Proposition \ref{lemma:elliptic_high_reg}, we have the bound
\begin{equation}
\label{eqn:elliptic_high_reg:perte}
\Vert \gradphi P \Vert_{H^{s}} \leq  C\left( \Vert \eta \Vert_{H^{s+1}} + \Vert V \Vert_{H^{s+1}} + \sqrt{\mu}\Vert w \Vert_{H^{s+1}}\right).
\end{equation}
\end{proposition}
\begin{remark}
\label{rk:dirichlet}
In other settings, the elliptic equation \eqref{eqn:elliptic} can be completed with different boundary conditions. For instance in the case of a non-flat bottom, the Neumann boundary condition at the bottom involves the trace of normal derivative of the pressure at the bottom rather than just its vertical derivative. In the case of a free-surface, the boundary condition at the top is a (usually homogeneous) Dirichlet boundary condition. In both cases, the proof of Proposition \ref{lemma:elliptic_high_reg} can be adapted, however the resulting estimates might exhibit important changes, namely the dependency of the estimate on the regularity on the free-surface and on the shallow-water parameter $\mu$ are not obvious. Further computations would then be necessary to evaluate these changes, and fall outside of the scope of this study.
\end{remark}
\subsection{Proofs of the results}
\label{subsection:pression:proofs}
\newcommand{\Fzero}{f_0}
\newcommand{\Fun}{\vec{F}_1}
\newcommand{\Fdeux}{f_2}
\newcommand{\Ftrois}{f_{3,i}}
\newcommand{\Gx}{\vec{G_x}}
\newcommand{\Gr}{\vec{G_r}}
\newcommand{\G}{\vec{G}}
In order to prove Propositions \ref{lemma:elliptic_high_reg} and \ref{lemma:elliptic:high_reg:perte}, we state two more general results.\\
We first state Proposition \ref{lemma:elliptic:1}, which is a classical existence and uniqueness result with an estimate on the solution, of the elliptic equation $\eqref{eqn:elliptic:gen:1}$ with boundary conditions \eqref{eqn:elliptic:gen:1:bc}, which is a generalization of the equation \eqref{eqn:elliptic} with boundary conditions \eqref{eqn:elliptic:bc}. However, this result involves a loss of derivatives in $\eta$.\\
We show how to deal with this issue in Proposition \ref{lemma:elliptic:2}, which deals with \eqref{eqn:elliptic:gen:2} with boundary conditions \eqref{eqn:elliptic:gen:2:bc} (generalizing \eqref{eqn:elliptic:ordre1} with boundary conditions \eqref{eqn:elliptic:bc}), and which does not involve any loss of derivatives. \nouveau{Note however that in Proposition \ref{lemma:elliptic:2}, we assume the existence of a solution, and its control in a low regularity norm, namely \eqref{hyp:P:lowreg}.} We also provide a bound on the difference of two solutions of \eqref{eqn:elliptic:gen:2} with boundary conditions \eqref{eqn:elliptic:gen:2:bc} in Proposition \ref{lemma:elliptic:2:tilde}. \\
Finally, we show how Propositions \ref{lemma:elliptic:1} and \ref{lemma:elliptic:2} imply Propositions \ref{lemma:elliptic_high_reg} and \ref{lemma:elliptic:high_reg:perte}.
\subsubsection{A first elliptic equation}
We solve the more general elliptic equation
\begin{equation}
\label{eqn:elliptic:gen:1}
 \gradphi[\mu] \cdot \frac{1}{\varrho} \gradphi[\mu] P  = \gradphi[\mu] \cdot \G.
\end{equation}
This is completed with Neumann conditions
\begin{equation}
\label{eqn:elliptic:gen:1:bc}
\left( \frac{1}{\varrho}  \partial_r^{\varphi} P\right)_{|r=i} =    \vec{e_{d+1}} \cdot {\G}_{|r=i},
\end{equation}
for $i \in \{0,1\}$. Here, $\G$ is a source term, and $\varphi$ is defined from $\eta$ through \eqref{eqn:def:phi}. Note that we assume that the top and bottom boundaries are flat, through the boundary conditions \eqref{eqn:euler_slag:bc} on $\eta$.
Recall also that $\gradphi$, $\gradphi[\mu]$ are defined from $\varphi$ through \eqref{eqn:def:gradphi} and \eqref{eqn:def:gradphimu}. Finally, for $s_0 > d/2$, $0 \leq k \leq s$ with $s \geq s_0 + \frac52$, we write
\begin{hyp}
\label{hyp:M:elliptic:sk}
\Vert \eta \Vert_{H^{s,k}} \leq M,
\end{hyp} 
for $M > 0$ a constant.
\begin{proposition}
\label{lemma:elliptic:1}
Let $s_0 > d/2$,  $2 \leq k \leq s_0+\frac32$ with $k \in \N$. We assume that $\varrho$ and $\eta$ are such that \eqref{hyp:non_cavitation}, \eqref{hyp:Mb}, \eqref{hyp:h:elliptic}, \eqref{hyp:M:elliptic:sk} with $s \geq s_0+\frac52$ hold, and that \nouveau{$\eta$ satisfies the last two boundary conditions in \eqref{eqn:euler_slag:bc}}. \\
Then there exists a unique solution $P \in \dot{H}^1(S_r)/\R$ solution to \eqref{eqn:elliptic:gen:1} with boundary condition \eqref{eqn:elliptic:gen:1:bc} and there exists a constant $C > 0$ depending only on $M$,$\Mb$ such that

$$\Vert \gradphi[\mu] P \Vert_{H^{s_0+\frac32,k}} \leq C (1+\Vert \eta \Vert_{H^{s_0+\frac52,k+1}}) \Vert \G \Vert_{H^{s_0+\frac32,k}}.$$

\end{proposition} 
The proof is a standard result on elliptic equations and we omit it. We refer to \cite[Lemma 4]{Desjardins2019} for a general sketch of the proof. The only difference with this latter result is the operators $\gradphi[\mu]$ that come from the isopycnal change of variables, that induce commutators that are treated with standard commutator estimates (namely, those of Lemma \ref{apdx:commutator}); this latter point is detailed in \cite[Lemma 4.1]{Duchene2022}, that deals with an elliptic equation of the form \eqref{eqn:elliptic:gen:1} (in particular, written in isopycnal coordinates), although with different boundary conditions than \eqref{eqn:elliptic:gen:1:bc}.
\begin{remark}
Note that in Proposition \ref{lemma:elliptic:1}, the factor $1+\Vert \eta \Vert_{H^{s_0+\frac52,k+1}}$ could be included in the constant $C$, as $C$ depends on $\Vert \eta\Vert_{H^s}$ with $s \geq s_0+\frac52$ according to \eqref{hyp:M:elliptic:sk}. We keep this factor in Proposition \ref{lemma:elliptic:1} to emphasize that the gradient of the pressure $\gradphi[\mu] P$ is less regular than $\eta$. This result is not optimal, and it will be improved in Proposition \ref{lemma:elliptic:2} thanks to Alinhac's good unknown (see however Remark \ref{rk:pas_euleriennes} for an alternative proof that does not use Alinhac's good unknown).\\
Another difference with Proposition \ref{lemma:elliptic:2} is that the source term in \eqref{eqn:elliptic:gen:1} solved in Proposition \ref{lemma:elliptic:1} needs to be in divergence form in order to have existence of a solution to \eqref{eqn:elliptic:gen:1} with the Neumann boundary conditions \eqref{eqn:elliptic:gen:1:bc}. In the equation \eqref{eqn:elliptic:gen:2} solved in Proposition \ref{lemma:elliptic:2}, the source term is not in divergence form, and we need to assume the existence of a solution, as well as its control in a low regularity norm, namely \eqref{hyp:P:lowreg}.
\end{remark}
\begin{remark}
\label{rk:elliptic:anisotrope}
In Propositions \ref{lemma:elliptic:1}, \ref{lemma:elliptic:2} and \ref{lemma:elliptic:2:tilde}, we state elliptic estimates in the anisotropic spaces $H^{s,k}$, for $k \leq s$ and $k,s$ large enough, as these are the spaces used in the proof, in which one needs to distinguish vertical and horizontal derivatives. In Propositions \ref{lemma:elliptic_high_reg} and \ref{lemma:elliptic:high_reg:perte}, we state elliptic estimates in the spaces $H^s(S_r)$, that is to say in the special case $s \in \N$ and $k=s$ with the notations above, as this is the only framework in which these propositions will be applied.
\end{remark}
\subsubsection{A second elliptic equation}
\newcommand{\Gun}{G^{(1)}}
\newcommand{\Gdeux}{\vec{G}^{(2)}}
\newcommand{\Guntilde}{\tilde{G}^{(1)}}
\newcommand{\Gdeuxtilde}{\vec{\tilde{G}}^{(2)}}
\newcommand{\Gi}{G_{i}}
We provide estimates on the quantity $P$ solving the following elliptic equation
\begin{equation}
\label{eqn:elliptic:gen:2}
 \gradphi[\mu] \cdot \frac{1}{\varrho} \gradphi[\mu] P  = \sqrt{\mu}\Gun + \gradphi[\mu] \cdot \Gdeux,
\end{equation}
where $\Gun,\Gdeux$ are  source terms. Once again, $\varphi$ is defined through \eqref{eqn:def:phi}. We assume that $\eta$ vanishes at the top and bottom boundaries $r=0$ and $r=1$, that is we assume the last two identities in \eqref{eqn:euler_slag:bc}. The equation \eqref{eqn:elliptic:gen:2} is completed with Neumann conditions
\begin{equation}
\label{eqn:elliptic:gen:2:bc}
\left( \frac{1}{\varrho} \partial^{\varphi}_r P\right)_{|r=i} = \vec{e_{d+1}} \cdot \Gdeux_{|r=i},
\end{equation}
for $i \in \{0,1\}$. Here, $\Gdeux_{|r=i}$ denotes the trace of $\Gdeux$ at $r=i$. We assume that $P$ is in $\dot{H}^1/\R$ and moreover that
\begin{hyp}
\label{hyp:P:lowreg}
\Vert \nouveau{\gradphi[\mu]} P \Vert_{H^{s_0+\frac32,2}} \leq M.
\end{hyp}
\begin{proposition}
\label{lemma:elliptic:2}
Let $s_0 > d/2$,  $ s \geq s_0+\frac52$, $3 \leq k \leq s$ with $k \in \N$. We assume that there exists $\Mb,M$ such that \eqref{hyp:non_cavitation}, \eqref{hyp:Mb}, \eqref{hyp:h:elliptic}, \eqref{hyp:epsilonmu}, \eqref{hyp:M:elliptic:sk}, \eqref{hyp:P:lowreg} hold. \nouveau{We assume further that $\eta$ satisfies the last two boundary conditions in \eqref{eqn:euler_slag:bc}}. \\
Assume that there exists a solution $P \in \dot{H}^1(S_r)/\R$ to \eqref{eqn:elliptic:gen:1} with boundary conditions \eqref{eqn:elliptic:gen:2:bc} satisfying \eqref{hyp:P:lowreg}. Then there exists a constant $C > 0$ depending only on $M$ and $\Mb$ such that
\begin{equation}
\label{eqn:elliptic:2}
\Vert \gradphi[\mu] P \Vert_{H^{s,k}} \leq C\left( \Vert \gradphi[\mu] P \Vert_{H^{s_0+\frac32,2}} +  \Vert \Gun \Vert_{H^{s-1,k-1}}  + \Vert \Gdeux \Vert_{H^{s,k}}  \right).
\end{equation}
\end{proposition}
\nouveau{
\begin{remark}
In Proposition \ref{lemma:elliptic:2}, the existence of $P$ as well as its control in a low regularity norm are assumed. In the proofs of Proposition \ref{lemma:elliptic_high_reg} and \ref{lemma:elliptic:high_reg:perte} as well as in Step 2 of the proof of Proposition \ref{prop:existence}, we first check that these two assumptions are satisfied, by applying Proposition \ref{lemma:elliptic:1}, and only then apply Proposition \ref{lemma:elliptic:2}. This is possible because the pressure $P$ involved in \eqref{eqn:euler_slag} satisfies both \eqref{eqn:elliptic} and \eqref{eqn:elliptic:ordre1}.
\end{remark}
}
We now write an estimate on the difference of two solutions of elliptic equations of the form \eqref{eqn:elliptic:gen:2} with boundary conditions \eqref{eqn:elliptic:gen:2:bc}.\\
Let $\tilde{\eta} \in H^{s_0+\frac32,2}$ for $s_0 > d/2$ satisfying the last two identities in \eqref{eqn:euler_slag:bc} as well as \eqref{hyp:h:elliptic}. Then, $\tilde{\varphi}$ is defined from $\tilde{\eta}$ through \eqref{eqn:def:phi}, and $\nabla^{\tilde{\varphi}}$ is defined from $\tilde{\varphi}$ through \eqref{eqn:def:gradphi}.
Let $\tilde{P}$ solving the following elliptic equation
\begin{equation}
\label{eqn:elliptic:gen:2:tilde}
 \nabla^{\tilde{\varphi}}_{\mu} \cdot \frac{1}{\varrho} \nabla^{\tilde{\varphi}}_{\mu} \tilde{P}  = \sqrt{\mu}\Guntilde + \nouveau{\nabla^{\tilde{\varphi}}_{\mu}} \cdot \Gdeuxtilde,
\end{equation}
where $\Guntilde,\Gdeuxtilde$ are  source terms. The equation \eqref{eqn:elliptic:gen:2:tilde} is completed with Neumann conditions
\begin{equation}
\label{eqn:elliptic:gen:2:tilde:bc}
\left( \frac{1}{\varrho} \partial^{\tilde{\varphi}}_r \tilde{P}\right)_{|r=i} = \vec{e_{d+1}} \cdot \Gdeuxtilde_{|r=i},
\end{equation}
for $i \in \{0,1\}$. Here, $\Gdeuxtilde_{|r=i}$ denotes the trace of $\Gdeuxtilde$ at $r=i$.
For the following proposition, we use the following assumption: there exists a constant $ M \geq 0$ such that
\nouveau{
\begin{hyp}
\label{hyp:borne:G}
\Vert \Gun \Vert_{H^{s-1,k-1}} + \Vert \Guntilde \Vert_{H^{s-1,k-1}} + \Vert \Gdeux \Vert_{H^{s,k}} + \Vert \Gdeuxtilde \Vert_{H^{s,k}} \leq M,
\end{hyp}
}for $s$ and $k$ as in the statement of Proposition \ref{lemma:elliptic:2:tilde}.
\begin{proposition}
\label{lemma:elliptic:2:tilde}
Let $s_0 > d/2$,  $ s \geq s_0+\frac52$, $3 \leq k \leq s$ with $k \in \N$. We assume that there exists $M,\Mb$ such that \eqref{hyp:non_cavitation}, \eqref{hyp:Mb}, \eqref{hyp:h:elliptic}, \eqref{hyp:epsilonmu}, \eqref{hyp:M:elliptic:sk} hold for both $\eta$ and $\tilde{\eta}$, and \eqref{hyp:P:lowreg} holds for both $P$ and $\tilde{P}$. \nouveau{We assume that $M,\Gun, \Guntilde, \Gdeux, \Gdeuxtilde$ are such that \eqref{hyp:borne:G} holds.} \nouveau{We assume that $\eta$ and $\tilde{\eta}$ satisfy the last two boundary conditions in \eqref{eqn:euler_slag:bc}}.\\
Assume that there exist $P,\tilde{P} \in \dot{H}^1(S_r)/\R$ solutions to \eqref{eqn:elliptic:gen:2} with boundary conditions \eqref{eqn:elliptic:gen:2:bc} and \eqref{eqn:elliptic:gen:2:tilde}  with boundary conditions \eqref{eqn:elliptic:gen:2:tilde:bc} respectively, both satisfying \eqref{hyp:P:lowreg}. Then there exists a constant $C > 0$ depending only on $M$ and $\Mb$ such that
\begin{equation}
\label{eqn:elliptic:2:tilde}
\Vert \gradphi[\mu] P - \nabla^{\tilde{\varphi}}_{\mu} \tilde{P} \Vert_{H^{s,k}} \leq C\left( \Vert \gradphi[\mu] P- \nabla^{\tilde{\varphi}}_{\mu} \tilde{P}  \Vert_{H^{s_0+\frac32,2}} +  \Vert \Gun - \Guntilde \Vert_{H^{s-1,k-1}}  + \Vert \Gdeux - \Gdeuxtilde \Vert_{H^{s,k}}  + \Vert \eta - \tilde{\eta} \Vert_{H^{s,k}} \right).
\end{equation}
\end{proposition}
\begin{remark}
\label{rk:pas_loc_lip}
The estimate \eqref{eqn:elliptic:2:tilde} (together with a control the term $\Vert \gradphi[\mu] P - \nabla^{\tilde{\varphi}}_{\mu} \tilde{P} \Vert_{H^{s_0+\frac32,2}}$ on the right-hand side of \eqref{eqn:elliptic:2:tilde}, that can be treated similarly as in Proposition \ref{lemma:elliptic:1}) implies the local Lipschitz continuity in $H^s$ of the map $(V,w,\eta) \mapsto \gradphi[\mu] P,$
with $\gradphi P$ defined through Proposition \ref{lemma:elliptic_high_reg}. The local Lipschitz continuity is understood as the one of a map from $(H^{s})^{d+2}$ to $(H^{s})^{d+1}$, with $s$ as in Proposition \ref{lemma:elliptic_high_reg}, and where $(V,w,\eta) \in (H^{s})^{d+2}$ satisfy the assumptions of Proposition~\ref{lemma:elliptic_high_reg}. However, we only use this property for a slightly different term in the proof of Proposition \ref{prop:existence}. For more details, see Step 2 of the proof of Proposition \ref{prop:existence}.
\end{remark}
The rest of this subsection is devoted to the proof of Propositions \ref{lemma:elliptic:2} and \ref{lemma:elliptic:2:tilde}. \\

The main difference between the proofs of Proposition \ref{lemma:elliptic:2} and  the elliptic estimates performed in \nouveau{\cite[Lemma 4.1]{Duchene2022}  that we refer to for a proof of Proposition \ref{lemma:elliptic:1} }is that we use the relation \eqref{apdx:alinhac:comm1} and Alinhac's good unknown to commute the differential operators $\mathbb{\Lambda}$ (defined in the proof) and $\gradphi$ without loss of derivatives.\\
The proof of Proposition \ref{lemma:elliptic:2:tilde} is very similar to the proof of Proposition \ref{lemma:elliptic:2} so that we only point out the adaptations that need to be performed. 
\begin{remark}
\label{rk:pas_euleriennes}
An alternative proof for Proposition \ref{lemma:elliptic:2} is the following, and it should be noted that such a proof does not use Alinhac's good unknown. One can change coordinates back to Eulerian coordinates in the elliptic equation \eqref{eqn:elliptic:gen:2} and its boundary condition \eqref{eqn:elliptic:gen:2:bc} (that is, applying the diffeomorphism $\varphi^{-1}$, inverse of $\varphi$ for the composition operation). Denoting $P_{\mathrm{eul}} := P \circ \varphi^{-1}$ (resp. $\Gun_{\mathrm{eul}},\Gdeux_{\mathrm{eul}}$), \eqref{eqn:elliptic:gen:2} becomes
\begin{equation}
\label{eqn:elliptic:gen:2:eul}
\nabla_{\mu}\cdot \frac{1}{\varrho \circ \varphi^{-1}} \nabla_{\mu} P_{\mathrm{eul}}= \sqrt{\mu} \Gun_{\mathrm{eul}} + \nabla_{\mu} \cdot \Gdeux_{\mathrm{eul}} \qquad \nouveau{\text{ in } S_z},
\end{equation}
where $\nabla_{\mu} := (\sqrt{\mu} \nabla_x^T,\partial_z)^T$. The boundary conditions \eqref{eqn:elliptic:gen:2:bc} become
\begin{equation}
\label{eqn:elliptic:gen:2:bc:eul}
\left( \frac{1}{\varrho\circ\varphi^{\nouveau{-}1}} \partial_z P_{\mathrm{eul}}\right)_{|\nouveau{z}=i} = \vec{e_{d+1}} \cdot \Gdeux_{\mathrm{eul}|\nouveau{z}=i}.
\end{equation}
Then, elliptic estimates can be derived from \eqref{eqn:elliptic:gen:2:eul} and \eqref{eqn:elliptic:gen:2:bc:eul} following the same scheme of proof as the proof of Proposition \ref{lemma:elliptic:2}, although without using Alinhac's good unknown. In this regard, these estimates in Eulerian coordinates are fairly standard, although special attention to the dependency in $\mu$ of the various estimates is still required. Then one can prove that one can bound $\Vert \gradphi[\mu] P \Vert_{H^{s}(S_r)}$ by $\Vert \nabla_{\mu} P_{\mathrm{eul}} \Vert_{H^{s}(S_z)}$ up to a constant that depends on $M$ defined in \eqref{hyp:M:elliptic:sk}. \\
The drawback of this method is that it is unclear how to adapt it to prove Proposition \ref{lemma:elliptic:2:tilde}. Hence we rather prove Proposition \ref{lemma:elliptic:2} using Alinhac's good unknowns rather than Eulerian coordinates. Let us point out the differences with more standard elliptic estimates. Despite the use of Alinhac's good unknown already mentioned, the dependency of the estimates with respect to $\mu$ is not obvious, and we keep track of this parameter. Also, as $P \in \dot{H}^1/\R$, we do not control $\Vert P \Vert_{L^2}$; these last two \nouveau{points} interact in a rather poor manner, see \eqref{eqn:Ldeux} where we need $\epsilon \leq \sqrt{\mu}$ to avoid a singularity in $\mu$. 
\end{remark}
\begin{proof}[Proof of Proposition \ref{lemma:elliptic:2}.]
\textbf{\underline{Step 1 :} The elliptic equation satisfied by Alinhac's good unknown} \\

\newcommand{\diff}{\mathbb{\Lambda}}
Let $ \diff$ be an operator of the form $|D|^2 \Lambda^{s-2}$ or $\Lambda^{s-k}\partial_r^k$ \nouveau{with $1 \leq k \leq s$}. We apply $\diff$ to the elliptic problem \eqref{eqn:elliptic:gen:2} satisfied by the pressure, assuming that $P$ is smooth enough so that the computations needed to derive the estimate in \ref{lemma:elliptic_high_reg} make sense. To close the argument, one should use a smoothing operator $ \Lambda_{\iota}$ instead of $\Lambda$ in Steps 2 and 3 below, and check that the estimates are uniform with respect to $\iota > 0$. Then, taking the limit $\iota \to 0$ yields the result. This is detailed in \cite[Chapter 2]{Lannes2013} so that we directly use $\Lambda$ and assume the needed regularity. \\
We apply $\diff$ to $\gradphi[\mu]\cdot \frac{1}{\varrho} \gradphi[\mu]P$. This yields, using \eqref{apdx:alinhac:comm2}:\\

$$\begin{aligned} 
\diff  \gradphi[\mu]\cdot \frac{1}{\varrho} \gradphi[\mu]P &= \gradphi[\mu]\cdot \left( \frac{1}{\varrho} \gradphi[\mu]P \right)^{\diff} + \epsilon \left(\diff  \eta \partial_r^{\varphi} \right)\left(\gradphi[\mu]\cdot \frac{1}{\varrho} \gradphi[\mu]P\right) + \epsilon \Ral{2}{\nablamu \frac{1}{\varrho} \gradphi[\mu]P} \\
&= \gradphi[\mu]\cdot \left( \frac{1}{\varrho} \diff  \left(\gradphi[\mu]P\right) + \frac{\epsilon \diff \eta}{1+\epsilon h} \partial_{r}\left(\frac{1}{\varrho} \gradphi[\mu]P\right) \right) + \epsilon \left(\diff  \eta \partial_r^{\varphi} \right)\left(\gradphi[\mu]\cdot \frac{1}{\varrho} \gradphi[\mu]P\right) + \epsilon \Ral{2}{\nablamu \frac{1}{\varrho} \gradphi[\mu]P}  + \gradphi[\mu] \cdot\left( [\diff,\frac{1}{\varrho} ]\gradphi[\mu]P \right), \end{aligned}$$
where, for a quantity $f$, $f^{\diff}$ denotes Alinhac's good unknown relative to this quantity and the operator $\diff$, see \eqref{apdx:alinhac:definition} for its definition. Using \eqref{apdx:alinhac:comm1}, the right-hand side of the previous expression becomes
$$\begin{aligned}
& \gradphi[\mu]\cdot \left( \frac{1}{\varrho} \gradphi[\mu]P^{\diff} + \epsilon \frac{1}{\varrho} (\diff \eta \partial_r^{\varphi})\gradphi[\mu]P + \epsilon \nouveau{\frac{1}{\varrho}}\Ral{1}{\nablamu P}+ \frac{\epsilon \diff \eta}{1+\epsilon h} \partial_{r}\left(\frac{1}{\varrho} \gradphi[\mu]P\right) \right) \\
&+ \epsilon \left(\diff  \eta \partial_r^{\varphi} \right)\left( \gradphi[\mu]\cdot \frac{1}{\varrho} \gradphi[\mu]P\right) + \epsilon \Ral{2}{\nablamu \frac{1}{\varrho} \gradphi[\mu]P} + \gradphi[\mu] \cdot \left([\diff,\frac{1}{\varrho}] \gradphi[\mu]P\right).
\end{aligned}$$
According to the elliptic equation \eqref{eqn:elliptic:gen:2} and multiplying by $\nouveau{-}(1+\epsilon h)$, we thus get the equation:
\begin{equation}
\label{eqn:elliptic:alinhac}
\begin{aligned} -(1+\epsilon h) \gradphi[\mu] &\cdot \frac{1}{\varrho} \gradphi[\mu]P^{\diff}  = (1+\epsilon h)\gradphi[\mu] \cdot \left(\frac{1}{\varrho} (\epsilon \diff  \eta \partial_r^{\varphi})\gradphi[\mu] P\right) \\
&+ \epsilon (1+\epsilon h)\gradphi[\mu] \cdot \nouveau{\frac{1}{\varrho}}\Ral{1}{\nablamu P} + (1+\epsilon h)\gradphi[\mu] \cdot \left(\frac{\epsilon \diff \eta}{1+\epsilon h} \partial_{r}\left(\frac{1}{\varrho} \gradphi[\mu]P\right)\right)\\
& + \epsilon (1+\epsilon h)\left(\diff  \eta \partial_r^{\varphi}\right)\left(\gradphi[\mu] \cdot \frac{1}{\varrho} \gradphi[\mu]P\right) + \epsilon (1+\epsilon h)\Ral{2}{\nablamu\frac{1}{\varrho} \gradphi[\mu]P} \\
&+ (1+\epsilon h)\gradphi[\mu] \cdot \left([\diff,\frac{1}{\varrho}] \gradphi[\mu]P\right) \nouveau{-} \sqrt{\mu}(1+\epsilon h)\diff \Gun\\
&\nouveau{-} (1+\epsilon h)\gradphi[\mu] \cdot {\Gdeux}^{\diff} \nouveau{-} \epsilon (1+\epsilon h) \diff \eta \partial_r^{\varphi} \gradphi[\mu] \cdot \Gdeux \nouveau{-} \epsilon (1+\epsilon h) \Ral{2}{\nablamu \Gdeux}. \end{aligned}
\end{equation}
Note that we used \eqref{apdx:alinhac:comm2} on the source term $\gradphi[\mu] \cdot \Gdeux$. In the case $\diff = |D|^2 \Lambda^{s-2}$ (that is, we only differentiate \eqref{eqn:elliptic:gen:2} in the tangential direction with respect to the boundary of the strip $S_r$), we can apply $\diff$ to the boundary conditions \eqref{eqn:elliptic:gen:2:bc} thanks to \eqref{apdx:alinhac:comm1} and the equation \eqref{eqn:elliptic:alinhac} is completed with the boundary conditions, for $i$ in $\{0,1\}$
\begin{equation}
\label{eqn:elliptic:alinhac:bc}
\begin{aligned}
&\left(\frac{1}{\varrho} \partial^{\varphi}_r P^{\diff}  + \epsilon \frac{1}{\varrho} \diff \eta \partial_r^{\varphi} \partial^{\varphi}_r P + \epsilon \frac{1}{\varrho}  \vec{e_{d+1}} \cdot \Ral{1}{\nabla_{\mu} P} \right)_{|r = i} \\
&= \vec{e_{d+1}} \cdot {\Gdeux}^{\diff}_{|r=i}; 
\end{aligned}
\end{equation}
note that the right-hand side of \eqref{eqn:elliptic:alinhac:bc} is exactly $\vec{e_{d+1}} \cdot \diff \Gdeux$, as by the definition \eqref{apdx:alinhac:definition} of the good unknown and by the boundary condition \eqref{eqn:euler_slag:bc} on $\eta$, we can write ${\Gdeux}^{\diff}_{|r=i} = \diff \Gdeux_{|r=i}$.\\
\renewcommand{\diff}{\mathbb{\Lambda}}
\textbf{\underline{Step 2 :} Estimate on $\Vert \nabla^{\varphi}P^{\diff} \Vert_{L^2}$ for $\diff = |D|^2 \Lambda^{s-2}$} \\
We use \eqref{eqn:elliptic:alinhac} with   $ \diff := |D|^2\Lambda^{s-2}$, and write it under the form
\begin{equation}
\label{eqn:elliptic:alinhac:bis}
-(1+\epsilon h) \gradphi[\mu] \cdot \frac{1}{\varrho} \gradphi[\mu]P^{\diff}  =  \epsilon \Fzero + (1+\epsilon h)\gradphi[\mu] \cdot \Fun  + \sqrt{\mu}|D| \Fdeux,
\end{equation} 
where we have set
\begin{equation}
\label{eqn:elliptic:2:restes}
\sys{		\epsilon \Fzero  &:= \epsilon (1+\epsilon h)\left(\diff  \eta \partial_r^{\varphi}\right)\left(\gradphi[\mu] \cdot \frac{1}{\varrho} \gradphi[\mu]P\right) + \epsilon (1+\epsilon h)\Ral{2}{\nouveau{\nablamu \frac{1}{\varrho} \gradphi[\mu]}P} + \sqrt{\mu} \epsilon[|D|^2,h] \Lambda^{s-2}\Gun \\
&- \epsilon  (1+\epsilon h)\diff \eta \partial_r^{\varphi} \gradphi[\mu] \cdot \Gdeux - \epsilon  \nouveau{(1+\epsilon h)}\Ral{2}{\nablamu \Gdeux},\\
		\Fun  &= \epsilon \frac{1}{\varrho} (\diff  \eta \partial_r^{\varphi})\gradphi[\mu] P + \epsilon \nouveau{\frac{1}{\varrho}}\Ral{1}{\nablamu P} +  \frac{\epsilon \diff \eta}{1+\epsilon h} \partial_{r}\left(\frac{1}{\varrho} \gradphi[\mu]P\right) - {\Gdeux}^{\diff},\\
		\Fdeux  &= -|D|((1+\epsilon h) \Lambda^{s-2} \Gun);}
\end{equation}
note that in the present case $\diff = |D|^2 \Lambda^{s-2}$, we can write $[\diff,\frac{1}{\varrho}]=0$ as $\varrho$ only depends on $r$, and also that for the seventh term on the right-hand side of \eqref{eqn:elliptic:alinhac} we can write
$$ \sqrt{\mu}(1+\epsilon h) \diff\Gun = \sqrt{\mu} |D|^2 ((1+\epsilon h) \Lambda^{s-2}\Gun) - \sqrt{\mu} \epsilon [|D|^2,h] \Lambda^{\nouveau{s-2}}\Gun.$$
We use an induction on $s$. The base case $s \nouveau{\leq} s_0 + \frac32$ in \eqref{eqn:elliptic:2} is automatically satisfied since $\Vert \gradphi[\mu] P \Vert_{H^{s_0+\frac32,2}}$ appears on the right-hand side. We thus take $s \geq s_0+\frac52$. 
We now test \eqref{eqn:elliptic:alinhac:bis} against $P^{\diff}$. Integrating by parts \nouveau{(using Lemma \ref{apdx:IPP})} and using the boundary conditions \eqref{eqn:elliptic:alinhac:bc} yields
\begin{equation}
\label{eqn:elliptic:temp:1}
\begin{aligned}
\left\Vert \nouveau{\sqrt{\frac{1+\epsilon h}{\varrho}}} \gradphi[\mu] P^{\diff} \right\Vert_{L^2}^2 &\leq \epsilon \Vert \Fzero \Vert_{L^2} \Vert P^{\diff} \Vert_{L^2} + \Vert \Fun \Vert_{L^2} \Vert \nouveau{(1+\epsilon h)}\gradphi[\mu] P^{\diff} \Vert_{L^2} + \Vert \Fdeux \Vert_{L^2} \Vert \sqrt{\mu} |D| P^{\diff} \Vert_{L^2}.
\end{aligned}
\end{equation}
We use Young's inequality for the first term of \eqref{eqn:elliptic:temp:1}, and introduce a parameter $0<\delta \leq 1$ to be fixed later (see \eqref{eqn:elliptic:temp:6}), independent of $\epsilon$ and $\mu$. This yields
\begin{equation}
\label{eqn:elliptic:temp:2}
\left\Vert \nouveau{\sqrt{\frac{1+\epsilon h}{\varrho}} }\gradphi[\mu] P^{\diff} \right\Vert_{L^2}\leq \delta \Vert  \Fzero \Vert_{L^2} + \epsilon\frac{1}{\delta} \Vert P^{\diff} \Vert_{L^2} + \nouveau{C}\Vert (\Fun,\Fdeux) \Vert_{L^2}.
\end{equation}
We now need to estimate each of those terms. We first need a control on $\Vert P^{\diff}\Vert_{L^2}$. This is provided by induction, namely we have:
$$P^{\diff} := \diff P + \epsilon \frac{\diff \eta}{1+\epsilon h} \partial_r P  = \diff P - \epsilon \diff \eta\partial_r^{\varphi} P.$$
Thus we have a control on $\Vert P^{\diff} \Vert_{L^2}$ by induction on $s$ (recall that $\diff := |D|^2\Lambda^{s-2}$ and the product estimate \eqref{apdx:pdt:tame}):
\begin{equation}
\label{eqn:Ldeux}
 \Vert P^{\diff} \Vert_{L^2} \leq \frac{C}{\sqrt{\mu}}\Vert \nabla_{\mu} P \Vert_{H^{s-1,0}} + \epsilon C \Vert \eta \Vert_{H^{s,0}} \Vert \gradphi[\mu] P \Vert_{H^{s_0+\frac12,1}},
\end{equation}
for $C > 0$ a constant that does not depend on $\epsilon$ nor $\mu$. Note that there is a singularity in $\mu$, however it is compensated by the factor $\epsilon$ in \eqref{eqn:elliptic:temp:2}, thanks to \eqref{hyp:epsilonmu}. \\
For the first term of $\Fzero $, we use the product estimate \eqref{apdx:pdt:tame} to get
$$ \Vert \nouveau{(1+\epsilon h)}\left(\epsilon \diff  \eta \partial_r^{\varphi}\right)\left(\gradphi[\mu] \cdot \frac{1}{\varrho} \gradphi[\mu]P\right) \Vert_{L^2} \leq C \epsilon \Vert \eta \Vert_{H^{s,0}} \Vert \gradphi[\mu] \cdot \frac{1}{\varrho} \gradphi[\mu]P \Vert_{H^{s_0+\frac32,2}}.$$
We use the equation \eqref{eqn:elliptic:2} to replace $\gradphi[\mu] \cdot \frac{1}{\varrho} \gradphi[\mu]P$ by the right-hand side. The product estimate \eqref{apdx:pdt:tame} then yields
$$ \Vert \nouveau{(1+\epsilon h)} \left(\epsilon \diff  \eta \partial_r^{\varphi}\right)\left(\gradphi[\mu] \cdot \frac{1}{\varrho} \gradphi[\mu]P\right) \Vert_{L^2} \leq C\epsilon \Vert \eta \Vert_{H^{s,0}} \Vert \sqrt{\mu}\Gun + \gradphi[\mu] \cdot \Gdeux\Vert_{H^{s_0+\frac32,2}}.$$
For the second term of $\Fzero $, we use the estimate \nouveau{\eqref{apdx:alinhac:R2}} to get
$$\begin{aligned} \Vert \nouveau{(1+\epsilon h)}\Ral{2}{\nablamu  \frac{1}{\varrho} \gradphi[\mu]P} \Vert_{L^2} &\leq C \Vert \nabla \eta \Vert_{H^{s-1,2}} \Vert \gradphi[\mu] \frac{1}{\varrho} \gradphi[\mu]P \Vert_{H^{s-1,2}}.\\
\end{aligned} $$
For the third term in $\Fzero$, we use the commutator estimate \eqref{apdx:commutator:horizontal} to get
$$\Vert[|D|^2,\epsilon h] \Lambda^{s-2}\Gun \Vert_{L^2} \leq C \epsilon \Vert \eta \Vert_{H^{s_0+\frac52,3}} \Vert \Gun \Vert_{H^{s-1,2}}. $$
The other terms of $\Fzero$ are treated the same way and we omit them. We can eventually write
$$\Vert \Fzero  \Vert_{L^2} \leq C \left(\sqrt{\mu} \Vert \Gun \Vert_{H^{s-1,2}}+ \Vert \Gdeux \Vert_{H^{s,3}} + \Vert \gradphi[\mu] P \Vert_{H^{s,2}}\right). $$
For the first term of $\Fun $, we use \nouveau{\eqref{hyp:non_cavitation}} as well as the product estimate \eqref{apdx:pdt:tame} to get
	$$ \Vert \frac{1}{\varrho} (\epsilon \diff  \eta \partial_r^{\varphi})\gradphi[\mu] P\Vert_{L^2} \leq \epsilon C \Vert \eta \Vert_{H^{s,0}} \Vert \gradphi[\mu] P \Vert_{H^{s_0+\frac32,2}}. $$
	For the second term of $\Fun $, we use \nouveau{\eqref{hyp:non_cavitation}} and the estimate \eqref{apdx:alinhac:R1:un}
	$$\begin{aligned} \Vert \nouveau{\frac{1}{\varrho}}\Ral{1}{\gradphi[\mu]P} \Vert_{L^2} &\leq C  \Vert \eta \Vert_{H^{s_0+\frac52,2}} \Vert \gradphi[\mu] P\Vert_{H^{s-1,0}} + C  \Vert \eta \Vert_{H^{s,0}} \Vert \gradphi[\mu] P\Vert_{H^{s_0+\frac32,1}} \leq C \Vert \gradphi[\mu] P\Vert_{H^{s-1,0}} + C \Vert \gradphi[\mu] P\Vert_{H^{s_0+\frac32,1}},\end{aligned}$$
as $C$ can depend on $M$ used in \eqref{hyp:M:elliptic:sk}. For the third term of $\Fun $, we use \nouveau{\eqref{hyp:non_cavitation}, \eqref{hyp:Mb} together with the product estimate \eqref{apdx:pdt:infty} for the factor $\frac{1}{\varrho}$} and the product estimate \eqref{apdx:pdt:tame}:
	$$  \left\Vert \frac{\diff \eta}{1+\epsilon h} \partial_{r}\left(\frac{1}{\varrho} \gradphi[\mu]P\right) \right\Vert_{L^2} \leq  C \Vert \eta \Vert_{H^{s,0}} \Vert \gradphi[\mu] P \Vert_{H^{s_0+\frac32,2}}. $$
	The bound on the last term of $\Fun$ comes directly from the definition \eqref{apdx:alinhac:definition} and the product estimate \cite[Lemma A.3]{Duchene2022}, and we can thus write
	$$\Vert \Fun  \Vert_{L^2} \leq \epsilon C \left(\Vert \nabla^{\varphi}_{\mu}P\Vert_{H^{s-1,0}} + \Vert \gradphi[\mu] P\Vert_{H^{s_0+\frac32,2}} + \Vert \Gdeux \Vert_{H^{s,2}}\right). $$
The term in $\Fdeux$ is directly controlled by $\Vert \Gun\Vert_{H^{s-1,1}}$ by the product estimate \eqref{apdx:pdt:tame}.
From the control on $\Fzero$ $\Fun$, $\Fdeux$ above and \eqref{eqn:Ldeux}, we get from \eqref{eqn:elliptic:temp:2}:
\begin{equation}
\label{eqn:elliptic:temp:3}
\begin{aligned}
\Vert \gradphi[\mu] P^{\diff} \Vert_{L^2}&\leq C \delta \left( \sqrt{\mu} \Vert \Gun \Vert_{H^{s-1,2}}+ \Vert \Gdeux \Vert_{H^{s,3}} + \epsilon \Vert \gradphi[\mu] P \Vert_{H^{s,2}}\right) + C\frac{\epsilon}{\delta} \left(\frac{1}{\sqrt{\mu}}\Vert \nabla_{\mu} P \Vert_{H^{s-1,0}} + \epsilon \Vert \eta \Vert_{H^{s,0}} \Vert \gradphi[\mu] P \Vert_{H^{s_0+\frac12,1}}\right) \\
&+ C \epsilon  \left(\Vert \Gdeux \Vert_{H^{s,2}} + \Vert \gradphi[\mu] P \Vert_{H^{s-1,0}} + \Vert \gradphi[\mu] P \Vert_{H^{s_0+\frac32,2}} \right) +\nouveau{C}\Vert \Gun\Vert_{H^{s-1,1}},
\end{aligned}
\end{equation}
\nouveau{where we used \eqref{hyp:non_cavitation} and \eqref{hyp:h:elliptic} to simplify the left-hand side}. We now use \eqref{apdx:alinhac:grad:horizontal} and \eqref{hyp:epsilonmu} to simplify the above expression to get
\begin{equation}
\label{eqn:elliptic:temp:4}
\begin{aligned}
\Vert \gradphi[\mu] P \Vert_{H^{s,0}}&\leq C \left( \Vert \Gun \Vert_{H^{s-1,2}} + \Vert \Gdeux \Vert_{H^{s,\nouveau{3}}}+ \frac{1}{\delta}\Vert \gradphi[\mu] P \Vert_{H^{s_0+\frac32,2}}\right) + C\epsilon \delta \Vert \gradphi[\mu] P \Vert_{H^{s,2}} + C\frac{1}{\delta} \Vert \gradphi[\mu] P \Vert_{H^{s-1,1}}.
\end{aligned}
\end{equation}
\textbf{\underline{Step 3 :} Control of higher vertical derivatives} \\
\renewcommand{\diff}{\mathbb{\Lambda}}
We now want an estimate on $\Vert \nablamu P \Vert_{H^{s,k}}$ for $k \geq 1$. To this end, we start by noticing 
\begin{equation}
\label{eqn:reduction_vertical}
 \Vert \gradphi[\mu]P \Vert_{H^{s,k}}  \leq c \Vert \gradphi[\mu]P \Vert_{H^{s,k-1}} + c\Vert \Lambda^{s-k}\partial_r^{k-1} \partial_{r} \partial_r^{\varphi} P \Vert_{L^2},
 \end{equation}
where $c$ only depends on $s$ and $k$. We aim to use the equation \eqref{eqn:elliptic:gen:2} to control the vertical derivatives. We use an induction on $s$ and $k$. More precisely, the base case $k=0$ and any $ s \geq 0$ is given by the previous step, although the term $\Vert \gradphi[\mu] P \Vert_{H^{s,2}}$ in \eqref{eqn:elliptic:temp:4} needs to be taken care of; this is done at the end of this step. The base case $s\leq s_0+\frac32$, $k \leq s$ is given by Assumption \eqref{hyp:P:lowreg}. Therefore, we assume from now on that $\Vert \gradphi[\mu] P \Vert_{H^{s',k'}}$ is controlled for both $s'<s$, $k' \leq s'$ and $s'=s, k'<k$.\\
The first term in \eqref{eqn:reduction_vertical} is controlled by induction. For the second term, we use the equation \eqref{eqn:elliptic:gen:2}: Let    $\diff:=\Lambda^{s-k}\partial_r^{k-1}$ (when $k\geq 2$) or $\diff := |D|^2 \Lambda^{s-3}$ (when $k=1$), we apply $\diff$ to \eqref{eqn:elliptic:gen:2}. Recall \eqref{eqn:elliptic:alinhac} derived in Step 1.
We write $R$ for the right-hand side, multiply by $1+\epsilon h$ and expand the left-hand side to get :
\begin{equation}
\label{eqn:vertical_PY}
\begin{aligned} \left( \frac{1 + \mu |\nabla_x \eta|^2}{\varrho\nouveau{(1+\epsilon h)}} \right)\partial_r^2 P^{\diff}  &= - \mu \nabla_x \cdot \left(\frac{(1+\epsilon h)}{\varrho} \nabla_x P^{\diff} + \epsilon \frac{\nabla_x \eta}{\varrho} \partial_r P^{\diff}\right) \nouveau{-} \mu \partial_r \left( \epsilon \frac{\nabla_x \eta \cdot \nabla_x P}{\varrho} \right) \nouveau{-}   \partial_r \left( \frac{1+ \mu \epsilon^2|\nabla_x \eta|^2}{\varrho(1+\epsilon h)}  \right) \partial_r P^{\diff}  \nouveau{-} (1+\epsilon h) R. \end{aligned}
\end{equation}
We take the $L^2$-norm of the previous expression. Recall that $1+\epsilon h \leq h^*$. We use the product estimate \eqref{apdx:pdt:tame}, and the composition estimate \cite[Lemma A.5]{Duchene2022} for $\frac{1}{1+\epsilon h } = 1 - \frac{\epsilon h}{1+\epsilon h}$.  We get
\begin{equation}
\label{eqn:elliptic:temp:5}
\Vert \partial_r^2 P^{\diff} \Vert_{L^2} \leq C (\Vert R \Vert_{L^2} + \Vert \gradphi[\mu]P^{\diff} \Vert_{H^{1,0}} ),
\end{equation}
by using that $\frac{h^*}{\rho_*}\frac{1}{\varrho (1+\epsilon h)} \geq c_* h_* > 0$.\\
The estimation of the terms in $R$ is very similar to the previous step, except for two terms. The first one is the term $\gradphi[\mu] \cdot \Ral{1}{\nablamu  P}$; it is enough to bound $\Vert\Ral{1}{\nablamu P} \Vert_{H^{1}(S_r)}$. We already have the $L^2$ bound \eqref{apdx:alinhac:R1}. Recall that, up to a harmless factor, $ \Ral{1}{\nablamu  P}$ is the symmetric commutator $ [\diff; (\partial \varphi)^{T}, \gradphi[\mu] P]$ so that for $L = \partial_{x_i}$ (with $i \in \{1, \dots, d\}$) or $L= \partial_r$, we write, by Leibniz' rule and $[L,\diff]=0$:
$$ L [\diff; (\partial \varphi)^{T}, \gradphi[\mu] P]= [\diff; L(\partial \varphi)^T , \gradphi[\mu] P ] + [\diff; (\partial \varphi)^T , L(\gradphi[\mu] P) ].$$
For the first term we use the commutator estimate \eqref{apdx:commutator_sym:gen}. For the second term, the same estimate would yield the term $\Vert \gradphi[\mu] P \Vert_{H^{s_0+\frac52,3}}$, which is not controlled by induction. We rather expand the symmetric commutator to write
$$ \Vert [\diff; (\partial \varphi)^T , (L\gradphi[\mu] P) ] \Vert_{L^2} \leq \Vert [\diff; (\partial \varphi)^T ] (L\gradphi[\mu] P) \Vert_{L^2} + \vert (\diff (\partial \varphi)^T) (L\gradphi[\mu] P) \Vert_{L^2}.$$
Then, using the commutator estimate \eqref{apdx:commutator:gen} for the first term and the product estimate \eqref{apdx:pdt:tame} for the second one, we eventually get
\begin{equation}
\label{eqn:elliptic:Ral:H1}
\Vert \Ral{1}{\nablamu  P} \Vert_{H^1} \leq C\Vert \gradphi[\mu] P \Vert_{H^{s-1,2 \vee k\wedge(s-1)}}.
\end{equation}
The second term in $R$ that we need to control is $ \gradphi[\mu] \cdot [\diff,\frac{1}{\varrho}] \gradphi[\mu]P$. Note however that, as $\varrho$ only depends on $r$, in the case $k \geq 2$ we get 
$$[\diff,\frac{1}{\varrho}] = [\partial_r^{k-1}, \frac{1}{\varrho}] \Lambda^{s-k},$$ and the commutator estimate \eqref{apdx:commutator:infty} allows us to conclude 
$$ \Vert \gradphi[\mu] \cdot [\diff,\frac{1}{\varrho}] \gradphi[\mu]P \Vert_{L^2} \leq C \Vert \gradphi[\mu] P \Vert_{H^{s-1,k-1}}.$$
We now turn to the control of $\Vert \gradphi[\mu]P^{\diff} \Vert_{H^{1,0}}$ in \eqref{eqn:elliptic:temp:5}. First note that $\Vert \gradphi[\mu]P^{\diff} \Vert_{L^2}$ is controlled by induction and \eqref{apdx:alinhac:grad}. We apply $\partial_{x_i}$ (with $i \in \{ 1 \dots, d \}$) to \eqref{apdx:alinhac:comm1} to get
$$ \begin{aligned} \Vert \gradphi[\mu]P^{\diff} \Vert_{H^{1,0}} &\leq C \Vert \gradphi[\mu] P \Vert_{H^{s,k-1}} + C \epsilon \Vert \eta \Vert_{H^{s,k-1}} \Vert \gradphi[\mu] P \Vert_{H^{s_0+\frac32,2}} + C \epsilon \Vert \gradphi[\mu] P \Vert_{H^{s-1,2\vee k\wedge(s-1)}},
\end{aligned}$$
where we used \eqref{eqn:elliptic:Ral:H1} to control the term $\Ral{1}{\nablamu P}$ from \eqref{apdx:alinhac:comm1}. Therefore, by induction, we have a bound for $\Vert \gradphi[\mu]P^{\diff} \Vert_{H^{1,0}}$. That is to say, we have a bound for $\Vert \partial_r \partial^{\varphi}_r P^{\diff} \Vert_{L^2}$ thanks to \eqref{eqn:elliptic:temp:5}. \\
Now applying $\partial_r$ to \eqref{apdx:alinhac:comm1} and using \eqref{eqn:elliptic:Ral:H1} to control the term $\Ral{1}{\nablamu P}$ from the right-hand side, we can write
$$ \Vert \partial_r  \Lambda^{s-k}\partial_r^{k-1} \partial_r^{\varphi} P \Vert_{L^2} \leq  \Vert \partial_r \partial_r^{\varphi} P^{\diff} \Vert_{L^2} + C \Vert \eta \Vert_{H^{s,k}} \Vert \nablamu P \Vert_{H^{s_0+\frac32,2}} +C \epsilon \Vert \gradphi[\mu] P \Vert_{H^{s-1,2\vee k\wedge(s-1)}} . $$
We thus get a control on $\Vert \Lambda^{s-k}\partial_r^{k}\partial_r^{\varphi}P \Vert_{L^2}$ by induction and the control on $\Vert \partial_r \partial_r^{\varphi} P^{\diff} \Vert_{L^2}$ so that by \eqref{eqn:reduction_vertical} and induction, we get, for any $s \geq s_0+\frac52$ and $0\leq k \leq s$
\begin{equation}
\label{eqn:elliptic:temp:6}
\begin{aligned}
\Vert \gradphi[\mu] P \Vert_{H^{s,k}} &\leq C\left( \Vert \Gun \Vert_{H^{s-1,k-1}} + \Vert \Gdeux \Vert_{H^{s,k}} +\frac{1}{\delta}\Vert \gradphi[\mu] P \Vert_{H^{s_0+\frac32,2}} \right)+ C \epsilon \delta \Vert\gradphi[\mu] P \Vert_{H^{s,2}} + C\frac{1}{\delta} \Vert\gradphi[\mu] P \Vert_{H^{s-1,1}}.
\end{aligned}
\end{equation}
We choose $\delta \leq \frac{1}{2C}$ so that, as $\epsilon \leq 1$, we have $C \epsilon \delta \leq \frac12$ and the last term in \eqref{eqn:elliptic:temp:6} can be absorbed on the left-hand side when $k \geq 2$. Note that $\delta$ is indeed independent of $\epsilon$ and $\mu$, in particular $\frac{1}{\delta}$ remains bounded. This yields \eqref{eqn:elliptic:2}.

\end{proof}
We now sketch the proof of Proposition \ref{lemma:elliptic:2:tilde}.
\begin{proof}[Proof of Proposition \ref{lemma:elliptic:2:tilde}]
\newcommand{\diff}{\mathbb{\Lambda}}
\newcommand{\Funtilde}{\tilde{\vec{F}}_1}
\newcommand{\Fzerotilde}{\tilde{f}_0}
\newcommand{\Fdeuxtilde}{\tilde{f}_2}
Recall that $s \geq s_0+\frac52$. First, we proceed as for Step 2 of the proof of Proposition \ref{lemma:elliptic:2}. Applying $\diff := |D|^2 \Lambda^{s-2}$ to the equations \nouveau{\eqref{eqn:elliptic:gen:2} and \eqref{eqn:elliptic:gen:2:tilde}} satisfied by $P$ and $\tilde{P}$ respectively, using Alinhac's good unknown, multiplying by $\nouveau{-}(1+\epsilon h)$ and $\nouveau{-}(1+\epsilon \tilde{h})$ respectively and then taking the difference between the two equations yields:
\begin{equation}
\label{eqn:elliptic:2:tilde:temp}
\begin{aligned}
\nouveau{-}(1+\epsilon h) \nabla_{\mu}^{\varphi} \cdot \frac{1}{\varrho} \nabla_{\mu}^{\varphi} (P^{\diff} - \tilde{P}^{\diff}) &= \epsilon(\Fzero - \Fzerotilde ) + (1+\epsilon h)\gradphi[\mu] \cdot \Fun - (1+\epsilon \tilde{h}) \nabla^{\tilde{\varphi}}_{\mu} \cdot \Funtilde  + \sqrt{\mu} |D|( \Fdeux - \Fdeuxtilde ) \\
&+ \left( (1+\epsilon h)\nabla_{\mu}^{\varphi} \cdot \frac{1}{\varrho} \nabla_{\mu}^{\varphi} - (1+\epsilon \tilde{h}) \nabla_{\mu}^{\tilde{\varphi}} \cdot \frac{1}{\varrho} \nabla_{\mu}^{\tilde{\varphi}}\right) \tilde{P}^{\diff}, 
\end{aligned}
\end{equation}
where $\Fzero,\Fun,\Fdeux$ are defined in \eqref{eqn:elliptic:2:restes}, and $\Fzerotilde ,\Funtilde ,\Fdeuxtilde $ are defined from $\tilde{P},\tilde{\eta},\tilde{h},\Guntilde,\Gdeuxtilde$ through \eqref{eqn:elliptic:2:restes}. \nouveau{ We multiply \eqref{eqn:elliptic:2:tilde:temp} by $P^{\diff} - \tilde{P}^{\diff}$ and integrate over $S_r$. We integrate by parts using Lemma \ref{apdx:IPP} for the left-hand side, the second, third and last terms on the right-hand side, and the boundary conditions \eqref{eqn:elliptic:gen:2:bc} and \eqref{eqn:elliptic:gen:2:tilde:bc} yield that the boundary terms cancel each other out. Using \eqref{hyp:non_cavitation} and \eqref{hyp:h:elliptic} yields that there exists $C > 0$ such that }
\nouveau{
\begin{equation}
\label{eqn:elliptic:2:tilde:temp2}
\begin{aligned}
\Vert \gradphi[\mu] (P^{\diff} - \tilde{P}^{\diff}) \Vert_{L^2}^2 &\leq C \epsilon \Vert \Fzero - \Fzerotilde  \Vert_{L^2} \Vert P^{\diff}- \tilde{P}^{\diff} \Vert_{L^2} + C \Vert \Fdeux - \Fdeuxtilde  \Vert_{L^2}\Vert\sqrt{\mu} |D| (P^{\diff} - \tilde{P}^{\diff}) \Vert_{L^2} \\
& + C \Vert \Fun - \Funtilde  \Vert_{L^2} \Vert (1+\epsilon h)\gradphi[\mu](P^{\diff} - \tilde{P}^{\diff}) \Vert_{L^2} + C \epsilon \Vert \Funtilde  \Vert_{L^2} \Vert \eta - \tilde{\eta} \Vert_{W^{1,\infty}} \Vert \nabla_{\mu} (P^{\diff} - \tilde{P}^{\diff})\Vert_{L^2}\\
&+ C \epsilon \Vert \eta - \tilde{\eta} \Vert_{W^{1,\infty}} \Vert \nabla_{\mu} \tilde{P}^{\diff} \Vert_{L^2} \Vert \gradphi[\mu](P^{\diff} - \tilde{P}^{\diff})\Vert_{L^2} + C \epsilon \Vert \nabla^{\tilde{\varphi}}_{\mu} \tilde{P}^{\diff}\Vert_{L^2}\Vert h - \tilde{h} \Vert_{L^{\infty}} \Vert \gradphi[\mu](P^{\diff} - \tilde{P}^{\diff}) \Vert_{L^2}  \\
&+ C\epsilon \Vert \nabla^{\tilde{\varphi}}_{\mu} \tilde{P}^{\diff} \Vert_{L^2} \Vert \eta - \tilde{\eta}\Vert_{W^{1,\infty}} \Vert \nabla_{\mu}(P^{\diff} - \tilde{P}^{\diff})\Vert_{L^2},
\end{aligned}
\end{equation}
}
where, for the terms in divergence form in \eqref{eqn:elliptic:2:tilde:temp}, we used 
$$(1+\epsilon h) \nabla^{\varphi}_{\mu} - (1+\epsilon \tilde{h}) \nabla^{\tilde{\varphi}}_{\mu} =  \nouveau{\epsilon (h-\tilde{h}) \begin{pmatrix}
\sqrt{\mu} \nabla_x \\ 0 \end{pmatrix} + \epsilon \begin{pmatrix} \sqrt{\mu} \nabla_x (\eta - \tilde{\eta}) \\0 \end{pmatrix} \partial_r,}$$
which is computed from the definition \eqref{eqn:def:gradphimu} of $\gradphi[\mu]$ and $\nabla^{\tilde{\varphi}}_{\mu}$.
Similarly to Step 2 of the proof of Proposition \ref{lemma:elliptic:2}, it remains to bound from above the terms $\Vert \Fzero - \Fzerotilde  \Vert_{L^2},\Vert \Fun - \Funtilde  \Vert_{L^2},\Vert \Fdeux - \Fdeuxtilde  \nouveau{\Vert_{L^2}} $, in which each term is multilinear in $P,\eta,\Gun,\Gdeux$ from $H^{s,k}(S_r)$ to $L^2\nouveau{(S_r)}$, so that we only treat the fourth term in $\Fzero - \Fzerotilde $, as a generic example. As from \eqref{eqn:def:gradphi} we get $(1+\epsilon h) \partial^{\varphi}_r = \nouveau{-}\partial_r$ (resp. $(1+\epsilon \tilde{h}) \partial^{\tilde{\varphi}}_r = \nouveau{-}\partial_r$), we write
\begin{equation}
\label{eqn:elliptic:2:tilde:temp3}
\begin{aligned}
(1+\epsilon h) \diff \eta \partial^{\varphi}_r \gradphi[\mu] \cdot \Gdeux - (1+\epsilon \tilde{h}) \diff \tilde{\eta} \partial^{\tilde{\varphi}}_r \nabla_{\mu}^{\tilde{\varphi}} \cdot \Gdeuxtilde &= \nouveau{-}\diff \eta \partial_r \nabla^{\varphi}_{\mu} \cdot (\Gdeux - \Gdeuxtilde) \nouveau{-} \diff \eta \partial_r ((\nabla_{\mu}^{\varphi} - \nabla_{\mu}^{\tilde{\varphi}} ) \cdot \Gdeuxtilde) \\
&\nouveau{-} (\diff \eta - \diff \tilde{\eta}) \partial_r \nabla_{\mu}^{\tilde{\varphi}} \cdot \Gdeuxtilde.
\end{aligned}
\end{equation}
From the definition \eqref{eqn:def:gradphimu} of $\gradphi[\mu]$ we infer
\begin{equation}
\label{eqn:elliptic:2:tilde:temp4}
\nabla_{\mu}^{\varphi} - \nabla_{\mu}^{\tilde{\varphi}} = \epsilon \left(\frac{1}{1+\epsilon h} \begin{pmatrix} \sqrt{\mu} \nabla_x \\ - \partial_r \end{pmatrix} (\eta - \tilde{\eta}) + \epsilon\frac{\tilde{h} - h}{(1+\epsilon h)(1+\epsilon \tilde{h})} \begin{pmatrix} \sqrt{\mu} \nabla_x \\ - \partial_r \end{pmatrix} \tilde{\eta} \right) \partial_r.
\end{equation}
From \eqref{eqn:elliptic:2:tilde:temp3}, using the triangular inequality, \nouveau{Hölder's inequality}, the product estimate \eqref{apdx:pdt:tame}, \eqref{eqn:elliptic:2:tilde:temp4} \nouveau{and \eqref{hyp:h:elliptic},} we get
\begin{equation}
\label{eqn:elliptic:2:tilde:temp5}
\Vert (1+\epsilon h) \diff \eta \partial^{\varphi}_r \gradphi[\mu] \cdot \Gdeux - (1+\epsilon \tilde{h}) \diff \tilde{\eta} \partial^{\tilde{\varphi}}_r \nabla_{\mu}^{\tilde{\varphi}} \cdot \Gdeuxtilde \Vert_{L^2} \leq C \left( \Vert \Gdeux - \Gdeuxtilde \Vert_{H^{s,3}} + \Vert \eta - \tilde{\eta} \Vert_{H^{s,3}} \right), 
\end{equation}
where we used $s \geq s_0+\frac52$. The other terms in \eqref{eqn:elliptic:2:tilde:temp2} are treated the same way and we omit them. We then get the following estimate, reminiscent of \eqref{eqn:elliptic:temp:3}:
\nouveau{
\begin{equation}
\label{eqn:elliptic:2:tilde:temp:6}
\begin{aligned}
\Vert \gradphi[\mu] (P^{\diff} - \tilde{P}^{\diff}) \Vert_{L^2}&\leq C \delta \left( \sqrt{\mu} \Vert \Gun - \Guntilde \Vert_{H^{s-1,2}}+ \Vert \Gdeux - \Gdeuxtilde \Vert_{H^{s,3}} + \epsilon \Vert \gradphi[\mu] P - \nabla^{\tilde{\varphi}}_{\mu} \tilde{P} \Vert_{H^{s,2}} + \epsilon \Vert \eta - \tilde{\eta} \Vert_{H^{s,3}}\right) \\
&+ C\frac{\epsilon}{\delta} \left(\frac{1}{\sqrt{\mu}}\Vert \nabla_{\mu} (P - \tilde{P}) \Vert_{H^{s-1,0}} + \Vert \gradphi[\mu] (P-\tilde{P}) \Vert_{H^{s_0+\frac12,1}} + \epsilon \Vert \eta - \tilde{\eta} \Vert_{H^{s,3}}\right) \\
&+ C \epsilon  \left(\Vert \Gdeux - \Gdeuxtilde \Vert_{H^{s,2}} + \Vert \gradphi[\mu] (P- \tilde{P}) \Vert_{H^{s-1,0}} + \Vert \gradphi[\mu] (P-\tilde{P}) \Vert_{H^{s_0+\frac32,2}} \right) +C\Vert \Gun - \Guntilde \Vert_{H^{s-1,1}} + C \epsilon \Vert \eta - \tilde{\eta} \Vert_{H^{s,3}};
\end{aligned}
\end{equation}}the only difference between \eqref{eqn:elliptic:2:tilde:temp:6} and \eqref{eqn:elliptic:temp:3} is the presence of the terms $\Vert \eta - \tilde{\eta} \Vert_{H^{s,3}}$, which come for instance from the last term in \eqref{eqn:elliptic:2:tilde:temp5}. We write:
\begin{equation}
\label{eqn:elliptic:2:tilde:temp:7}
\Vert \nabla^{\varphi}_{\mu} P^{\diff} - \nabla^{\tilde{\varphi}}_{\mu} \tilde{P}^{\diff} \Vert_{L^2} \leq \Vert \nabla^{\varphi}_{\mu} (P^{\diff} - \tilde{P}^{\diff}) \Vert_{L^2} + \Vert \eta - \tilde{\eta}\Vert_{H^{s_0+\frac32,2}} \Vert \partial_r \tilde{P}^{\diff} \Vert_{L^2},
\end{equation}
using \eqref{eqn:elliptic:2:tilde:temp4} and the product estimate \eqref{apdx:pdt:tame}. Using now \eqref{apdx:alinhac:comm1}, the left-hand side of \eqref{eqn:elliptic:2:tilde:temp:7} controls $\Vert \nabla^{\varphi}_{\mu} P - \nabla^{\tilde{\varphi}}_{\mu} \tilde{P}\Vert_{H^{s,0}}$ up to remainder terms; more precisely, from \eqref{eqn:elliptic:2:tilde:temp:7} and \eqref{eqn:elliptic:2:tilde:temp:6} and using  \eqref{apdx:alinhac:comm1}, we eventually get the following estimate, reminiscent of \eqref{eqn:elliptic:temp:4}:
\begin{equation}
\label{eqn:elliptic:2:tilde:temp:8}
\begin{aligned}
\Vert \gradphi[\mu] P - \nabla^{\tilde{\varphi}}_{\mu}\tilde{P} \Vert_{H^{s,0}}&\leq \frac{C}{\delta} \left( \Vert \Gun - \Guntilde \Vert_{H^{s-1,2}} + \Vert \Gdeux - \Gdeuxtilde \Vert_{H^{s,3}} + \Vert \eta - \tilde{\eta}\Vert_{H^{s,3}} + \nouveau{\Vert \gradphi[\mu] (P-\tilde{P}) \Vert_{H^{s_0+\frac32,2}}}\right) \\
&+ C\epsilon \delta \Vert \gradphi[\mu] P - \nabla^{\tilde{\varphi}}_{\mu}\tilde{P} \Vert_{H^{s,2}} + C\frac{1}{\delta} \Vert \gradphi[\mu] (P-\tilde{P}) \Vert_{H^{s-1,1}},
\end{aligned}
\end{equation}
where we used $\delta \leq 1$ to merge some terms coming from \eqref{eqn:elliptic:2:tilde:temp:6}. Then, a similar adaptation of Step 3 of the proof of Proposition \ref{lemma:elliptic:2} yields \eqref{eqn:elliptic:2:tilde} of Proposition \ref{lemma:elliptic:2:tilde}.
\end{proof}

\subsubsection{Proof of Proposition \ref{lemma:elliptic_high_reg}}
First, we apply Proposition \ref{lemma:elliptic:1} to the equation \eqref{eqn:elliptic}, \nouveau{ with the boundary condition \eqref{eqn:elliptic:bc}}, and with 
$$\G := \nouveau{-}\left( \frac{\sqrt{\mu}}{\epsilon}  (\vec{\Ub + \epsilon U}) \cdot \gradphi (\Vb + \epsilon V),\mu (\vec{\Ub + \epsilon U}) \cdot \gradphi w + b\right)^T.$$
Then, \eqref{apdx:NL:1} and \eqref{apdx:NL:buoyancy} from Lemma \ref{lemma:NL}  yields
$$ \Vert \G \Vert_{H^{s,k}} \leq C\left( \Vert \eta \Vert_{H^{s +1 ,k +1 }} + \Vert V \Vert_{H^{s +1 ,k +1}} + \sqrt{\mu}\Vert w \Vert_{H^{s +1 , k +1}} \right)$$
for $s \geq s_0+\frac32$ and $1 \leq k \leq s$. We thus get the existence and uniqueness of the pressure $P$, as well as the bound
$$ \Vert \gradphi[\mu] P \Vert_{H^{s_0+\frac32,k}} \leq C\left( \Vert \eta \Vert_{H^{s_0+\frac52,k +1 }} + \Vert V \Vert_{H^{s_0+\frac52,k +1 }} + \sqrt{\mu}\Vert w \Vert_{H^{s_0+\frac52,k +1 }}\right),$$
for $0 \leq k \leq s_0+\frac32$.\\
Because $\gradphi \cdot (\vec{\Ub + \epsilon U})=0$, the pressure $P$ also satisfies \eqref{eqn:elliptic:ordre1}, with the boundary condition \eqref{eqn:elliptic:bc}. We use Proposition \ref{lemma:elliptic:2} with 
$$\Gun := \frac{\sqrt{\mu}}{\epsilon} \NLdeux,\qquad
		\Gdeux := (0,b)^T.$$
Using also \eqref{apdx:NL:1} and \eqref{apdx:NL:buoyancy} from Lemma \ref{lemma:NL}, we directly get \eqref{eqn:elliptic:high_reg}, for $s \geq s_0+\frac52$, $3 \leq k \leq s$.
\subsubsection{Proof of Proposition \ref{lemma:elliptic:high_reg:perte}}
We start by defining a pressure, reminiscent of the hydrostatic pressure
$$P_h(t,x,r) := \int_0^r \varrho(r') b(t,x,r')dr',$$
so that 
$$ \frac{1}{\varrho} \partial_r^{\varphi}P_h + b = \epsilon \frac{h}{1+\epsilon h} b. $$
By the product estimate \eqref{apdx:pdt:infty} and the estimate \eqref{apdx:NL:buoyancy} on $b$ in Lemma \ref{lemma:NL}, we readily get 
$$ \Vert \gradphi P_h \Vert_{H^{s,k}} \leq C \Vert \eta \Vert_{H^{s+1,k+1}}.$$
Notice that there is a loss of derivative, but the estimate is uniform in $\mu$.\\
We now define $P_{\nh} := P - P_h$, and look for an estimate on $\Vert \gradphi P_{\nh} \Vert_{H^{s,k}}$ uniform in $\mu$. Together with the estimate on $P_h$, it then shows the desired estimate. \\
Plugging the previous relation into \eqref{eqn:elliptic} yields 
$$ \begin{aligned} (1+\epsilon h) \gradphi[\mu] \cdot \frac{1}{\varrho} \gradphi[\mu] P_{\nh}  &= - \frac{\mu}{\epsilon} (1+\epsilon h) \NLdeux \nouveau{+}  \epsilon \partial_r \left( b\frac{h}{1+\epsilon h} \right) - \mu \nouveau{(1+\epsilon h)}\gradphi[x] \cdot \frac{1}{\varrho} \gradphi[x] P_h, \end{aligned}$$
with the boundary condition
$$ \frac{1}{\varrho} \left. \partial_{r}^{\varphi} P_{\nh}\right|_{r=0; r=1} = \left. - \epsilon b \frac{h}{1+\epsilon h} \right|_{r=0; r=1}.$$
Using Proposition \ref{lemma:elliptic:2}, and \eqref{apdx:NL:2} (allowing a better dependency in $\mu$, at the price of a loss of derivatives) and \eqref{apdx:NL:buoyancy} from Lemma \ref{lemma:NL}, we get
$$ \Vert \gradphi[\mu] P_{\nh} \Vert_{H^{s,k}} \leq (\sqrt{\mu} + \epsilon) C\left( \Vert \eta \Vert_{H^{s+1,k+1}} + \Vert V \Vert_{H^{s+1,k+1}} + \sqrt{\mu}\Vert w \Vert_{H^{s+1,k+1}} \right).$$
Recalling that $\gradphi[\mu] := \begin{pmatrix}
\sqrt{\mu} \gradphi[x] \\ \partial_r^{\varphi}
\end{pmatrix}$ yields the desired estimate, after dividing by $\sqrt{\mu}$ and using \eqref{hyp:epsilonmu}.
\section{A priori energy estimates}
\label{section:energy}
The goal of this section is to derive energy estimates for smooth solutions of \eqref{eqn:euler_slag}.
In Subsection \ref{subsection:functional}, we define the energy functional that we will use. One main concern is that, roughly speaking, we only control $\sqrt{\mu} w$ in the energy. Therefore, any source term containing a factor $w$ has to either cancel out (with a symmetry argument, see the proof of Proposition \ref{lemma:energy}), or be multiplied by a factor $\sqrt{\mu}$. 

In Subsection \ref{subsection:energy_low_reg}, we derive a first energy estimate, with loss of derivative. In Subsection \ref{subsection:energy_high_reg}, we use Alinhac's good unknown to improve this result into an energy estimate without loss of derivatives. Note that this improved estimate relies on the one derived in Subsection \ref{subsection:energy_low_reg}, for the low regularity case $ s \leq s_0+\frac32$.

\subsection{The energy functional}
\label{subsection:functional}
We define the energy that we will use in the following definition. Note that there is a low regularity part and a high regularity part.
\begin{definition}
\label{def:energy}
\newcommand{\diff}{\mathbb{\Lambda}}
Let $s_0>d/2$, $s \in \N$ with $s \geq s_0+\frac32$. Let $(V,w,\eta) \in C^0([0,T), (H^{s})^{d+2})$ for a certain $T>0$. \\
Let 
\begin{equation}
\label{eqn:def:D}
\mathcal{D} := \{ \Lambda^{s-l} \partial_r^{l}, 1 \leq l \leq s \} \cup \{|D|^2 \Lambda^{s-2}\}
\end{equation}
be the set of operators that we will use in order to define the energy functional. \\
We define, for $0 \leq t <T$:  

$$  \Elow(t) := \Vert V(t,\cdot)\Vert_{H^{s_0+\frac32,2}}^2  + \mu \Vert w(t,\cdot) \Vert_{H^{s_0+\frac32,2}}^2 + \Vert \eta(t,\cdot)\Vert_{H^{s_0+\frac32,2}}^2,$$
and \nouveau{(assuming \eqref{hyp:stable} and \eqref{hyp:h:elliptic})}
\begin{equation}
\label{eqn:energy_def}
 \E(t) := \Elow + \sum_{ \diff \in \mathcal{D}} \Vert \sqrt{\varrho (1+\epsilon h) }V^{\diff}(t,\cdot)\Vert_{L^2}^2  + \mu \Vert\sqrt{\varrho (1+\epsilon h)} w^{\diff}(t,\cdot) \Vert_{L^2}^2  + \Vert  \sqrt{\varrho'} \diff \eta(t,\cdot)\Vert_{L^2}^2, 
\end{equation}
where $V^{\diff}$ and $w^{\diff}$ are Alinhac's good unknowns associated with the operator $\diff$, defined in \eqref{apdx:alinhac:definition}.
\end{definition}
In the next lemma we prove that this energy is equivalent to the $H^{s}$-norm of the unknowns $(V,w,\eta)$.
\begin{lemma}
\label{lemma:equiv_energy}
\newcommand{\diff}{\mathbb{\Lambda}}
Let $s_0>d/2$, $s \in \N$ with $s \geq s_0+\frac32$. Let $(V,w,\eta) \in C^0([0,T), (H^{s})^{d+2})$ for a certain $T>0$. Recall the definition \eqref{eqn:def:D} of $\mathcal{D}$. We make the assumptions \eqref{hyp:non_cavitation}, \eqref{hyp:stable}, \eqref{hyp:h:elliptic} as well as 
$$\Vert \eta \Vert_{H^{s}} \leq M,$$
for $M > 0$ a constant. 
Then there exists a constant $C>0$ depending only on $M$ such that
$$ \frac{1}{C} \left( \Vert V \Vert_{H^{s}}^2 + \mu \Vert w \Vert_{H^{s}}^2 + \Vert \eta \Vert_{H^{s}}^2 \right) \leq \E \leq C \left( \Vert V \Vert_{H^{s}}^2 + \mu \Vert w \Vert_{H^{s}}^2 + \Vert \eta \Vert_{H^{s}}^2 \right). $$
\end{lemma}
\begin{proof}
Recall that $H^s(S_r) = H^{s,k}(S_r)$ when $k = s \in \N$. The assumptions \eqref{hyp:non_cavitation}, \eqref{hyp:stable}, \eqref{hyp:h:elliptic} yield the existence of a constant $c > 0$ such that $ \frac{1}{c} \leq \varrho (1+\epsilon h) \leq c$ and $\frac{1}{c} \leq \varrho' \leq c$. Using this, as well as \eqref{apdx:alinhac:basic} on $V$ and $\sqrt{\mu} w$ yields the result.
\end{proof}

\subsection{Estimates with loss of regularity}
\label{subsection:energy_low_reg}
The main result of this subsection is the energy estimate with loss of derivatives in Proposition \ref{lemma:energy_low_reg}. Note that in the estimate of Proposition \ref{lemma:energy_low_reg}, the functional $\Elow$ only controls $H^{s_0+\frac32}$-norms of the unknowns, whereas it is bounded by a constant depending on the $H^{s_0+\frac52}$-norms of the unknowns.\\

{
\newcommand{\diff}{\mathbb{\Lambda}}
Let   $ \diff$ is a differential operator of the form $\Lambda^{s-l} \partial_r^l$ for $0 \leq l \leq s$. Applying $\diff$ to \eqref{eqn:euler_slag} formally yields

\begin{equation}
\label{eqn:euler_isopycnal_quasilin_low_reg}
\sys{ \partial_t \dot{V}  &= \Run, & (1)\\
\mu \partial_t \dot{w}  &=  \Rdeux, & (2) \\
\partial_t \dot{\eta}  &= \Rtrois,  & (3) \\}
\end{equation}
where $(\dot{V}, \dot{w}, \dot{\eta}) := \diff (V,w,\eta)$ and
$$\begin{aligned}
 \Run &:=   - \diff \left( \left(\Vb + \epsilon V \right)\cdot \nabla_x V + \frac{1}{\varrho} \nabla_x^{\varphi} P \right), \qquad 
 \Rtrois := - \diff \left( \left(\Vb + \epsilon V \right) \cdot \nabla_x \eta \nouveau{-} w \right),\\
 \Rdeux &:= - \mu \diff \left( \left( \Vb + \epsilon V\right) \cdot \nabla_x w \right) - \diff \left( \frac{1}{\varrho} \partial_r^{\varphi}P + b \right). \end{aligned}$$
We make the assumptions that the unknowns $(V,w,\eta)$ satisfy
\begin{hyp}
\label{hyp:Mzero:low_reg}
\Vert \eta(t,\cdot) \Vert_{H^{s_0+\frac52}} + \Vert V(t,\cdot) \Vert_{H^{s_0+\frac52}} + \sqrt{\mu} \Vert w(t,\cdot) \Vert_{H^{s_0+\frac52}} \leq M_0,
\end{hyp}with $M_0 > 0$ a constant. \\
We also make the assumption that the stratification $\varrho$ and the shear flow $\Vb$ satisfy
\begin{hyp}
\label{hyp:Mb:low_reg}
		 | \varrho|_{W_r^{s_0+\frac32,\infty}} + | \Vb |_{W_r^{s_0+\frac52, \infty}} \leq \Mb,
\end{hyp}with $\Mb > 0$ a constant.
}We now state the energy estimates with a loss of derivatives.
\begin{proposition}
\label{lemma:energy_low_reg}
\newcommand{\diff}{\mathbb{\Lambda}}
Let $s_0 > d/2$. We assume that there exist $T>0$ and $(V,w,\eta) \in C^0([0,T), (H^{s_0+\frac52})^{d+2})$ satisfying \eqref{eqn:euler_slag} as well as \eqref{hyp:non_cavitation}, \eqref{hyp:stable}, \eqref{hyp:h:elliptic}, \eqref{hyp:Mzero:low_reg}, \eqref{hyp:Mb:low_reg}.
Then there exists $C > 0$ a constant depending only on $M_0,\Mb$ such that
$$\frac{d}{dt} \Elow(t) \leq C\left( \Vert \eta(t,\cdot) \Vert_{H^{s_0+\frac52}} + \Vert V(t,\cdot) \Vert_{H^{s_0+\frac52}} + \sqrt{\mu} \Vert w(t,\cdot) \Vert_{H^{s_0+\frac52}} \right)^{2}.$$
\end{proposition}
The proof consists in estimating the terms $\Run,\Rdeux,\Rtrois$ with product, commutators and composition estimates, and Proposition \ref{lemma:elliptic_high_reg}. This is a standard strategy and we refer to \cite[Theorem 1]{Desjardins2019} and \cite[Lemmas 4.5 and 4.6]{Duchene2022} where very similar equations are thoroughly studied.
\subsection{Estimates without loss of regularity}
\label{subsection:energy_high_reg}
The main result of this subsection is the energy estimates without loss of derivatives provided in Proposition \ref{lemma:energy}. As opposed to the estimate in Proposition \ref{lemma:energy_low_reg}, the following energy estimate is without loss of derivatives. However, it requires the additional assumption $ |\Vb'|_{L^{\infty}} \leq \sqrt{\mu}$, see Remark \ref{rk:trick} for a more detailed comparison. The proof relies on the use of Alinhac's good unknowns. Before proving the proposition, we derive the quasilinearized system satisfied by Alinhac's good unknowns in Lemma \ref{lemma:quasilin}.\\

In what follows we assume 
\begin{hyp}
\label{hyp:M:high_reg}
\Vert \eta(t,\cdot) \Vert_{H^{s}} + \Vert V(t,\cdot) \Vert_{H^{s}} + \sqrt{\mu} \Vert w(t,\cdot) \Vert_{H^{s}} \leq M,
\end{hyp}
for $t \in [0,T]$ for some $T > 0$ and with $M > 0$ a constant. 
\begin{proposition}
\newcommand{\diff}{\mathbb{\Lambda}}
\label{lemma:energy}
Let $s_0 > d/2$, $s \in \N$ with $s \geq s_0 + \frac52$.  We make the assumption $\epsilon \leq \sqrt{\mu}$ as well as 
$$ |\Vb'|_{L^{\infty}} \leq \sqrt{\mu}. $$
Assume that there exist $T>0$ and $(V,w,\eta) \in C^0([0,T), (H^{s})^{d+2})$ satisfying \eqref{eqn:euler_slag} with the boundary conditions \eqref{eqn:euler_slag:bc} as well as  \eqref{hyp:non_cavitation}, \eqref{hyp:Mb}, \eqref{hyp:stable}, \eqref{hyp:h:elliptic}, \eqref{hyp:M:high_reg}. \\
Then there exists $C > 0$ a constant depending only on $\Mb,M$ such that 
$$\frac{d}{dt} \E(t) \leq C \E(t).$$
\end{proposition}

{
\newcommand{\diff}{\mathbb{\Lambda}}
Before proving Proposition \ref{lemma:energy}, let us first exhibit the quasilinear structure of the system. Let $s \geq s_0+\frac52$. Let   $ \diff$ be a differential operator of the form $ \diff := \Lambda^{s-k} \partial_r^k \text{ for }k \geq 1 \text{ or } |D|^2\Lambda^{s-2}$. Applying $\diff$ to \eqref{eqn:euler_slag} and using \eqref{apdx:alinhac:comm1} formally yields

\begin{equation}
\label{eqn:euler_isopycnal_quasilin}
\sys{ \partial_t V^{\diff} + \left(\Vb + \epsilon V \right)\cdot \nabla_x V^{\diff} + \frac{1}{\varrho} \nabla_x^{\varphi}P^{\diff} &= \epsilon \Run + \Rdeux, & (1)\\
\mu\left( \partial_t w^{\diff} + \left( \Vb + \epsilon V\right) \cdot \nabla_x w^{\diff} \right) + \frac{1}{\varrho} \partial_r^{\varphi}P^{\diff} + \frac{\varrho'}{\varrho} \dot{\eta}  &= F - \epsilon \diff \tilde{b} + \epsilon \Rtrois + \mu \Rquatre,  & (2) \\
\partial_t \dot{\eta} + \left(\Vb + \epsilon V \right) \cdot \nabla_x \dot{\eta} - w^{\diff} &= \Rcinq,  & (3) \\
\nabla_x^{\varphi} \cdot V^{\diff} + \partial_r^{\varphi}w^{\diff} &= \nouveau{-}\frac{\nabla_x \eta }{1+\epsilon h} \cdot \partial_r \Vb^{\diff} +\nouveau{ \Rsix}, & (4) }
\end{equation}
where $V^{\diff},w^{\diff}, \Vb^{\diff}, P^{\diff}$ are Alinhac's good unknowns defined through \eqref{apdx:alinhac:def}, $\dot{\eta} := \diff \eta$ and
\begin{equation}
\label{eqn:notation:quasilin}
\sys{
\tilde{b} &:= \frac{1}{\epsilon} \left( b - \frac{\varrho'}{\varrho} \eta \right), \\
 \Run &:=  \nouveau{-}\diff \eta \partial_r^{\varphi} \partial_t^{\varphi} V \nouveau{-}  \Ral{1}{\partial_t V} - (\Ub + \epsilon \vec{U})\cdot \left( \diff \eta \partial_r^{\varphi} \gradphi V + \Ral{1}{ \nabla V } \right) \nouveau{- \frac{1}{\varrho} \left(  \diff \eta \partial_r^{\varphi} \nabla_x^{\varphi} P + \Ral{1}{\nabla P}\right)}, \\
 \Rdeux &:= \nouveau{-}[\diff,\Ub + \epsilon \vec{U}] \cdot \gradphi V \nouveau{-} [\diff,\frac{1}{\varrho}]\gradphi[x] P, \\
 \Rtrois &:= \nouveau{-} \frac{1}{\varrho} \left( \diff \eta \partial_r^{\varphi} \partial_r^{\varphi} P + \Ral{1}{\nabla P} \right), \\
 \Rquatre &:= \nouveau{\nouveau{-} \epsilon \diff \eta \partial_r^{\varphi} \partial_t^{\varphi} w \nouveau{-} \epsilon \Ral{1}{\partial_t w} - \epsilon (\Ub + \epsilon \vec{U})\cdot \left( \diff \eta \partial_r^{\varphi} \gradphi w + \Ral{1}{\nabla w} \right)}\nouveau{-}[\diff,\Ub + \epsilon \vec{U} ]\cdot \gradphi w, \\
 \Rcinq &:= \nouveau{-}[\diff,\Vb + \epsilon V] \cdot \nabla_x \eta \nouveau{-} \epsilon \frac{\diff\eta}{1+\epsilon h} \partial_r w, \\
 \Rsix &:=  \nouveau{-\epsilon} \diff \eta\partial_r^{\varphi}(\nabla^{\varphi} \cdot \vec{U}) \nouveau{- \epsilon }\Ral{2}{\nabla  \vec{U}}  \nouveau{-} \diff \eta\partial_r^{\varphi}(\nabla^{\varphi} \cdot \Vb) - \Ral{2}{\nabla  \Vb}, \\
F &:=  \nouveau{-}[\diff,\frac{1}{\varrho}]\partial_r^{\varphi}P \nouveau{-} [\diff,\frac{\varrho'}{\varrho}]\eta.
}
\end{equation}
We justify this quasilinear form in the following lemma.

\begin{lemma}
\label{lemma:quasilin}
Let $s_0 > d/2$, $s \in \N$ with $s > s_0 + \frac52$. We make the assumption $\epsilon \leq \sqrt{\mu}$. Assume that there exists $T>0$ and $(V,w,\eta) \in C^0([0,T), (H^{s})^{d+2})$ satisfying \eqref{eqn:euler_slag} as well as \eqref{hyp:non_cavitation}, \eqref{hyp:Mb}, \eqref{hyp:stable}, \eqref{hyp:h:elliptic}, \eqref{hyp:M:high_reg}.

Then there exists $C > 0$ a constant depending only on $\Mb,M$ such that
$$\sys{ 
\Vert \tilde{b} \Vert_{H^{s}} &\leq C\left( \Vert \eta(t,\cdot) \Vert_{H^{s}} + \Vert V(t,\cdot) \Vert_{H^{s}} + \sqrt{\mu} \Vert w(t,\cdot) \Vert_{H^{s}} \right), \\
\Vert (\Run,\Rdeux,\Rtrois,\Rquatre,\Rcinq,\Rsix) \Vert_{L^2} &\leq C\left( \Vert \eta(t,\cdot) \Vert_{H^{s}} + \Vert V(t,\cdot) \Vert_{H^{s}} + \sqrt{\mu} \Vert w(t,\cdot) \Vert_{H^{s}} \right), \\
\Vert F \Vert_{L^2} &\leq C\left( \Vert \eta(t,\cdot) \Vert_{H^{s}} + \Vert V(t,\cdot) \Vert_{H^{s}} + \sqrt{\mu} \Vert w(t,\cdot) \Vert_{H^{s}} \right).
}$$
Moreover, if $\diff = |D|^2 \Lambda^{s-2}$ then $ F = 0$.
\end{lemma}
}
\begin{proof}

\newcommand{\diff}{\mathbb{\Lambda}}
First note that the gradient of pressure $\nabla^{\varphi} P$ is well defined and has the same regularity as $V$,$w$, and $\eta$ according to Proposition \ref{lemma:elliptic_high_reg}.\\
To obtain estimates on the unknowns, we apply   $ \diff := \Lambda^{s-k} \partial_r^k$  or $|D|^2 \Lambda^{s-2}$ to \eqref{eqn:euler_phi}.
Recall the definition \eqref{apdx:alinhac:def} of Alinhac's good unknowns $(V^{\diff},w^{\diff},P^{\diff})$. we now apply $\diff$ to \eqref{eqn:euler_phi}; using \eqref{apdx:alinhac:comm1} yields
\begin{equation}
\sys{ \partial_t^{\varphi} V^{\diff} + \left(\Ub + \epsilon \vec{U} \right)\cdot \gradphi V^{\diff} + \frac{1}{\varrho} \nabla_x^{\varphi}P^{\diff} &= \epsilon \Run + \Rdeux & (1)\\
\mu\left( \partial_t^{\varphi} w^{\diff} + \left( \Ub + \epsilon \vec{U} \right) \gradphi w^{\diff} \right) + \frac{1}{\varrho} \partial_r^{\varphi}P^{\diff} + \frac{\varrho'}{\varrho} \dot{\eta} &= F - \epsilon \diff \tilde{b} + \epsilon \Rtrois + \mu \Rquatre,  & (2) \\
\partial_t \dot{\eta} + \left(\Vb + \epsilon V \right) \cdot \nabla_x \dot{\eta} - w^{\diff} &= \Rcinq  & (3) \\
\nabla_x^{\varphi} \cdot V^{\diff} + \partial_r^{\varphi}w^{\diff} &= -\frac{\nabla_x \eta }{1+\epsilon h} \cdot \partial_r \Vb^{\diff} + \Rsix,, & (4) }
\end{equation}
with the notations defined in \eqref{eqn:notation:quasilin}. 
Note that using the semi-Lagrangian property \eqref{eqn:slag} of the isopycnal coordinates, we get exactly \eqref{eqn:euler_isopycnal_quasilin}. \\
In what follows, $c$ denotes a constant that only depends on $s_0,s$. For simplicity we write 
$$ \tilde{M} := \Vert \eta(t,\cdot) \Vert_{H^{s}} + \Vert V(t,\cdot) \Vert_{H^{s}} + \sqrt{\mu} \Vert w(t,\cdot) \Vert_{H^{s}}$$
so that we now bound the $H^{s}$-norm of $\tilde{b}$ and the $L^2$-norm of the other terms from above by $C \tilde{M}$.\\
\underline{\bf Estimate for $\tilde{b}$:}\\
Let us write 
$$F(r,y) := \frac{1}{\epsilon^2}\left(   1 - \frac{\varrho(r + y)}{\varrho}  - \frac{\varrho'(r)}{\varrho} y \right), $$
so that $\tilde{b} = F(r,\epsilon \eta)$. Now note that the assumed regularity on $\varrho$ implies that $F$ is at least $W^{3,\infty}$. \\
As $F(0) = \partial_y F = 0$, we can write, by Taylor-Laplace formula:
$$ F(r,y) = \frac{1}{\epsilon^2}\frac12 y^2 \int_0^1 \partial_y^2F(r,\alpha y) \alpha d\alpha .$$
Thus the composition estimate \cite[Lemma A.5]{Duchene2022} yields
$$ \Vert \tilde{b} \Vert_{H^{s}} \leq C \Vert \eta \Vert_{H^{s}} \leq C\tilde{M}.$$
Note in particular that there is no singularity in $\epsilon$.\\
\underline{\bf Estimate for $\Run$:}\\
For the first and third terms, we write
$$ \diff \eta \partial_r^{\varphi} \partial_t^{\varphi} V + (\Ub + \epsilon \vec{U})\cdot  \diff \eta \partial_r^{\varphi} \gradphi V = \diff \eta  \partial_r^{\varphi}\left(\partial_t^{\varphi}V + (\Ub + \epsilon \vec{U}) \cdot \gradphi V \right) - \diff \eta \partial_r^{\varphi}(\Ub + \epsilon \vec{U}) \cdot \gradphi V.$$
Taking the $L^2-$norm and using the product estimate \eqref{apdx:pdt:tame}, we can bound the right-hand side of the previous expression from above by
$$c \Vert \eta \Vert_{H^{s,k}} \left( \Vert \partial_r^{\varphi} \left(\partial_t^{\varphi}V + (\Ub + \epsilon \vec{U}) \cdot \gradphi V \right) \Vert_{H^{s_0+\frac12,1}} + \Vert \partial_r^{\varphi}(\Ub + \epsilon \vec{U}) \cdot \gradphi V \Vert_{H^{s_0+\frac12,1}} \right).$$
For the first term, we use the equation \eqref{eqn:euler_phi} together with the product estimate \cite[Lemma A.3]{Duchene2022}:
$$\Vert \partial_r^{\varphi} \left(\partial_t^{\varphi}V + (\Ub + \epsilon \vec{U}) \cdot \gradphi V \right) \Vert_{H^{s_0+\frac12,1}} \leq C \Vert \frac{1}{\varrho} \gradphi[x] P \Vert_{H^{s_0+\frac32,2}} \leq C\tilde{M}.$$
We also used the estimate on the pressure given by \eqref{eqn:elliptic_high_reg:perte} from Proposition \ref{lemma:elliptic:high_reg:perte}. \\
The second term is controlled by product estimates \cite[Lemma A.3]{Duchene2022} and \eqref{apdx:pdt:infty}, namely
$$\Vert \partial_r^{\varphi}(\Ub + \epsilon \vec{U}) \cdot \gradphi V \Vert_{H^{s_0+\frac12,1}} \leq C \left( |\Vb|_{W^{2,\infty}} + \epsilon \Vert \vec{U} \Vert_{H^{s_0+\frac32,2}} \right) \Vert V \Vert_{H^{s_0+\frac32,2}} \leq C\tilde{M}.$$This eventually yields a control on the first and third terms of $\Run$. \\
For the second term of $\Run$, the estimate \eqref{apdx:alinhac:R1} yields:
$$\Vert \Ral{1}{\partial_t V} \Vert_{L^2} \leq C   \Vert \nabla_{t,x,r}\eta \Vert_{H^{s-1}} \Vert \gradphi[t,x,r] V \Vert_{H^{s-1}}$$
Note that we crucially used $k=s$ in \eqref{apdx:alinhac:R1} in order to gain one derivative. We now use the equations \eqref{eqn:euler_phi} in order to trade the time derivatives in the right-hand side for space derivatives. Namely, taking the $H^{s-1}$ norm of the first and third equations in \eqref{eqn:euler_phi} yields
$$ \sys{ \Vert \partial_t^{\varphi} V \Vert_{H^{s-1}} &\leq \Vert (\Ub + \epsilon \vec{U}) \cdot \gradphi V + \frac{1}{\varrho} \gradphi[x] P \Vert_{H^{s-1}}, \\
\Vert \partial_t \eta \Vert_{H^{s-1}} &\leq \Vert (\Vb + \epsilon V) \cdot \nabla_x \eta - w \Vert_{H^{s-1}}. }$$We now use triangular inequality, product estimates \cite[Lemma A.3]{Duchene2022} and \eqref{apdx:pdt:infty} as well as the control on the pressure provided in Proposition \ref{lemma:elliptic_high_reg}, to write
$$\sys{ \Vert \partial_t^{\varphi} V \Vert_{H^{s-1}} &\leq C\tilde{M}, \\
		\Vert \partial_t \eta \Vert_{H^{s-1}} &\leq C(\Vert \eta \Vert_{H^s} + \Vert w \Vert_{H^{s-1}}).  }$$Eventually, we use the incompressibility constraint to estimate $\Vert w \Vert_{H^{s-1}}$ uniformly in $\mu$, by loosing one derivative, as detailed in Lemma \ref{apdx:lemma:w_trick}, to get
$$ \Vert \partial_t \eta \Vert_{H^{s-1}} \leq C\tilde{M}.$$
We thus write
$$\Vert \Ral{1}{\partial_t V} \Vert_{L^2} \leq C \epsilon  \Vert \eta \Vert_{H^{s}} \Vert V \Vert_{H^{s}} + C \epsilon  \Vert \eta \Vert_{H^{s}} \Vert  V \Vert_{H^{s}}.$$
The fourth term in $\Run$ is directly controlled using \eqref{apdx:alinhac:R1} and the product estimate \cite[Lemma A.3]{Duchene2022}. \\
For the last two terms of $\Run$, we use the product estimate \cite[Lemma A.3]{Duchene2022} and the bound \eqref{apdx:alinhac:R1} to get  
$$\left\Vert \frac{1}{\varrho} \left(\epsilon \diff \eta \partial_r^{\varphi} \nabla_x^{\varphi} P + \epsilon \Ral{1}{\gradphi P} \right)\right\Vert_{L^2}\leq \epsilon C  \Vert \gradphi[x] P \Vert_{H^{s,k}} \leq \frac{\epsilon }{\sqrt{\mu}} C\tilde{M},$$
where we used Proposition \ref{lemma:elliptic_high_reg} to estimate the pressure, and \eqref{hyp:non_cavitation}. Assuming $\epsilon \leq \sqrt{\mu}$ allows us to conclude
$$\Vert \Run \Vert_{L^2} \leq C \tilde{M}.$$
\underline{\bf Estimate for $\Rdeux$:}\\ For the first term of $\Rdeux$, we apply the commutator estimates \eqref{apdx:commutator:gen}, \eqref{apdx:commutator:horizontal} and \eqref{apdx:commutator:infty}. We distinguish the case of $\Vb$ and $V$ as one is controlled in $W^{s,\infty}$ and the other in $H^{s}$:
$$\begin{aligned}
\Vert [{\diff},\Vb + \epsilon V] \cdot \nabla_x V \Vert_{L^2} &\leq c\Vert [{\diff},\Vb]\nabla_x V \Vert_{L^2} + \epsilon c\Vert [{\diff},  V]\nabla_x V \Vert_{L^2} \\
&\leq c\Vert \Vb \Vert_{W^{k,\infty}} \Vert \nabla_x V \Vert_{H^{s-1,k-1}} + \epsilon c \Vert V \Vert_{H^{s,k}} \Vert \nabla_x V \Vert_{H^{s-1,k-1}} \\
&\leq C \tilde{M}.
\end{aligned} $$
We control the second term of $\Rdeux$ the exact same way, by $\Vert \gradphi[x] P \Vert_{H^{s-1}}$. Thus, by \eqref{eqn:elliptic_high_reg:perte}, we get
$$\Vert \gradphi[x] P \Vert_{H^{s-1,k-1}} \leq C\tilde{M}.$$ 
Note that the bound is uniform in $\mu$, and we get
$$ \Vert \Rdeux \Vert_{L^2} \leq C\tilde{M}.$$\underline{\bf Estimate for $\Rtrois$ and $\Rquatre$:}\\
The terms in $\Rtrois$ and $\Rquatre$ have the same form as the terms in $\Run$ or $\Rdeux$ and we omit them.\\
\underline{\bf Estimate for $\Rcinq$:}\\ The first term of $\Rcinq$ is of the same form as the first term of $\Rdeux$ and is treated the same way. The second term is controlled by the product estimate \eqref{apdx:pdt:tame}. We get 
$$\Vert \Rcinq \Vert_{L^2} \leq C\tilde{M}. $$
\underline{\bf Estimate for $\Rsix$:}\\ The first and third terms of $\Rsix$ are bounded by the product estimate \eqref{apdx:pdt:tame}. The second and fourth terms are controlled thanks to the estimates \eqref{apdx:alinhac:R2} and \eqref{apdx:alinhac:R1:Vbar}. We get:
$$\Vert \Rsix \Vert_{L^2} \leq C\tilde{M}.$$
\underline{\bf Estimate for $F$:}\\
The commutator estimate \eqref{apdx:commutator:infty} and Proposition \ref{lemma:elliptic_high_reg} allow us to bound the terms of $F$:
$$ \begin{aligned} \Vert [{\diff},\frac{1}{\varrho}]\partial_r^{\varphi} P\Vert_{L^2} + \Vert[{\diff},\frac{\varrho'}{\varrho}]\eta\Vert_{L^2} &\leq C \left( \Vert \partial_r^{\varphi} P \Vert_{H^{s-1,k-1}} +  \Vert \eta \Vert_{H^{s-1,k-1}} \right)\leq C \tilde{M}.
\end{aligned} $$
Note that as $\varrho$ and $\varrho'$ only depend on $r$, if $k=0$ then the commutators in $F$ are zero.
\end{proof}

\begin{remark}
\label{rk:summary_proof}
Before turning to the proof of the energy estimates, let us give an insight on the contributions of the different terms in the quasilinearized system \eqref{eqn:euler_isopycnal_quasilin}.{\newcommand{\diff}{\mathbb{\Lambda}} The proof yields a slightly more precise estimate than the one stated in Proposition \ref{lemma:energy}, namely
$$ \frac{d}{dt} \E(t) \leq C \left(1+\frac{\epsilon}{\sqrt{\mu}} + \frac{|\Vb'|_{L^{\infty}}}{\sqrt{\mu}} \right) \E(t).$$
This is done by testing the equations \eqref{eqn:euler_isopycnal_quasilin} against $(\varrho(1+\epsilon h)V^{\diff}, \varrho(1+\epsilon h) w^{\diff}, \varrho' \diff \eta )$, with $ \diff := \Lambda^{s-k} \partial_r^k$ with $s \geq s_0+\frac52$, $ k \geq 1$ or $\diff = |D|^2 \Lambda^{s-2}$, and estimating the subsequent terms. The terms on the left hand-side of \eqref{eqn:euler_isopycnal_quasilin} cancel each other out. The terms in $\Rdeux$, $\Rcinq$ are of order $1$.
The terms $\tilde{b}$ and $\Rtrois$ account for the factor $\frac{\epsilon}{\sqrt{\mu}}$ above, as they are multiplied by $w$ in the energy estimate, and the energy functional only controls $\sqrt{\mu} w$. \\
For the term $F$ in \eqref{eqn:euler_isopycnal_quasilin}, we then distinguish two cases. 
First, if we only differentiate the system with respect to the horizontal variable $x$ (see case 1 in the proof of Proposition \ref{lemma:energy}, with $\mathbb{\Lambda} = |D|^2\Lambda^{s-2}$), then $F = 0$, as $\varrho$ only depends on $r$. The pressure term is canceled by the use of the divergence-free condition, as well as the boundary condition on $w$. \\
Second, if we also differentiate the system with respect to the vertical variable $r$ (see the case 2 in the proof, with $\mathbb{\Lambda} = \Lambda^{s-k}\partial_r^k$), then we only need to control (derivatives of) $\partial_r w$, and not $w$ itself. Moreover, the divergence-free condition allows us to express $\partial_r w$ with respect to the other unknowns. Therefore, the commutators in the second equation of \eqref{eqn:euler_isopycnal_quasilin} are not multiplied by $w$ anymore, and thus are of order $1$. The same discussion holds for the pressure term and yields the contribution of size $\frac{|\Vb'|_{L^{\infty}}}{\sqrt{\mu}}$. A comparison with the result in \cite{Desjardins2019} is provided in remark \ref{rk:trick}. }
\end{remark}
\begin{proof}[Proof of Proposition \ref{lemma:energy}]
\newcommand{\diff}{\mathbb{\Lambda}}
Let $ 0 \leq k \leq s$ and $ \diff$ be either $|D|^2 \Lambda^{s-2}$ or $\Lambda^{s-k}\partial_r^k$ (for $k \geq 1$) . Recall that $V^{\diff}$ is Alinhac's good unknown associated with the operator ${\diff}$ defined in \eqref{apdx:alinhac:definition}. \\
According to Lemma \ref{lemma:quasilin}, $(V^{\diff},w^{\diff},\dot{\eta})$ (with $\dot{\eta} := {\diff}\eta$) satisfy \eqref{eqn:euler_isopycnal_quasilin}, with $P$ defined according to Proposition \ref{lemma:elliptic_high_reg}. \\

Testing the first equation against $\varrho (1+\epsilon h) V^{\diff}$, the second one against $\varrho (1+\epsilon h) w^{\diff}$, the third one against $\varrho' \dot{\eta}$ and summing the corresponding expressions, we get :
\begin{equation}
\label{eqn:energy}
\begin{aligned}
&\int_{S_r} \varrho (1+\epsilon h) \partial_t V^{\diff} \cdot V^{\diff} + \mu \int_{S_r} \varrho (1+\epsilon h) \partial_t w^{\diff} w^{\diff} +  \int_{S_r} \varrho' \partial_t \dot{\eta}\dot{\eta} + (i) + (ii) = (iii) + (iv) + (v) + (vi) + (vii), 
\end{aligned} 
\end{equation}
where
$$ \begin{aligned} 
(i) &:= \int_{S_r} (\Vb + \epsilon V) \cdot \nabla_x V^{\diff} \cdot V^{\diff} \varrho (1+ \epsilon h) + \mu \int_{S_r} (\Vb + \epsilon V) \cdot \nabla_x w^{\diff} w^{\diff} \varrho (1+ \epsilon h)  + \int_{S_r} \varrho' (\Vb + \epsilon V) \cdot \nabla_x \dot{\eta} \dot{\eta}, \\
(ii) &:= \int_{S_r} \frac{\varrho'}{\varrho} \dot{\eta} w^{\diff} \varrho (1+\epsilon h) - \int_{S_r} w^{\diff} \dot{\eta} \varrho',\\
(iii) &:= -\int_{S_r} \nabla_x^{\varphi} P^{\diff} \cdot V^{\diff} (1+\epsilon h) \nouveau{-}\int_{S_r} \partial_r^{\varphi} P^{\diff} \cdot w^{\diff} (1+\epsilon h),\\
(iv) &:= \epsilon \int_{S_r}\left( -\diff \tilde{b} w^{\diff} \varrho (1+\epsilon h)+ \Run V^{\diff} \varrho (1+\epsilon h) + \Rtrois w^{\diff} \varrho (1+\epsilon h)\right),\\
(v) &:= \mu \int_{S_r} \Rquatre w^{\diff} \varrho (1+\epsilon h),\\
(vi) &:=  \int_{S_r} \left(\Rdeux V^{\diff} \varrho (1+\epsilon h) + \Rcinq \cdot{\eta} \varrho' \right),\\
(vii) &:= \int_{S_r} F w^{\diff} \varrho (1+\epsilon h). \\
\end{aligned} $$

We start by studying the left-hand side of \eqref{eqn:energy}. Note that as $\eta$ is in $H^{s_0+\frac52,3}$, we can differentiate equation \eqref{eqn:euler_slag}[(3)] on the evolution of $\eta$ with respect to $r$ to obtain an equation on $h := - \partial_r \eta$:
$$\partial_t h + (\Vb + \epsilon V) \cdot \nabla_x h \nouveau{-} \nabla_x \eta \cdot \partial_r( \Vb + \epsilon V) \nouveau{+} \partial_r w = 0.$$
By the divergence-free condition, the two last terms are equal to $(1+\epsilon h) \nabla_x \cdot (\Vb + \epsilon V)$  so that we can write
\begin{equation}
\label{eqn:h}
\epsilon \partial_t h + \nabla_x \cdot ((1+\epsilon h) (\Vb + \epsilon V)) = 0.
\end{equation}
Therefore, one can integrate by parts horizontally in the integrals in $(i)$ and use \eqref{eqn:h} to get

$$ (i) = \frac12 \int_{S_r} \varrho \partial_t (1+\epsilon h) |V^{\diff}|^2 + \mu  \frac12 \int_{S_r} \varrho \partial_t (1+\epsilon h) |w^{\diff}|^2  + \frac12 \int_{S_r} \varrho'|\dot{\eta}|^2. $$
By the Leibniz rule, the left-hand side of \eqref{eqn:energy} becomes
$$\frac12 \frac{d}{dt} \left( \int_{S_r} \varrho (1+\epsilon h) |V^{\diff}|^2 + \mu  \frac12 \int_{S_r} \varrho (1+\epsilon h) |w^{\diff}|^2  + \frac12 \int_{S_r} \varrho' |\dot{\eta}|^2 \right) + (ii),$$
whose first term is a summand of $\E$. \\ 
We now examine the remaining terms in \eqref{eqn:energy}.
\begin{itemize}
\item[(ii)] Note that the leading order terms in $\epsilon$ of $(ii)$ cancel each other. Therefore, we can bound the remaining term by Cauchy-Schwarz inequality by
$$\epsilon C \Vert w^{\diff} \Vert_{L^2} \Vert \dot{\eta} \Vert_{L^2}. $$

\item[(iv)] We use Cauchy-Schwarz inequality on each term of $(iv)$, to get
$$\begin{aligned} 
(iv) &\leq \epsilon \Vert \diff \tilde{b}\Vert_{L^2} \Vert   \varrho (1+\epsilon h)w^{\diff} \Vert_{L^2} + \epsilon \Vert \Run \Vert_{L^2} \Vert \varrho (1+\epsilon h)V^{\diff} \Vert_{L^2} + \epsilon \Vert \Rtrois \Vert_{L^2}  \Vert  \varrho (1+\epsilon h)w^{\diff}\Vert_{L^2}.
\end{aligned}$$
Using Lemma \ref{lemma:quasilin} and $\epsilon \leq \sqrt{\mu}$, we can bound the terms in $(iv)$ by $C \E$.
\item[(v)] As for the terms in $(iv)$, we get 
$$ \mu \int_{S_r} \Rquatre w^{\diff} \varrho (1+\epsilon h) \leq C \E. $$

\item[(vi)] As for the terms in $(iv)$, we get 
$$\int_{S_r} \left(\Rdeux V^{\diff} \varrho (1+\epsilon h) + \Rcinq \cdot{\eta} \varrho' \right) \leq C \E.$$

\end{itemize}
We now turn to the terms in $(iii)$ and $(vii)$. For these terms, we need to distinguish two cases : \\
{\bf\underline{Case 1:} ${\diff} = |D|^2 \Lambda^{s-2}$} \\
In this case, we use an integration by parts to get rid of the pressure term. Note indeed that the impermeability condition yields : $|D|^2\Lambda^{s-2} w_{|r=0;1} = 0$ and $\eta$ is constant on the boundaries so $|D|^2\Lambda^{s-2} \eta_{|r=0;1} = 0$. Thus, $w^{\diff}$ has zero trace on the top and bottom. 
\begin{itemize}
\item[(iii)]
We integrate by parts in order to use the divergence-free condition. Namely, writing $\vec{U}^{\diff} := \begin{pmatrix} V^{\diff} \\ w^{\diff}  \end{pmatrix}$:
$$\begin{aligned} (iii) = \int_{S_r} \nabla^{\varphi} P^{\diff} \cdot \vec{U}^{\diff} (1+\epsilon h) &= \int_{r = 0; 1} P^{\diff} w^{\diff} - \int_{S_r} P^{\diff} \nabla^{\varphi} \cdot \vec{U}^{\diff}\\
&=  \frac{1}{\epsilon} \int_{S_r} P^{\diff} \nabla_x^{\varphi} \cdot \Vb^{\diff} - \int_{S_r} P^{\diff} \Rsix\\
&= -\frac{1}{\epsilon} \int_{S_r} \nabla_x^{\varphi}P^{\diff} \cdot \Vb^{\diff}  - \int_{S_r} P^{\diff} \Rsix.\\
\end{aligned}$$
For the first term on the right-hand side, we use Cauchy-Schwarz inequality. Recall the estimate of Proposition \ref{lemma:elliptic_high_reg}:
$$ \Vert \sqrt{\mu} \gradphi[x] P^{\diff} \Vert_{L^2} \leq C \left( \Vert V \Vert_{H^{s}} + \Vert \sqrt{\mu} w \Vert_{H^s} + \Vert \eta \Vert_{H^s} \right).$$
Recall also the expression $\Vb^{\diff} = \epsilon \frac{{\diff} \eta}{1+\epsilon h} \Vb'.$
This yields:
$$\frac{1}{\epsilon} \int_{S_r} \nabla_x^{\varphi}P^{\diff} \cdot \Vb^{\diff} \leq \frac{| \Vb'|_{L^{\infty}}}{\sqrt{\mu}}  C \left( \Vert V \Vert_{H^{s}} + \Vert \sqrt{\mu} w \Vert_{H^s} + \Vert \eta \Vert_{H^s} \right).$$
Then, we use the assumption $|\Vb'|_{L^{\infty}} \leq \sqrt{\mu}$ to bound this factor by $C$. \\
For the second term, we use Cauchy-Schwarz inequality. Recall the bound on $\Rsix$ provided in Lemma \ref{lemma:quasilin}. For the pressure, recall the expression \eqref{eqn:Ldeux} derived in the proof of Proposition \ref{lemma:elliptic_high_reg}:
$$\Vert P^{\diff} \Vert_{L^2} \leq \Vert \nabla P \Vert_{H^{s-1,0}} + \epsilon \Vert \eta \Vert_{H^{s,0}} \Vert \gradphi[\mu] P \Vert_{H^{s_0+\frac12,1}}.$$
Hence, using the estimate \eqref{eqn:elliptic_high_reg:perte} to control the first term independently of $\mu$ and the estimate \eqref{eqn:elliptic:high_reg} to control the second term, this yields
$$\Vert P^{\diff} \Vert_{L^2} \leq C \left( \Vert V \Vert_{H^{s}} + \Vert \sqrt{\mu} w \Vert_{H^s} + \Vert \eta \Vert_{H^s} \right).$$
We eventually get
$$(iii) \leq C \left( \Vert V \Vert_{H^s}^2  + \mu \Vert w \Vert_{H^s}^2 + \Vert \eta \Vert_{H^s}^2 \right.).$$ 
\item[(vii)] Recall that $F$ is the sum of two commutators of $|D|^2 \Lambda^{s-2}$ with functions that do not depend on $x$, and thus is zero in this case.
\end{itemize}
{\bf \underline{Case 2:} ${\diff} = \Lambda^{s-k} \partial_r^k$ for $k \geq 1$} \\
In this case, we use the fact that $k \geq 1$ to substitute $w^{\diff}$ by terms that depend only on $V$ and $\eta$, thanks to the incompressibility constraint. To do so, recall the definition 
$$w^{\diff} := \Lambda^{s-k} \partial_r^{k-1} \partial_r w + \frac{ \Lambda^{s-k} \partial_r^k \eta}{1+\epsilon h} \partial_r w.$$
The divergence-free condition yields
$$ \partial_r w = (1+\epsilon h) \nabla_x^{\varphi} \cdot V + \nabla_x \eta \cdot \Vb'. $$
Thus, applying $\Lambda^{s-k} \partial_r^{k-1}$ and taking the $L^2$-norm yields
$$\Vert  \Lambda^{s-k} \partial_r^{k-1} \partial_r w \Vert_{L^2}  \leq \Vert \Lambda^{s-k} \partial_r^{k-1}\left( (1+\epsilon h) \nabla_x^{\varphi} \cdot V\right) \Vert_{L^2} + \Vert  \Lambda^{s-k} \partial_r^{k-1}\left( \nabla_x \eta \cdot \Vb'\right) \Vert_{L^2} $$
We use product estimates, as $s-1 \geq s_0+\frac12$ :\\
If $k-1 \geq 1$, we use \cite[Lemma A.3]{Duchene2022} in the expression of $\gradphi[x] \cdot V$ (recall the definition \eqref{eqn:def:gradphi} of $\gradphi[x]$) to get
$$\begin{aligned} \Vert \Lambda^{s-k} \partial_r^{k-1}\left( (1+\epsilon h) \nabla_x^{\varphi} \cdot V\right) \Vert_{L^2} &\leq c \Vert \nabla_x \cdot V \Vert_{H^{s-1,k-1}} + c \Vert h \Vert_{H^{s-1,k-1}}\Vert \nabla_x \cdot V \Vert_{H^{s-1,k-1}}  + c \Vert \nabla_x \eta \Vert_{H^{s-1,k-1}}\Vert\partial_r V \Vert_{H^{s-1,k-1}}. \end{aligned} $$
If $k-1=0$, we use \eqref{apdx:pdt:tame} to get
$$\begin{aligned} \Vert \Lambda^{s-k} \partial_r^{k-1}\left( (1+\epsilon h) \nabla_x^{\varphi} \cdot V\right) \Vert_{L^2} \leq &c \Vert \nabla_x \cdot V \Vert_{H^{s-1,0}} + c \Vert h \Vert_{H^{s-1,0}}\Vert \nabla_x \cdot V \Vert_{H^{s_0+\frac12,1}}\\
&+ c \Vert h \Vert_{H^{s_0+\frac12,1}}\Vert \nabla_x \cdot V \Vert_{H^{s-1,0}}  + c \Vert \nabla_x \eta \Vert_{H^{s-1,0}}\Vert\partial_r V \Vert_{H^{s_0+\frac12,1}} \\
&+ c \Vert \nabla_x \eta \Vert_{H^{s_0+\frac12,1}}\Vert\partial_r V \Vert_{H^{s-1,0}}. \end{aligned}$$
In any case, as $s \geq s_0+\frac52 \geq 2$:
$$\Vert \Lambda^{s-k} \partial_r^{k-1}\left( (1+\epsilon h) \nabla_x^{\varphi} \cdot V\right) \Vert_{L^2} \leq C \left( \Vert V \Vert_{H^{s}} + \Vert \eta \Vert_{H^{s}} \right). $$
The second term $\nabla_x \eta \cdot \Vb'$ is controlled the same way and we get 
\begin{equation}
\label{eqn:w:temp1}
\Vert w^{\diff} \Vert \leq C \left( \Vert V \Vert_{H^{s}} + \Vert \eta \Vert_{H^{s}} \right).
\end{equation}
We can now bound the terms $(iii)$ and $(vii)$.
\begin{itemize}
\item[(iii)]
Writing $\vec{U}^{\diff} := \begin{pmatrix} V^{\diff} \\ w^{\diff}  \end{pmatrix}$, we bound each term using Cauchy-Schwarz inequality and the previous bound \eqref{eqn:w:temp1} on $\Vert w^{\diff} \Vert_{L^2}$, namely
$$\int_{S_r} \nabla^{\varphi} P^{\diff} \cdot \vec{U}^{\diff} (1+\epsilon h) \leq \Vert \nabla^{\varphi} P^{\diff} \Vert_{L^2} \left( \Vert V^{\diff} \Vert_{L^2} + \Vert w^{\diff} \Vert_{L^2} \right).$$
We use Proposition \ref{lemma:elliptic_high_reg} to control $\Vert \partial_r^{\varphi} P^{\diff} \Vert_{L^2}$. However, in order to have a control on $\Vert \gradphi[x] P^{\diff} \Vert_{L^2}$ that is uniform in $\mu$, we rather use \eqref{apdx:alinhac:comm1} together with \eqref{apdx:alinhac:R1} to write
$$\Vert \gradphi[x] P^{\diff} \Vert_{L^2} \leq C\Vert \diff \gradphi[x] P\Vert_{L^2}  + C \Vert \gradphi[x] P \Vert_{H^{s_0+\frac32,1}} + C \Vert \gradphi[x] P \Vert_{H^{s-1}}.$$
The last two terms are controlled with Equation \eqref{eqn:elliptic_high_reg:perte}. For the first one, remark that 
$$ \begin{aligned}\diff \gradphi[x] P &= \Lambda^{s-k}\partial_r^{k-1} \partial_r \gradphi[x]P \\
&= -\Lambda^{s-k}\partial_r^{k-1} (1+\epsilon h) \partial_r^{\varphi} \gradphi[x]P = -\Lambda^{s-k}\partial_r^{k-1} \left((1+\epsilon h) \gradphi[x] \partial_r^{\varphi} P\right).
\end{aligned}$$
Thus, we control the $L^2$ norm of this term with $ C\Vert \partial_r^{\varphi} P \Vert_{H^{s,k}} $ by the product estimate \cite[Lemma A.3]{Duchene2022}. \\
This eventually yields a control of $\Vert \gradphi[x] P^{\diff} \Vert_{L^2}$ that is uniform in $\mu$, and allows us to conclude:
$$\int_{S_r} \nabla^{\varphi} P^{\diff} \cdot \vec{U}^{\diff} (1+\epsilon h) \leq C \left( \Vert V^{\diff} \Vert_{L^2} + \Vert V \Vert_{H^{s}} + \sqrt{\mu}\Vert w \Vert_{H^{s}} +\Vert \eta \Vert_{H^{s}} \right)^2 $$
\item[(vii)] For the term $(vii)$, we use Cauchy-Schwarz inequality, Lemma \ref{lemma:quasilin} as well as the previous bound \eqref{eqn:w:temp1} on $\Vert w^{\diff} \Vert_{L^2}$, to get
$$ \int_{S_r} F w^{\diff} \leq C \left( \Vert V^{\diff} \Vert_{L^2} + \Vert V \Vert_{H^{s}} + \sqrt{\mu}\Vert w \Vert_{H^{s}} +\Vert \eta \Vert_{H^{s}} \right)^2 $$
\end{itemize}

Summing these estimates together yields 
$$\frac{d}{dt} \E \leq C\left( \E + \Vert V \Vert_{H^{s,k}}^2 + \mu \Vert w \Vert_{H^{s,k}}^2 + \Vert \eta \Vert_{H^{s,k}}^2 \right). $$
Lemma \ref{lemma:equiv_energy} allows us to conclude.
\end{proof}
\begin{remark}
\label{rk:nrj:triche}
\newcommand{\diff}{\mathbb{\Lambda}}
Once we apply a differential operator $\diff$ of the form $|D|^2 \Lambda^{s-2}$ or $\Lambda^{s-k}\partial_r^k$ for $k \leq s$ to the system \eqref{eqn:euler_slag}, we get the quasilinear system \eqref{eqn:euler_isopycnal_quasilin}. Note however that the terms $\nabla_x V^{\diff},\nabla_x w^{\diff},\nabla_x |D|^2 \Lambda^{s-2}\eta$ in \eqref{eqn:euler_isopycnal_quasilin} are distributions, as $V,w,\eta$ are in $H^{s,s}(S)$, but may not be in $H^{s+1,s}(S_r)$ (recall the definition \eqref{eqn:sobolev:anisotrope:def}). This can be solved by applying a smoothing (in x) operator, for instance $\chi(\delta \Lambda) \diff$ where $\chi$ is a smooth bump function and $\delta > 0$, instead of $\mathbb{\Lambda}$ in the proof of Proposition \ref{lemma:energy}. The same estimates hold, are uniform in $\delta$, and sending $\delta \to 0$ yields the estimate of Proposition \ref{lemma:energy}. This is detailed in the context of elliptic equations in \cite[Lemma 2.38]{Lannes2013} and we omit it.
\end{remark}
\begin{remark}
\label{rk:trick}
The proof yields a slightly more precise estimate than the one stated in Proposition \ref{lemma:energy}, namely
$$ \frac{d}{dt} \E(t) \leq C \left(1+\frac{\epsilon}{\sqrt{\mu}} + \frac{|\Vb'|_{L^{\infty}}}{\sqrt{\mu}} \right) \E(t).$$
The contributions of the different source terms of \eqref{eqn:euler_isopycnal_quasilin} in this estimate have already been summarized in Remark \ref{rk:summary_proof}. Let us now detail how the additional assumptions in \cite{Desjardins2019} allow the authors to obtain a time of existence of order $O\left(\left( \frac{\epsilon}{\sqrt{\mu}} \right)^{-1}\right)$.\\
The terms in $\Rdeux$, $\Rcinq$ are of order $1$. However, assuming $\Vb = 0$ and making the strong Boussinesq assumption (that is, the coefficient $\frac{1}{\varrho}$ in front of the pressure terms is a constant), as done in \cite{Desjardins2019}, sets the terms of order $1$ to zero and therefore $\Run$ and $\Rcinq$ are of order  $\epsilon$.\\
In the proof of Proposition \ref{lemma:energy}, we then distinguished two cases. In \cite{Desjardins2019}, the authors only use the strategy of the first case in the proof of Proposition \ref{lemma:energy}, that is the pressure term is canceled by the use of the divergence-free conditions. Hence, they need the trace of the vertical derivatives of $w$ to vanish at the boundary, which is not true in general, but is the purpose of the good preparation of the initial data that they assume. Furthermore, with the assumptions of the strong Boussinesq approximation and constant Brunt-Väisälä frequency, the commutators in the term $F$ are zero. The upside is that they do not have any term of order $1$, and no shear velocity at equilibrium, and thus obtain a time of existence of order $\frac{\sqrt{\mu}}{\epsilon}$. \\
\end{remark}
\begin{remark}
\label{rk:wellprepared}
Finally, let us state that the distinction of the two cases described above holds true in Eulerian coordinates, and is by no means specific to the isopycnal coordinates. Conversely, we believe that making the same assumptions as in \cite{Desjardins2019} would yield the same time of existence in isopycnal coordinates. However, the propagation of the good preparation assumption seems to require more technical computations than in Eulerian coordinates. The well-preparation of the initial data would then allow to use apply the strategy of the Case 1 in the proof of Proposition \ref{lemma:energy} for both the case $\mathbb{\Lambda} = |D|^2 \Lambda^{s-2}$ and  $\mathbb{\Lambda} = \Lambda^{s-k}\partial_r^k$, $k \geq 1$. More precisely, with this additional assumption, the integration by parts in the Case 1, $(iii)$ does not yield boundary terms.
\end{remark}
\section{Conclusion}
\label{section:conclusion}
We first state and prove Proposition \ref{prop:existence}, that is an existence result for PDEs of the form of \eqref{eqn:euler_slag}. Then we show how to apply this proposition to prove Theorem \ref{thm:wp}. Then, we state the result of Theorem \ref{thm:wp} in Eulerian coordinates (see Corollary \ref{cor:euleriennes}), for the sake of completeness.
\subsection{Completion of the proof of Theorem \ref{thm:wp}}
\label{subsection:proof_thm}
Note that \eqref{eqn:euler_euleriennes} is a quasilinear initial boundary value problem, and as such the existence of solutions is not obvious, see \cite{BenzoniGavage2006} for a quite complete introduction on these systems. However, \eqref{eqn:euler_slag} does not contain any vertical advection, and can therefore be viewed as \nouveau{a} family of initial value problems on $\R^d$ parametrized by $r \in [0,1]$. Still, the operator $\gradphi$ having coefficients that depend on the gradient of $\eta$, the linear system (that is, taking $\epsilon = 0$ in \eqref{eqn:euler_slag}) and the non-linear system \eqref{eqn:euler_slag} do not have the same structure, and a standard Picard iteration scheme relying on the energy estimates in Proposition \ref{lemma:energy} cannot be used directly. We thus choose to sketch the proof of Theorem \ref{thm:wp} in this subsection.  We first show the existence and uniqueness of solutions for the system \eqref{eqn:isopyc:gen} in Proposition \ref{prop:existence}. Note that we strongly use the semi-Lagrangian structure of the system \eqref{eqn:isopyc:gen}, see Remark \ref{rk:slag_existence}. Then, we apply Proposition \ref{prop:existence} to the Euler equations \eqref{eqn:euler_slag} to prove the existence of solutions on a time interval independent of the parameters $\epsilon$ and $\mu$, relying on the energy estimate of Proposition \ref{lemma:energy}. \\
Let $s_0 > d/2$, $s \geq s_0+\frac52$ an integer. In Proposition \ref{prop:existence}, we aim at constructing a solution in $C^0([0,T),H^s(\nouveau{S_r}))$ of the system
\begin{equation}
\label{eqn:isopyc:gen}
\sys{ \partial_t \vec{U} + (\Vb + V) \cdot \nabla_x \vec{U} + \frac{1}{\varrho} \gradphi P + \vec{b} &= 0, \\
\partial_t \eta + (\Vb + V) \cdot \nabla_x \eta - w &= 0, \\
\gradphi \cdot ( \Ub + \vec{U}) &= 0 ,}
\qquad \text{ in } [0,T) \times S_r;
\end{equation}
here, we write $\vec{U} := (V^T,w)^T$, $\Ub = (\Vb^T,0)^T$, $\varphi$ is defined from $\eta$ through \eqref{eqn:def:phi} and $\gradphi$ is defined from $\eta$ through \eqref{eqn:def:gradphi}. We assume that $\Vb$ and $\varrho$ satisfy \eqref{hyp:Mb}. We also assume that $\vec{b}$ is a source term that may depend on the solution $(\vec{U},\eta)$, satisfying
\begin{hyp}
\label{hyp:existence:b}
\sys{ \Vert \vec{b}(\vec{U},\eta) \Vert_{H^s} \leq C (1+\Vert (\vec{U},\eta) \Vert_{H^s}),\\
\Vert \vec{b}(\vec{U_1},\eta_1) - \vec{b}(\vec{U_2},\eta_2) \Vert_{H^s} \leq C \Vert (\vec{U_1} - \vec{U_2},\eta_1-\eta_2) \Vert_{H^s},}
\end{hyp}  with $C > 0$ a constant. We also assume that the traces $\eta_{|r=-1}$ and $\eta_{|r=0}$ of $\eta$ at the top and bottom are zero, i.e. that the top and bottom boundaries are flat, and the impermeability condition:
\begin{equation}
\label{hyp:zero_trace}
\sys{
w_{|r=0} &= w_{|r=1} = 0, \\
\eta_{|r=1} &= \eta_{|r=0} = 0.}
\end{equation}The factor $\varrho$ only depends on $r$ and satisfies
\begin{hyp}
\label{hyp:existence:rho}
c_* \leq \varrho \leq c^*.
\end{hyp}The system \eqref{eqn:isopyc:gen} is completed with the initial data $(\vec{U}_{\ini},\eta_{\ini})\in H^{s}(S)^{d+2}$ satisfying
\begin{hyp}
\label{hyp:existence:Min}
\Vert (\vec{U}_{\ini},\eta_{\ini}) \Vert_{H^s} \leq M_{\ini}
\end{hyp}for a constant $M_{\ini} > 0$, as well as the boundary conditions \eqref{hyp:zero_trace}, the divergence-free equation in \eqref{eqn:isopyc:gen} and
\begin{hyp}
\label{hyp:existence:hin}
c_* \leq 1+h_{\ini} \leq c^*,
\end{hyp}
with $h_{\ini} := - \partial_r \eta_{\ini}$.
\begin{proposition}
\label{prop:existence}
Let $s_0 > d/2$, $s \geq s_0+\frac52$, $(\Vb,\varrho) \in W^{s+1,\infty}([0,1])^{d+1}$ that satisfy \eqref{hyp:Mb}. Let $(\vec{U}_{\ini},\eta_{\ini})\in H^{s}(\nouveau{S_r})^{d+2}$ satisfy  \eqref{hyp:existence:Min},\eqref{hyp:existence:hin}, the boundary condition \eqref{hyp:zero_trace} and the divergence-free condition in \eqref{eqn:isopyc:gen}. Assume that the terms $\vec{b}$ and $\varrho$ satisfy \eqref{hyp:existence:b} and \eqref{hyp:existence:rho} respectively. Then there exist $T > 0$ that only depends on $\Mb, M_{\ini}, c_*, h_*, h^*, d,$ $s_0,s$ and a unique solution $(\vec{U},\eta) \in C^0([0,T),H^s(\nouveau{S_r})^{d+2})$ to \eqref{eqn:isopyc:gen} with initial data $(\vec{U}_{\ini},\eta_{\ini})$ and boundary conditions \eqref{hyp:zero_trace}. Moreover, if $T$ is finite, then 
\begin{equation}
\label{eqn:blowup}
\limsup_{t\to T}\Vert \vec{U}(t,\cdot) \Vert_{H^{s}} = + \infty.
\end{equation}
\end{proposition}
\begin{proof}[Proof of Proposition \ref{prop:existence}]
\newcommand{\J}{\mathcal{J}_{\delta}}
\newcommand{\diff}{\mathbb{\Lambda}}
In order to prove Proposition \ref{prop:existence}, we first write a regularized system, for which classical results on ordinary differential equations (ODE) apply. We then show how to define the pressure term for this system, and that it is Lipschitz continuous with respect to the unknowns $(\vec{U},\eta)$, in the Sobolev spaces $H^s$. Then, we use an energy estimate and a compactness argument to extract a solution of the original system in $L^{\infty}([0,T),H^{s}(S_r)^{d+2})$. We then show that this solution is unique, and derive the blow-up criterion \eqref{eqn:blowup}. Finally, we show that the solution is in $C^0([0,T),H^{s}(S_r)^{d+2})$ and satisfies the divergence-free condition in \eqref{eqn:isopyc:gen}.\\
Let $M_0,h_*,h^*$ be positive constants, $T> 0$ to be fixed later and for $(\vec{U}, \eta)\in C^0([0,T],(H^{s})^{d+2})$, we write the conditions
\begin{hyp}
\label{hyp:existence:M}
\sup_{t \in [0,T]}\Vert \vec{U}, \eta\Vert_{H^s}  \leq 2M_0,
\end{hyp}
and
\begin{hyp}
\label{hyp:existence:h}
\frac12 h_* \leq 1+ h \leq 2 h^*.
\end{hyp}
We will construct a solution in the Banach space
\begin{equation}
\label{eqn:existence:def:B}
\mathcal{B} := \{ (\vec{U}, \eta)\in C^0([0,T],(H^{s})^{d+2}), \text{ \eqref{hyp:existence:M} and \eqref{hyp:existence:h} hold true.}\}.
\end{equation}
{\bf \underline{Step 0:} Some technical tools} \saut
Let $\mathcal{J}_{\delta}$ for $\delta > 0$ be the mollification operator defined with the convolution with a smooth cutoff function in the horizontal Fourier space (see  \cite{AJMajda2002}), to regularize $\nabla_x$. It satisfies the following properties, see \cite[Lemma 3.5]{AJMajda2002}
\begin{lemma}
\label{lemma:moll}
Let $s,t \in \R$, $f \in H^t(\R^d)$. For $\delta > 0$, we can write
\begin{equation}
\label{eqn:moll:borne}
\vert \mathcal{J}_{\delta} f \vert_{H^s(\R^d)} \leq C_{\delta,s,t} \vert f \vert_{H^t(\R^d)},
\end{equation}
for some constant $ C_{\delta,s,t}$. For $\delta' > 0$, we also have 
\begin{equation}
\label{eqn:moll:lip}
\vert (\mathcal{J}_{\delta} - \mathcal{J}_{\delta'})f\vert_{H^{t-1}} \leq C |\delta - \delta'| \vert f \vert_{H^{t}},
\end{equation}
where $C$ does not depend on $\delta$ nor $\delta'$.
\end{lemma}We will also need a Leray projector defined as follows
\begin{definition}
We define
\begin{equation}
\label{eqn:def:pi}
\begin{aligned} \Pi : L^2(S_r) &\to L^2(S_r) \\
						\vec{U} &\mapsto \vec{U} - \gradphi \psi \end{aligned}
\end{equation}
where $\psi \in \dot{H}^1/\R$ is the solution to
\begin{equation}
\begin{aligned}
\gradphi \cdot \gradphi \psi &= \gradphi \cdot \vec{U},\\
\partial_{r}^{\varphi} \psi &= 0 \qquad \text{ at } r=-1 \text{ and } r=0. \end{aligned}
\end{equation}
\end{definition}
We state the main properties of such an operator.
\begin{lemma}
\label{lemma:pi}
Let $(\vec{U},\eta) \in L^2(S_r)^{d+2}$ such that \eqref{hyp:existence:M}, \eqref{hyp:existence:h} hold, and let $s_0 > \frac{d}{2}$.
\begin{itemize}
\item The field $\Pi \vec{U}$ is divergence-free:
		$$ \gradphi \cdot \Pi \vec{U} = 0.$$
\item For $s \geq s_0+\frac52$, The operator $\Pi$ is bounded from $H^s(S_r)$ to $H^s(S_r)$, and its norm only depends on $M_0,h_*,h^*$.
\item If $\vec{U} \in H^1(S_r)$, there exists $C > 0$ such that
	\begin{equation}
	\label{eqn:pi:divcurl}
	\Vert \gradphi (\Pi\vec{U} - \vec{U}) \Vert_{L^2} \leq C \Vert \gradphi \cdot \vec{U} \Vert_{L^2}. 
	\end{equation}
\end{itemize} 
\end{lemma}
\begin{proof}
The existence of $\psi$ in the definition \eqref{eqn:def:pi} of $\Pi$ is a direct consequence of Proposition \ref{lemma:elliptic:1}, and the second point is a direct consequence of Proposition \ref{lemma:elliptic:2}, given \eqref{hyp:existence:M} and \eqref{hyp:existence:h}. The first point follows from the definition of $\psi$ in the definition \eqref{eqn:def:pi} of $\Pi$. For the last point, we take $\tilde{\psi}$ such that $\tilde{\psi}(t,x,\nouveau{-}r+\eta(t,x,r)) = \psi(t,x,r)$. By change of variables, \eqref{eqn:pi:divcurl} boils down to
\begin{equation}
	\label{eqn:pi:divcurl:tilde}
	\Vert \nabla (\nabla \tilde{\psi}) \Vert_{L^2} \leq C \Vert \nabla \cdot \nabla \tilde{\psi} \Vert_{L^2}. 
\end{equation}
This last inequality is a div-curl estimate (see for instance \cite[Lemma 2.1]{Amrouche1998}), as $\nabla \tilde{\psi}$ is a curl-free vector field, with zero normal trace to the boundary of the strip, and the boundary is flat so its curvature is zero. 	
\end{proof}\noindent{\bf\underline{Step 1:} Mollification} \saut
Let $\mathcal{J}_{\delta}$ for $\delta > 0$ the mollification operator defined above. We then introduce the following mollified system
\begin{equation}
\tag{\ensuremath{E_{\delta}}}
\label{eqn:moll}
\sys{ \partial_t \vec{U}_{\delta} + \mathcal{J}_{\delta} \left[ \mathcal{J}_{\delta}(\Vb + V_{\delta}) \cdot \nabla_x \mathcal{J}_{\delta} \vec{U}_{\delta} \right]+ \frac{1}{\varrho} \gradphi P_{\delta}  + \vec{b} & = 0, \\
		\partial_t \eta_{\delta} + \mathcal{J}_{\delta}\left[ \mathcal{J}_{\delta}(\Vb + V) \cdot \nabla_x \mathcal{J}_{\delta}\eta_{\delta} \right] - w_{\delta} &= 0,}
\end{equation}
together with the initial conditions
$$ \sys{(\vec{U}_{\delta})_{|t=0} &= \vec{U}_{\ini},\\
		(\eta_{\delta})_{t=0} &= \eta_{\ini}. }$$
We define the pressure $P_{\delta}$ through the elliptic equation
\begin{equation}
\label{eqn:existence:elliptic:ordre2}
\gradphi \cdot \frac{1}{\varrho} \gradphi P_{\delta} = - \gradphi \cdot \left((\Ub + \vec{U}) \cdot \gradphi \Pi (\Ub + \vec{U}) \right) - \gradphi \cdot \vec{b},
\end{equation}
completed with the boundary conditions 
\begin{equation}
\label{eqn:existence:elliptic:bc}
\nouveau{\frac{1}{\varrho}} \partial_r^{\varphi} P_{\delta} = \nouveau{-}\vec{e_{d+1}} \cdot \left( (\Ub + \vec{U}) \cdot \gradphi  \Pi (\Ub + \vec{U}) + \vec{b}\right)  \qquad \text{ at } r = 0 \text{ and } r=1 .
\end{equation}
Note that because $\gradphi \cdot \Pi (\Ub + \vec{U}) =0$, \eqref{eqn:existence:elliptic:ordre2} also writes
\begin{equation}
\label{eqn:existence:elliptic:ordre1}
\gradphi \cdot \frac{1}{\varrho} \gradphi P_{\delta} = - \partial_i^{\varphi} (\Ub +\vec{U}_j)  \partial_j^{\varphi} (\Pi (\Ub + \vec{U}))_i - \gradphi \cdot \vec{b}.
\end{equation}
Using the boundary condition $w =\Pi w=0 $ at the boundaries, the boundary condition \eqref{eqn:existence:elliptic:bc} can be recast as
\begin{equation}
\label{eqn:existence:elliptic:bc:ordre1}
\nouveau{\frac{1}{\varrho}}\partial_r^{\varphi} P_{\delta} =  \nouveau{-}\vec{e_{d+1}} \cdot \vec{b} \qquad \text{ at } r = 0 \text{ and } r=1;
\end{equation}
note the abuse of notation, as here $\Pi w$ denotes the vertical component of $\Pi \vec{U}$.\\
{\bf \underline{Step 2:} $(\vec{U},\eta) \mapsto \gradphi P $ is well-defined and locally Lipschitz continuous} \saut
Recall that $s \geq s_0+\frac52$. As for the proof of Proposition \ref{lemma:elliptic_high_reg}, the elliptic equation \eqref{eqn:existence:elliptic:ordre2} together with \eqref{eqn:existence:elliptic:bc} defines uniquely $P$ in $\dot{H}^1/\R$ according to Proposition \ref{lemma:elliptic:1}. As for the proof of Proposition \ref{lemma:elliptic_high_reg}, Proposition \ref{lemma:elliptic:2} also yields
\begin{equation}
\label{eqn:existence:elliptic:borne}
\Vert \gradphi P \Vert_{H^s} \leq C \Vert (\vec{U},\eta)\Vert_{H^s},
\end{equation}
where $C$ only depends on an upper bound on $\Vert (\vec{U},\eta)\Vert_{H^s}$, $M_{\ini}, c_*,c^*$.\\
Let now $(\vec{U},\eta),(\tilde{\vec{U}},\tilde{\eta}) \nouveau{\in \mathcal{B}}$. Let $P$ and $\tilde{P}$ be the pressures defined by the equations \eqref{eqn:existence:elliptic:ordre2} and \eqref{eqn:existence:elliptic:bc} with respect to the solution $(\vec{U},\eta)$ and $(\tilde{\vec{U}},\tilde{\eta})$ respectively. Recall that $\varphi$ and $\tilde{\varphi}$ are defined through \eqref{eqn:def:phi} from $\eta$ and $\tilde{\eta}$ respectively. Applying Proposition \ref{lemma:elliptic:2:tilde} to the equation \eqref{eqn:existence:elliptic:ordre1} with boundary conditions \eqref{eqn:existence:elliptic:bc:ordre1}, we get (recall that $\mu = 1$ here):
\begin{equation}
\label{eqn:existence:elliptic:temp1}
\begin{aligned}
\Vert \nabla^{\varphi} P - \nabla^{\tilde{\varphi}} \tilde{P} \Vert_{H^{s}} \leq C &\left( \nouveau{\Vert \nabla^{\varphi} P - \nabla^{\tilde{\varphi}} \tilde{P}\Vert_{H^{s_0+\frac32,2}}} +  \Vert \partial_i^{\varphi} (\Ub + \vec{U})_j \partial_j^{\varphi} (\Ub + \vec{U})_i - \partial_i^{\tilde{\varphi}} (\Ub + \tilde{\vec{U}})_j \partial_j^{\tilde{\varphi}} (\Ub + \tilde{\vec{U}})_i \Vert_{H^{s-1}} \right. \\
&+ \left. \Vert \vec{b} - \tilde{\vec{b}}\Vert_{H^s} + \Vert \eta - \tilde{\eta} \Vert_{H^s}\right),
\end{aligned}
\end{equation}
where $\tilde{\vec{b}}$ is defined from $\tilde{\eta}$ through \eqref{eqn:def:b}, and where $C>0$ depends on $M_0$. Now note that in the second term of the right-hand side of \eqref{eqn:existence:elliptic:temp1}, the terms of the form $\partial_i^{\varphi} \Ub_j \partial_j^{\varphi} \Ub_i$ are zero, either because the vertical component of $\Ub$ is zero, or because $\Ub$ does not depend on $x$. Then, the third and fourth terms on the right-hand side of \eqref{eqn:existence:elliptic:temp1} are the difference of multilinear terms, and as such are readily bounded from above by terms involving the differences $\vec{U}-\tilde{\vec{U}}$ or $\eta-\tilde{\eta}$; more precisely we can write
\begin{equation}
\label{eqn:existence:elliptic:temp2}
\Vert \nabla^{\varphi} P - \nabla^{\tilde{\varphi}} \tilde{P} \Vert_{H^{s}} \leq C \left(\Vert\vec{U} - \tilde{\vec{U}} \Vert_{H^s} + \Vert \eta - \tilde{\eta} \Vert_{H^s}  \nouveau{+\Vert \nabla^{\varphi} P - \nabla^{\tilde{\varphi}} \tilde{P}\Vert_{H^{s_0+\frac32,2}}}\right).
\end{equation}
It now remains to bound from above the last term on the right-hand side of \eqref{eqn:existence:elliptic:temp2}. This is not a direct consequence of Proposition \ref{lemma:elliptic:1}, but rather an adaptation of its proof, so that we write a brief argument here. Let us start with an estimate on $\Vert \nabla^{\varphi} P - \nabla^{\tilde{\varphi}} \tilde{P}\Vert_{L^2}$. To this end, multiply \eqref{eqn:existence:elliptic:ordre2} satisfied by $P,\vec{U}$ and $\eta$ by $1+h$ (where $h := - \partial_r \eta$),  multiply \eqref{eqn:existence:elliptic:ordre2} satisfied by $\tilde{P}, \tilde{\vec{U}}$ and $\tilde{\eta}$ by $1+\tilde{h}$ (where $\tilde{h} := - \partial_r \tilde{\eta}$), and take the difference of both equations to get
\begin{equation}
\label{eqn:existence:elliptic:temp:3}
(1+h)\gradphi \cdot \frac{1}{\varrho} \gradphi (P - \tilde{P}) = -(1+h)\gradphi \cdot \frac{1}{\varrho} \gradphi \tilde{P} +(1+\tilde{h})\nabla^{\tilde{\varphi}} \cdot \frac{1}{\varrho} \nabla^{\tilde{\varphi}} \tilde{P} + (1+h)\gradphi \cdot \vec{G} - (1+\tilde{h}) \nabla^{\tilde{\varphi}} \cdot \tilde{\vec{G}},
\end{equation}
where 
$$\vec{G} := -(\Ub + \vec{U}) \cdot \gradphi \Pi (\Ub + \vec{U})  -  \vec{b},$$
and $\tilde{\vec{G}}$ is defined similarly from $\tilde{\vec{U}}, \tilde{\vec{b}}$ and $\tilde{\varphi}$. The elliptic equation \eqref{eqn:existence:elliptic:temp:3} is completed with the boundary conditions \eqref{eqn:existence:elliptic:bc} satisfied by both $P$ and $\tilde{P}$. Now we multiply \eqref{eqn:existence:elliptic:temp:3} by $P-\tilde{P}$, and integrate by parts using Lemma \ref{apdx:IPP}. Using the boundary conditions \eqref{eqn:existence:elliptic:bc} satisfied by both $P$ and $\tilde{P}$, the boundary terms cancel each other out, and using \eqref{hyp:existence:rho} and \eqref{hyp:existence:h}, we eventually get
\begin{equation}
\label{eqn:existence:elliptic:temp:4}
\begin{aligned}
\left\Vert \sqrt{\frac{1+h}{\varrho}} \gradphi (P - \tilde{P}) \right\Vert_{L^2}^2 &= -\int_{S_r} (1+h)\frac{1}{\varrho} \gradphi \tilde{P}  \cdot \gradphi (P-\tilde{P}) + \int_{S_r} (1+\tilde{h}) \frac{1}{\varrho} \nabla^{\tilde{\varphi}} \tilde{P} \cdot \nabla^{\tilde{\varphi}} (P-\tilde{P})\\
& + \int_{S_r} (1+h)\vec{G} \cdot \gradphi(P-\tilde{P}) - \int_{S_r} (1+\tilde{h}) \tilde{\vec{G}}\cdot \nabla^{\tilde{\varphi}} (P-\tilde{P}).
\end{aligned}
\end{equation}
Now making differences of the form $P-\tilde{P}$, $\eta - \tilde{\eta}$ and $\vec{G}-\tilde{\vec{G}}$ appear in \eqref{eqn:existence:elliptic:temp:4} and bounding the right-hand side from above by the triangular inequality as well as using Hölder's inequality, \eqref{hyp:existence:rho} and \eqref{hyp:existence:h}, we get
\begin{equation}
\label{eqn:existence:elliptic:temp:5}
\begin{aligned}
\Vert  \gradphi (P - \tilde{P}) \Vert_{L^2}^2 &\leq  C\Vert ((1+h)\nabla^{\varphi} -(1+\tilde{h})\nabla^{\tilde{\varphi}}) \tilde{P} \Vert_{L^2} \Vert \gradphi (P-\tilde{P})\Vert_{L^2} + C \Vert h - \tilde{h} \Vert_{L^{\infty}} \Vert \nabla^{\tilde{\varphi}} \tilde{P} \Vert_{L^2} \Vert \gradphi (P-\tilde{P})\Vert_{L^2}\\
&+ C\Vert \nabla^{\tilde{\varphi}} \tilde{P} \Vert_{L^2} \Vert ((1+h)\gradphi - (1+\tilde{h})\nabla^{\tilde{\varphi}}) (P-\tilde{P})\Vert_{L^2} + C\Vert \vec{G}- \tilde{\vec{G}} \Vert_{L^2} \Vert (1+h)\gradphi(P-\tilde{P})\Vert_{L^2} \\
&+ C\Vert \tilde{\vec{G}} \Vert_{L^2} \Vert ((1+h)\gradphi- (1+\tilde{h})\nabla^{\tilde{\varphi}})(P-\tilde{P})\Vert_{L^2}.
\end{aligned}
\end{equation}
Now using
$$ (1+h)\nabla^{\varphi} - (1+\tilde{h})\nabla^{\tilde{\varphi}} = (h-\tilde{h}) \begin{pmatrix} \nabla_x \\ 0 \end{pmatrix} + \begin{pmatrix}
\nabla_x (\eta - \tilde{\eta}) \\ 0 \end{pmatrix} \partial_r,$$
as well as the estimates \eqref{eqn:existence:elliptic:borne}, \eqref{hyp:existence:M}, \eqref{hyp:existence:h} we get that \eqref{eqn:existence:elliptic:temp:5} reduces to
\begin{equation}
\label{eqn:existence:elliptic:temp:6}
\Vert  \gradphi (P - \tilde{P}) \Vert_{L^2} \leq C \Vert (\vec{U} - \tilde{\vec{U}}, \eta-\tilde{\eta}) \Vert_{H^{s_0+\frac32,2}};
\end{equation}
where we also used \cite[Lemma A.1]{Duchene2022}, the definition of $\vec{G},\vec{\tilde{G}}$, Lemma \ref{lemma:pi}, and
$$ \Vert \gradphi(P-\tilde{P}) \Vert_{L^2} \leq c_1 \Vert \nabla(P-\tilde{P}) \Vert_{L^2} \leq c_2 \Vert \gradphi(P-\tilde{P}) \Vert_{L^2},$$
for some $c_1,c_2 > 0$, which stems from the product estimate \eqref{apdx:pdt:tame}. Note that $c_1$ and $c_2$ depend on $M_0$ as defined in \eqref{hyp:existence:M}. Using \eqref{eqn:existence:elliptic:temp:6}, we can conclude for the $L^2$ estimate
\begin{equation}
\label{eqn:existence:elliptic:temp:7}
\Vert  \gradphi P - \nabla^{\tilde{\varphi}}\tilde{P} \Vert_{L^2} \leq \Vert \gradphi(P-\tilde{P}) \Vert_{L^2} + \Vert (\gradphi - \nabla^{\tilde{\varphi}}) \tilde{P} \Vert_{L^2} \leq C \Vert (\vec{U} - \tilde{\vec{U}}, \eta-\tilde{\eta}) \Vert_{H^{s_0+\frac32,2}},
\end{equation}
after using \eqref{eqn:existence:elliptic:borne} and
$$\gradphi - \nabla^{\tilde{\varphi}} = \begin{pmatrix} \frac{\nabla_x (\eta - \tilde{\eta})}{1+h} + \nabla_x \tilde{\eta} \frac{\tilde{h} - h}{(1+h)(1+\tilde{h})} \\ \frac{-\partial_r (\eta - \tilde{\eta})}{1+h} - \partial_r \tilde{\eta} \frac{\tilde{h} - h}{(1+h)(1+\tilde{h})} \end{pmatrix} \partial_r,$$
which stems from the definition \eqref{eqn:def:gradphi} of $\gradphi$ and $\nabla^{\tilde{\varphi}}$. Now to obtain an estimate on $ \Vert  \gradphi P - \nabla^{\tilde{\varphi}}\tilde{P} \Vert_{H^{s_0+\frac32,2}}$, we use the interpolation from \cite[Theorem 5.2 (3)]{AdamsFournier2003}. We denote by $N := \lceil s_0+\frac32 \rceil$ the smallest integer that is greater than or equal to $s_0+\frac32$. We have $N \geq 2$ from $d \geq 1$ and the definition of $s_0$ from Proposition \ref{prop:existence}. As $s$ defined in the statement of Proposition \ref{prop:existence} is an integer that satisfies $s \geq s_0+\frac52$, we thus have $N+1\leq s$. Thus, from \cite[Theorem 5.2 (3)]{AdamsFournier2003}, there exists $\alpha \in (0,1)$ such that
\begin{equation}
\label{eqn:existence:elliptic:temp:8}
\Vert \gradphi P - \nabla^{\tilde{\varphi}} \tilde{P} \Vert_{H^{s_0+\frac32,2}} \leq \Vert \gradphi P - \nabla^{\tilde{\varphi}} \tilde{P} \Vert_{H^{N}} \leq C \Vert \gradphi P - \nabla^{\tilde{\varphi}} \tilde{P} \Vert_{L^2}^{\alpha} \Vert \gradphi P - \nabla^{\tilde{\varphi}} \tilde{P} \Vert_{H^{s}}^{1-\alpha}.
\end{equation}
Plugging \eqref{eqn:existence:elliptic:temp:8} into \eqref{eqn:existence:elliptic:temp2}, using \eqref{eqn:existence:elliptic:temp:7} and using Young's inequality, we get
\begin{equation}
\label{eqn:existence:elliptic:temp9}
\Vert \nabla^{\varphi} P - \nabla^{\tilde{\varphi}} \tilde{P} \Vert_{H^{s}} \leq C \left(\Vert\vec{U} - \tilde{\vec{U}} \Vert_{H^s} + \Vert \eta - \tilde{\eta} \Vert_{H^s} \right) + \frac{1}{2} \Vert \nabla^{\varphi} P - \nabla^{\tilde{\varphi}} \tilde{P}\Vert_{H^{s}},
\end{equation}
i.e., after putting the last term on the right-hand side of \eqref{eqn:existence:elliptic:temp9} on its left-hand side, the gradient of the pressure is locally Lipschitz continuous in $H^{s}$. \\ 
{\bf \underline{Step \nouveau{3}:} Existence and uniqueness} \saut
We now follow the steps of \cite{AJMajda2002} for the existence. As the system \eqref{eqn:moll} is a system of ODEs with a locally Lipschitz continuous source term according to the regularizing property in Lemma \ref{lemma:moll} and Step 2, Cauchy-Lipschitz theorem yields a solution $(\vec{U},\eta)_{\delta}$ on a time $T_{\delta} > 0$ for any $\delta > 0$, in the Banach space $\mathcal{B}$ defined in \eqref{eqn:existence:def:B}.
We use a continuity argument to show that the times of existence $T_{\delta}$ are bounded from below by some $T > 0$ independent of $\delta$.\\
We start by writing the following classical energy estimate (see for example \cite[Proposition 3.7]{AJMajda2002}), where $\gradphi P + \vec{b}$ is treated as a source term:
\begin{equation}
\label{eqn:existence:nrj}
\frac{d}{dt} \Vert (\vec{U}_{\delta}(t,\cdot),\eta_{\delta}(t,\cdot)) \Vert_{H^s}^2 \leq C \Vert (\vec{U}_{\delta}(t,\cdot),\eta_{\delta}(t,\cdot)) \Vert_{H^s}^2,
\end{equation}
where $C > 0$ is a constant depending only on $M_0,h_*,h^*$. Suppose that $T_{\delta}$ is finite, for a certain $\delta > 0$.
Then, Grönwall lemma allows us to bound $\Vert (\vec{U}_{\delta}(t,\cdot),\eta_{\delta}(t,\cdot)) \Vert_{H^s}^2$ by $e^{Ct}M_{0}$. Taking $T < \ln(2 M_0) \frac{1}{C}$, we have that:
$$\Vert (\vec{U}_{\delta}(t,\cdot),\eta_{\delta}(t,\cdot)) \Vert_{H^s}^2 < 2 M_0.$$
Recall the equation on $h_{\delta}$, obtained by differentiating the equation on $\eta_{\delta}$ in \eqref{eqn:moll}:
\begin{equation}
\label{eqn:existence:h:equation}
\partial_t h_{\delta}  = \partial_r \left( \J(\J (\Vb + V_{\delta}) \cdot \nabla_x \J \eta_{\delta}) - w \right).
\end{equation}
Integrating this equation in time allows us to bound $1+h_{\delta}$ from above and from below by a continuous function of $\Vert(\vec{U}_{\delta},\eta_{\delta}) \Vert_{H^{s_0+\frac52}}$, namely
\begin{equation}
\label{eqn:h:bound}
1-t\Vert (\vec{U}_{\delta},\eta_{\delta}) \Vert_{L^{\infty}([0,t],H^{s_0+\frac52})}\leq 1+h_{\delta}(t,x,r) \leq 1+t\Vert (\vec{U}_{\delta},\eta_{\delta}) \Vert_{L^{\infty}([0,t],H^{s_0+\frac52})}.
\end{equation}
Therefore, for $T$ small enough, we still have $\frac12 h_* \leq 1+h(t,x,r) \leq 2 h^*$ for $t < T$ and all $(x,r)\in S_r$. Hence we can use Cauchy-Lipschitz theorem once again to extend the previous solution up to a time $T$, independent of $\delta$. \\
The $L^2$ energy estimate \eqref{eqn:existence:nrj} yields that $((\vec{U}_{\delta},\eta_{\delta}))_{\delta}$ is a Cauchy sequence in $C^0([0,T],L^2(S_r))$ with convergence rate $\delta$. Thus, by interpolation (see \cite[Theorem 5.2 (3)]{AdamsFournier2003}), it is also a Cauchy sequence in $H^{s-1}$. The convergence is thus strong in $H^{s_0+\frac32,2}$, so that we may pass to the limit in every non-linear term in \eqref{eqn:isopyc:gen} and the limit $(\vec{U},\eta)$ is indeed a solution of \eqref{eqn:isopyc:gen}. Finally, as the sequence  $((V_{\delta},w_{\delta},\eta_{\delta}))$ is bounded in $H^{s}$, it has a weak limit by Banach-Alaoglu's theorem, and this weak-limit coincides with the strong limit in $H^{s_0+\frac32,2}$. We finally get the existence of a solution of \eqref{eqn:isopyc:gen} in $C^0([0,T],H^{s-1}) \cap L^{\infty}([0,T],H^s)$. \\
Uniqueness follows from the $L^2$ energy estimate \eqref{eqn:existence:nrj} performed on \eqref{eqn:isopyc:gen} directly (or, equivalently, on \eqref{eqn:moll} with $\delta = 0$).\\
{\bf \underline{Step \nouveau{4}:} Continuity in time of the solution} \saut
Following very closely \cite[Theorem 3.5]{AJMajda2002}, we can prove the following 
$$\Vert (\vec{U}(t,\cdot) - \vec{U}(t_0,\cdot), \eta(t,\cdot) - \eta(t_0,\cdot)) \Vert_{H^s} \underset{t \to t_0}{\rightarrow} 0,$$
for any $t_0 \in [0,T)$. The proof does not need any adaptation and we omit it.\\
{\bf \underline{Step \nouveau{5}:} Blow-up criterion} \saut
Let us assume that $(\vec{U},\eta)$ is a bounded solution of \eqref{eqn:isopyc:gen} in $C^0([0,T],H^s)$ for some $T > 0$, satisfying the assumptions \eqref{hyp:existence:M} and \eqref{hyp:existence:h}. Then by using the solution of \eqref{eqn:isopyc:gen}	constructed in Step 3 starting from $(\vec{U}(T),\eta(T))$ at initial time $T$, as well as its uniqueness, we may extend $(\vec{U},\eta)$ on $[0,T']$ for some $T' > T$. Therefore, if $T^*$ is the maximal time of existence of this solution, then either \eqref{eqn:blowup} holds or
\begin{equation}
\label{eqn:blowup:eta} 
\limsup_{t\to T} \Vert \eta(t,\cdot)\Vert_{H^s} = + \infty,
\end{equation}
or
\begin{equation}
\label{eqn:blowup:h} 
\liminf_{t\to T} \inf_{(x,r)\in S_r} 1+h = 0.
\end{equation}
We now show that the case \eqref{eqn:blowup} happens in any case. Standard energy estimates on the equation on $\eta$ in \eqref{eqn:isopyc:gen} yield
\begin{equation}
\label{eqn:eta:nrj}
\frac{d}{dt} \Vert \eta \Vert_{H^s} \leq C (\Vert \eta \Vert_{H^s} + \Vert \vec{U} \Vert_{H^s} ),
\end{equation}
where $C$ only depends on an upper bound on $\Vert \vec{U} \Vert_{H^s}$. Thus, by Grönwall Lemma, if \eqref{eqn:blowup:eta} holds true, then so does \eqref{eqn:blowup}.\\
Assume now that \eqref{eqn:blowup:h} holds for some finite $T > 0$. Applying $- \partial_r$ and using the divergence-free condition in \eqref{eqn:isopyc:gen} yields
\begin{equation}
\label{eqn:existence:h}
\partial_t (1+h) + \nabla_x \cdot ((1+h)(\Vb + V)) = 0.
\end{equation}
We can now solve this equation explicitly, let $t \in [0,T),(x,r)\in S_r$. Let $x \in C^1([0,t])$ be the solution of the ODE 
$$ x'(s) =( \Vb(r) + V(s,x(s),r)),$$
with terminal condition $x(t)= x$, and call $x_0 := x(0)$. Then, one can check that 
$$1+h(t,x,r) = (1+h_{\ini}(x_0,r)) e^{-\int_0^t \nabla_x \cdot (\Vb(r) + V(s,x(s),r))ds}.$$
Hence if \eqref{eqn:blowup:h} holds but the initial data $h_{\ini}$ satisfies \eqref{hyp:existence:h}, then \eqref{eqn:blowup} must hold.\\
{\bf \underline{Step \nouveau{6}:} Propagation of the divergence-free condition} \saut
Using the equation on $\eta$ in \eqref{eqn:isopyc:gen}, we can write the equation on $\vec{U}$ in \eqref{eqn:isopyc:gen} as
$$\partial_t^{\varphi} (\Ub + \vec{U}) + (\Ub + \vec{U}) \cdot \gradphi (\Ub + \vec{U}) + \gradphi P + \vec{b} =0,$$
as $\Ub$ does not depend on time. Applying $\gradphi \cdot$ to this equation and using the equation \eqref{eqn:existence:elliptic:ordre1} satisfied by the pressure, we get
\begin{equation}
\label{eqn:existence:div}
\partial_t^{\varphi} \gradphi \cdot(\Ub + \vec{U}) + \partial_i^{\varphi} (\Ub + \vec{U})_j \partial_j^{\varphi}(\Pi (\Ub + \vec{U}) -  (\Ub + \vec{U}))_i + (\Ub + \vec{U}) \cdot \gradphi \left( \gradphi \cdot (\Ub + \vec{U}) )\right) = 0.
\end{equation}
Testing this equation against $(1+h)\gradphi \cdot (\Ub + \vec{U})$, we get
\begin{equation}
\label{eqn:existence:div:nrj}
\frac{d}{dt} \Vert \gradphi \cdot (\Ub + \vec{U}) \Vert_{L^2} \leq C \left( \Vert \gradphi \cdot (\Ub + \vec{U}) \Vert_{L^2} + \Vert \gradphi\left(\Pi (\Ub + \vec{U}) - (\Ub + \vec{U}) \right) \Vert_{L^2} \right),
\end{equation}
where $C$ depends on $\Mb,M_{\ini}, h_*,h^*$. Using the estimate \eqref{eqn:pi:divcurl} and Grönwall Lemma, we get
$$\Vert \gradphi \cdot (\Ub + \vec{U})(t,\cdot)\Vert_{L^2} \leq \Vert \gradphi \cdot (\Ub + \vec{U})(0,\cdot)\Vert_{L^2} e^{Ct}.$$
Hence the divergence-free condition in \eqref{eqn:isopyc:gen} is propagated by the equation if it holds initially.
\end{proof}
\begin{remark}
\label{rk:slag_existence}
In Step $3$, we crucially used that the only unbounded operator in \eqref{eqn:euler_slag} is $\nabla_x$, which we can easily mollify using Fourier transform as $x \in \R^d$. In Eulerian coordinates, note that there is a term $w \partial_z V$, see for instance \eqref{eqn:euler_euleriennes}. Thus, as we do not have any boundary condition on $V$, such a regularization is not possible. Therefore, the proof in Eulerian coordinates would require other ingredients. 
\end{remark}
\begin{remark}
\label{rk:flat_bd}
In Step 6 of the proof of Proposition \ref{prop:existence}, we used the div-curl estimate \eqref{eqn:pi:divcurl}. With a non-flat bottom or rigid lid, this estimate would contain boundary terms on the right-hand side involving the trace of $\vec{U}$, and therefore the bootstrap argument in Step 6 would not hold true. Hence, the proof in the case of non-flat boundaries would require adaptations.
\end{remark}
\begin{remark}
\label{rk:free-surf}
This strategy can be adapted to the free-surface case, however, we leave this adaptation for a future work.
\end{remark}
To conclude the proof of Theorem \ref{thm:wp}, we apply Proposition \ref{prop:existence} with 
$$\begin{aligned}
\vec{U} = \begin{pmatrix}
\epsilon V \\ \sqrt{\mu}\epsilon w 
\end{pmatrix}, \qquad 
\eta = \epsilon \eta, \qquad
\vec{b} = (0,b)^T,
\end{aligned}$$
where $b$ is defined through \eqref{eqn:def:b}, and with $\Ub = (\Vb^T,0)^T$. Therefore, there exists a unique solution of \eqref{eqn:euler_slag} with initial conditions $V_{\ini},w_{\ini},\eta_{\ini}$ satisfying the assumptions in Theorem \ref{thm:wp}. The blow-up criterion \eqref{eqn:blowup} together with the estimate of Proposition \ref{lemma:energy} and the assumptions $\epsilon\leq \sqrt{\mu}$,  $|\Vb'|_{L^{\infty}} \leq \sqrt{\mu}$ yield a time of existence independent of $\epsilon$ and $\mu$.
\subsection{Well-posedness in Eulerian coordinates}
\label{subsection:eulerian}
We now state the result of Theorem \ref{thm:wp} in Eulerian coordinates, for the sake of completeness. Let us first detail the assumptions necessary for Corollary \ref{cor:euleriennes} to hold. \\
 Let $ s \in \N$. First, let $\Mb > 0$. We assume that $\Vb, \rhob \in W^{s+1,\infty}$ and
\begin{hyp}
\label{hyp:Mb:euleriennes}
| \eul{\Vb} |_{W^{s+1,\infty}} + |\rhob|_{W^{s+1,\infty}} \leq \Mb.
\end{hyp}Second, let $M_{\ini} > 0$. If $(\eul{V}_{\ini},\eul{w}_{\ini},\rho_{\ini}) \in H^{s}(S_z)^{d+2}$ correspond to initial data, then 
\begin{hyp}
\label{hyp:Min:euleriennes}
 \Vert \eul{V}_{\ini} \Vert_{H^{s}}^2 + \mu \Vert \eul{w}_{\ini} \Vert_{H^{s}}^2 + \Vert \rho_{\ini} \Vert_{H^{s}}^2 \leq M_{\ini}.
\end{hyp}Moreover, the initial data satisfy the boundary conditions and divergence-free condition :
\begin{hyp}
\label{hyp:wp:euleriennes}
\begin{aligned}  \left. \eul{w}_{\ini} \right._{|z=0} = \left. \eul{w}_{\ini} \right._{|z=-1} &= 0, \\
\rho_{\ini|z=0} = \rho_{\ini|z=-1} &= 0,\\
					 \nabla_x \cdot \eul{V}_{\ini}  +  \partial_z \eul{w}_{\ini} &= 0. \end{aligned} 
\end{hyp}The stratification is stable, both at equilibrium and with the initial perturbation, i.e. there exists $c_* > 0$ such that:
\begin{hyp}
\label{hyp:stable:euleriennes}
\begin{aligned}
-\rhob' &\geq c_*,\\
\nouveau{-\partial_r( \rhob + \epsilon \rho_{\ini})} &\nouveau{\geq c_*.}
\end{aligned}
\end{hyp}We also assume the non-cavitation assumption on the initial data,
\begin{hyp}
\label{hyp:non_cavitation_eul}
\rhob + \epsilon \rho_{\ini} \geq \rho_* > 0.
\end{hyp}We now state our main result in Eulerian coordinates.
\begin{corollary}
\label{cor:euleriennes}
Let $d \in \N^*$, $s_0>\frac{d}{2}$, $s \in \N$ with $s \geq s_0+\frac52$. Let $\eul{\Vb}, \rhob \in W^{s+1,\infty}$ satisfying \eqref{hyp:Mb:euleriennes}, and $(\eul{V}_{\ini},\eul{w}_{\ini},\rho_{\ini}) \in H^{s}(\R^d\times[-1,0])^{d+2}$ satisfying \eqref{hyp:Min:euleriennes}, \eqref{hyp:wp:euleriennes}, \eqref{hyp:stable:euleriennes}, \eqref{hyp:non_cavitation_eul}. We also make the assumptions
	$$\sys{ \epsilon &\leq \sqrt{\mu}, \\ |\Vb'|_{L^\infty} &\leq \sqrt{\mu}. }$$
Then there exists $T > 0$ depending only on $\Mb, M_{\ini}$ such that the following holds. There exists a unique solution $(\eul{V},\eul{w},\nouveau{\rho}) \in C^0([0,T),H^{s}(S_z)^{d+2})$ to \eqref{eqn:euler_euleriennes:pert} with initial data $(\eul{V}_{\ini},\eul{w}_{\ini},\rho_{\ini})$.
\end{corollary}
\begin{proof}
Let us define the quantities involved in Theorem \ref{thm:wp}. Let $(x,r) \in S_r$ and
$$\begin{aligned}
\etab(r) &:= -r, \qquad & 
\varrho(r) &:= \rhob(\etab(r)), \\
\Vb &:= \eul{\Vb}(\etab(r)), \qquad &
\epsilon \eta_{\ini}(x,r) &:= \left( (\rhob + \epsilon \rho_{\ini})^{-1}(x,\nouveau{\varrho(r)}) - \etab(r) \right), \\
V_{\ini}(x,r) &:= \eul{V}_{\ini}(\etab(r) + \epsilon \eta_{\ini}(x,r)),  \qquad & 
V _{\ini}^{\eff}(x,r) &:=  V _{\ini}(x,\nouveau{\etab(r) + \epsilon \eta(t,x,r)}) + \frac{1}{\epsilon} \left( \eul{\Vb}(\etab(r) + \epsilon \eta_{\ini}(x,r)) - \Vb(r) \right), \\
w_{\ini}(x,r) &:= \eul{w}_{\ini}(\etab(r)+ \epsilon \eta_{\ini}(x,r)). \\
\end{aligned}$$ 
Note the second term in $V _{\ini}^{\eff}$ that, in spite of the factor $\frac{1}{\epsilon}$, is of order $1$ (by Taylor expansion, or composition estimates). Because of this term, we can write
$$ (\Vb + \epsilon V _{\ini}^{\eff})(x,r) = (\eul{\Vb} + \epsilon \eul{V}_{\ini})(x,\etab(r)+\epsilon \eta_{\ini}(x,r)), $$
and $\Vb$ is independent of $x$. \\
Theorem \ref{thm:wp} then claims that there exists a unique solution $(V ^{\eff}, w, \eta)$ to \eqref{eqn:euler_slag}, related to the quantities at equilibrium and initial conditions above. We can then define 
$$\sys{ \epsilon \rho(t,x,z) &:= \nouveau{\varrho}(\nouveau{(\etab + \epsilon \eta)^{-1}}(t,x,z)) - \rhob(z), \\
\eul{V}(t,x,z) &:= V ^{\eff}(t,x,\nouveau{(\etab + \epsilon \eta)^{-1}}(t,x,z)) -  \frac{1}{\epsilon} \left( \eul{\Vb}(z) - \Vb(\nouveau{(\etab + \epsilon \eta)^{-1}}(t,x,z)) \right), \\
\eul{w}(t,x,z) &:= w(t,x,\nouveau{(\etab + \epsilon \eta)^{-1}}(t,x,z)). 
}$$
Note that the definition of $V^{\eff}$ is such that $\eul{\Vb} + \epsilon \eul{V} = (\Vb + \epsilon V ^{\eff})(t,x,\nouveau{(\etab + \epsilon \eta)^{-1}(t,x,r)})$. \\
Changing variables in \eqref{eqn:euler_phi} shows that $(\eul{V},\eul{w},\rho)$ is indeed a solution to \eqref{eqn:euler_euleriennes} with the desired boundary conditions and initial data. \\
For the uniqueness, note that two solutions of \eqref{eqn:euler_euleriennes} with the same initial data yield, via change of variables, two solutions of \eqref{eqn:euler_slag}. Thus, according to Theorem \ref{thm:wp}, the solutions are the same. 
\end{proof}
\appendix
\section{Product, commutator and composition estimates}
\label{appendixA}
In this appendix we gather technical results in Sobolev spaces, such as product, commutator and composition estimates. Most of these results are similar to results from \cite{Lannes2013} and \cite{Duchene2022}.
\begin{lemma}[Integration by parts in isopycnal coordinates]
\label{apdx:IPP}
Let $s_0 > d/2$, \nouveau{$\eta \in H^{s_0+\frac52,3}(S_r)$ satisfying \eqref{hyp:h} and the last two boundary conditions in \eqref{eqn:euler_slag:bc},} and $h := - \partial_r \eta$. Let $\varphi$ be defined from $\eta$ through \eqref{eqn:def:phi}. Let $\vec{f} \in H^{1}(S_r)^{d+1}$ and $g$ are in $H^{1}(S_r)$, then\\
$$\int_{S_r} (1+\epsilon h) g \nabla^{\varphi} \cdot \vec{f} \ dx dr= \int_{r=0} \vec{f}\cdot\vec{e_{d+1}} g dx \nouveau{-} \int_{r=1} \vec{f}\cdot\vec{e_{d+1}} gdx - \int_{S_r} (1+\epsilon h) \vec{f} \cdot  \nabla^{\varphi} g dx dr,$$
where $\vec{e_{d+1}}$ is the vertical upward vector, i.e. the normal vector to the bottom and rigid lid, pointing upward.
\end{lemma}
\begin{proof}
The regularity of $\eta$ together with the classical Sobolev embedding and \cite[Lemma A.1]{Duchene2022} yields that $\varphi$ is a $\mathcal{C}^1$ diffeomorphism. The proof then consists in going back in Eulerian coordinates to use the standard integration by parts formula, and then applying the isopycnal change of coordinates again. Note that $(1+\epsilon h)$ is the Jacobian of the isopycnal change of coordinates.
\end{proof}



\begin{lemma}[Product estimates in anisotropic spaces, \cite{Duchene2022}]
\label{apdx:pdt}
Let $d\in \N^*$, $s_0 > d/2$, $s \in \R$ and $k \in \N$ with $0 \leq k \leq s$.
\begin{itemize}[label=\textbullet]
\item There exists $C > 0$ such that for any $f \in L^2(S_r)$ and $g \in H^{s_0+\frac12,1}(S_r)$,  we can write
\begin{equation}
\label{apdx:pdt:tame}
\Vert fg\Vert_{L^2} \leq C \Vert f \Vert_{L^2} \Vert g \Vert_{H^{s_0+1/2,1}}.
\end{equation}
\item If \nouveau{$f \in W^{k,\infty}([0,1])$ and $g \in H^{s,k}(S_r)$}, then 
\begin{equation}
\label{apdx:pdt:infty}
 \Vert fg \Vert_{H^{s,k}} \leq C | f |_{W^{k,\infty}} \Vert g \Vert_{H^{s,k}}.
 \end{equation}
\end{itemize}
\end{lemma}
\begin{proof}
Both estimates are standard; their proof is a minor adaptation from \cite[Lemma A.3]{Duchene2022}.
\end{proof}
We now state commutator estimates in $H^{s,k}$ spaces. Note the use of $H^{s-1,k\wedge(s-1)}$-norms, so that in the case $k=s$, we gain one derivative. 
\begin{lemma}[Commutator estimates] 
\label{apdx:commutator}
Let $d\in \N^*$, $s_0 > d/2$ and $s >0$.
Let $k \in \N$ with $0 \leq k \leq s$.
 \begin{itemize}[label=\textbullet]
\item Assume $k \geq 2$, and let $l\in \N$ with $l \leq k$, then there exists $C > 0$ such that for any $f \in H^{s\vee(s_0+\frac32),k}$ and $g \in H^{(s-1)\vee(s_0+\frac12),k\wedge(s-1)}$, we can write 
\begin{equation}
\label{apdx:commutator:gen}
\Vert [\Lambda^{s-l} \partial_r^l,f]g \Vert_{L^2} \leq C \Vert f \Vert_{H^{s\vee(s_0+\frac32),k}} \Vert g \Vert_{H^{(s-1)\vee(s_0+\frac12),k\wedge(s-1)}}.
\end{equation}

\item If $s \geq  1$, then there exists $C > 0$ such that for any $f \in H^{s,0} \cap H^{s_0+\frac32,1}$, $g \in H^{s-1,0}\cap H^{s_0+\frac12,1}$, we can write 
\begin{equation}
\label{apdx:commutator:horizontal}
\Vert [|D|^2 \Lambda^{s-2},f]g \Vert_{L^2} \leq C \Vert f \Vert_{H^{s_0+\frac32,1}} \Vert g \Vert_{H^{s-1,0}}  + C \Vert f \Vert_{H^{s,0}} \Vert g \Vert_{H^{s_0+\frac12,1}}.
\end{equation}

\item Let $k \geq 1$. There exists $C > 0$ such that if $f \in \nouveau{W^{k,\infty}([0,1])}$ and $g \in H^{s-1,k-1}$, then
\begin{equation}
\label{apdx:commutator:infty}
 \Vert [\Lambda^{s-k} \partial_r^k,f]g \Vert_{L^2} \leq C | f |_{W^{k,\infty}} \Vert g \Vert_{H^{s-1,k-1}}. 
 \end{equation}

\end{itemize}
\end{lemma}
\begin{proof}
For estimate \eqref{apdx:commutator:gen}, we first write
\begin{equation}
\label{eqn:comm:temp}
[\Lambda^{s-k}\partial_r^l,f]g = \Lambda^{s-l}([\partial_r^l,f]g) + [\Lambda^{s-k},f]\partial_r^l g.
\end{equation}
For the first term in \eqref{eqn:comm:temp}, we use the product estimate \cite[Lemma A.3]{Duchene2022} and standard commutator estimates with partial derivatives (see for instance \nouveau{\cite[Lemma A.8]{Duchene2022}}). For the second term in \eqref{eqn:comm:temp}, we use the commutator estimate \cite[Theorem 5 (i)]{Lannes2006} and then the product estimate \eqref{apdx:pdt:tame}.\\
For \eqref{apdx:commutator:horizontal}, we use \cite[Theorem 5 (i)]{Lannes2006} and the embedding from \cite[Lemma A.1]{Duchene2022}.
For inequality \eqref{apdx:commutator:infty}, we start by noticing that as $f$ only depends on $r$, it commutes with $\Lambda^{s-k}$. Then the standard the commutator estimate between derivatives and functions in $L^2$ yields the result.
\end{proof}
\begin{lemma}[Symmetric commutator estimates]
\label{apdx:commutator_sym}
Let $d\in \N^*$, $s_0 > d/2$, $s \in \R$  with $s \geq s_0+\frac52$ and $k \in \N$ with $2 \leq k \leq s$.

\begin{itemize}[label=\textbullet]
\item 
If $l\in \N$ with $0 \leq l \leq k$. Then there exists $C > 0$ such that for any $f \in H^{s,k}$, $g \in H^{s-1,k\wedge(s-1)}$, we can write 
\begin{equation}
\label{apdx:commutator_sym:gen}
\Vert [\Lambda^{s-l} \partial_r^l;f,g] \Vert_{L^2} \leq C \Vert f \Vert_{H^{s-1,k\wedge(s-1)}} \Vert g \Vert_{H^{s-1,k\wedge(s-1)}}.
\end{equation}
\item 
There exists $C > 0$ such that for $f,g \in H^{s-1,0}\cap H^{s_0+\frac32,1}$ we can write 
\begin{equation}
\label{apdx:commutator_sym:horizontal}
\Vert [|D|^2\Lambda^{s-2};f,g] \Vert_{L^2} \leq C \Vert f \Vert_{H^{s_0+\frac32,1}} \Vert g \Vert_{H^{s-1,0}} + C \Vert f \Vert_{H^{s-1,0}} \Vert g \Vert_{H^{s_0+\frac32,1}}.
\end{equation}
\end{itemize}
\end{lemma}

\begin{proof}
\newcommand{\diff}{\mathbb{\Lambda}}
Let $\mathbb{\Lambda}:=|D|^2 \Lambda^{s-2}$ or $\Lambda^s$. Then, as for the proof of Lemma \ref{apdx:commutator}, \eqref{apdx:commutator_sym:gen} and \eqref{apdx:commutator_sym:horizontal} both follow easily from
\begin{equation}
\label{apdx:comm:symm:temp}
\vert [\diff;f,g] \vert_{L^2(\R^d)} \leq C \vert f \vert_{W^{1,\infty}}\vert g\vert_{H^{s-1}} + C \vert f \vert_{H^{s-1}}\vert g\vert_{W^{1,\infty}},
\end{equation}
which is a small adaptation of the proof of \cite[Proposition B.10 (2)]{Lannes2013} (with $\delta = 1$); we highlight the necessary adaptations for the sake of completeness. Note \cite[Remark B.11]{Lannes2013} which extends \cite[Proposition B.10]{Lannes2013} to Fourier multipliers of smooth symbols of degree $s$, as is the case for $\diff$. \\
The proof relies on paradifferential calculus. Although the present study is not the place 
to develop the whole theory, let us write the fundamental decomposition
\begin{equation}
\label{apdx:comm:symm:temp2}
fg = T_f g + T_g f + R(f,g),
\end{equation}
with $T_f,T_g$ defined in \cite[Proof of Proposition B.10, Step 1]{Lannes2013}, as well as $R(f,g)$ which is a remainder term.  Using \eqref{apdx:comm:symm:temp2}, we can write a paralinearization formula for the symmetric commutator, analogous to \cite[(B.14)]{Lannes2013}:
\begin{equation}
\label{apdx:comm:symm:temp3}
[\diff;f,g] = [\diff,T_f]g + [\diff,T_g]f + (\diff T_f) g - T_{\diff f} g + (\diff T_g)f - T_{\diff g} f + \diff R(f,g) - R(\diff f,g) - R(f,\diff g).
\end{equation}
The crucial difference is that the second term in the right-hand side of \cite[(B.14)]{Lannes2013} is absent from \eqref{apdx:comm:symm:temp3}, precisely because the commutator here is a symmetric one. Then, the first two terms in \eqref{apdx:comm:symm:temp3} are bounded from above thanks to \cite[(B.15) and Step 8]{Lannes2013}; the other ones by \cite[(B.18), (B.19), (B.20)]{Lannes2013} respectively. This yields \eqref{apdx:comm:symm:temp}.\\
\end{proof}

\section{Definitions and results on Alinhac's good unknown}
\label{appendixB}
The system \eqref{eqn:euler_slag} of the Euler equations in isopycnal coordinates mainly differs from the one in Eulerian coordinates by the operators $\nabla^{\varphi}$, which arise from the chain rule when differentiating a composition of the unknowns with the change of variables. In other words, the operator $\nabla^{\varphi}$ is the gradient in Eulerian coordinates. Recall that the change of variables $\varphi$ is 
\begin{equation}
\label{eqn:varphi:B}
\varphi(t,x,r) = (t,x,-r + \epsilon \eta(t,x,r)).
\end{equation}
Recalling the notation $h := - \partial_r \eta$, we write its Jacobian matrix $\partial \varphi$, as defined in \eqref{eqn:def:jacphi}. Recall the definition \eqref{eqn:def:gradphi} of $\gradphi,\partial_t^{\varphi},\partial_r^{\varphi}$. The non-linearity of these operators raises an issue when one wants to differentiate the system, say apply an operator $\mathbb{\Lambda}$ of the form $\Lambda^{s-k}\partial_r^k$ with $s > 0$, $ 0 \leq k \leq s$, as $\mathbb{\Lambda}$ and $\nabla^{\varphi}$ do not commute, and the commutator involves higher derivatives in $\eta$ than what the energy method allows us to control. This is detailed in the following lemma.
\begin{lemma}
\newcommand{\diff}{\mathbb{\Lambda}}

Let \nouveau{$s_0 > d/2$}, $s \in \R$ with $s > s_0 + \frac32$, $0 \leq k \leq s$, with $k \in \N$. Assume that $\eta \in H^{s +1 ,2\vee k +1}$ and there exist $h_*,h^*$ such that $0<h_* \leq 1+\epsilon h \leq h^*$ with $h := - \partial_r \eta$. Then there exists $C>0$ such that the following holds. We write   $ \diff$ an operator of the form $\Lambda^{s-k} \partial_r^k$ or $|D|^2 \Lambda^{s-2}$ (with $k=0$ by convention). Let $f \in H^{s-1,k\wedge(s-1)\vee2}$. Then
\begin{equation}
 \label{apdx:alinhac:comm_brutal}
 \Vert [\diff, \gradphi] f \Vert_{L^2} \leq\epsilon C \Vert \eta \Vert_{H^{s\vee(s_0+\frac32) +1 ,2\vee k +1}} \Vert \nabla f \Vert_{H^{s-1,k\wedge(s-1)\vee2}}.
 \end{equation}
\end{lemma}
\begin{proof}
\newcommand{\diff}{\mathbb{\Lambda}}
Let   $ \diff$ an operator of the form $\Lambda^{s-k} \partial_r^k$ or $|D|^2 \Lambda^{s-2}$. \\
We only treat $\gradphi[x]$, and $\partial_r^{\varphi}$ is treated the same way. Recall the expression \eqref{eqn:def:gradphi} of $\gradphi[x]$. As $\diff$ and $\nabla_x$ commute, we can write
$$ \Vert [\diff,\gradphi[x]]f \Vert_{L^2} = \epsilon \Vert [\diff, \frac{1}{1+\epsilon h} \nabla_x \eta ] f \Vert_{L^2}.$$
Recall indeed that $h = - \partial_r \eta$. Thus the commutator estimates \eqref{apdx:commutator:gen} and \nouveau{\eqref{apdx:commutator:horizontal}} yield
$$ \Vert [\diff,\gradphi[x]]f \Vert_{L^2} \leq C \epsilon \left\Vert \frac{1}{1+\epsilon h} \nabla_x \eta \right\Vert_{H^{s\vee(s_0+\frac32) ,2\vee k}} \left\Vert f \right\Vert_{H^{s-1,k\wedge(s-1)\vee2}}.$$
Then, the product estimate \cite[Lemma A.3]{Duchene2022} and the composition estimate \cite[Lemma A.5]{Duchene2022} yield the desired estimate.
\end{proof}
This issue is solved by using Alinhac's good unknowns, introduced in \cite{Alinhac1989}. See also \cite[Section 2.3]{MasmoudiRousset2012} for its use for the free-surface Navier-Stokes equation, and \cite{Lannes2013} for its use for the water waves equations, as well as \cite[Proposition 3.17]{Castro2014a} where the importance of the relation \eqref{apdx:alinhac:comm1} is emphasized. Recall the notation \eqref{eqn:varphi:B}. \nouveau{In what follows we make the assumptions, for $\eta \in H^{s_0+\frac32,2}(S_r)$,
\begin{hyp}
\label{hyp:h:bc:apdx:B}
\eta_{|r=0} = \eta_{|r=1} = 0,\qquad
h_* \leq 1+\epsilon h \leq h^*,
\end{hyp}
for some constants $0<h_*\leq h^*$, with $h := - \partial_r \eta$, as well as
\begin{hyp}
\label{hyp:M:eta:apdxB}
\Vert \nabla_{t,x,r} \eta \Vert_{H^{s-1,k\wedge(s-1)}} \leq M,
\end{hyp}
for some constant $M > 0$ and $s,k$ that will be defined upon the use of this assumption.}
\begin{lemma}[Alinhac's good unknown, \cite{Alinhac1989}]
\label{apdx:alinhac:def}
\newcommand{\diff}{\mathbb{\Lambda}}

Let $T > 0$, $s_0> d/2$, $s \in \R$ with $s \geq s_0 + \frac52$, $k,l \in \N $ with $0 \leq l \leq k \leq s$ and $k \geq 2$. Assume that $\eta \in C^0([0,T],H^{s_0+\nouveau{\frac52,3}})$ satisfies \eqref{hyp:h:bc:apdx:B} and \eqref{hyp:M:eta:apdxB} for some $M > 0$. Then there exists $C>0$ depending only on $M$ such that the following holds. We write $\diff$ an operator of the form $\Lambda^{s-l} \partial_r^l$ (with $l \geq 1$) or $|D|^2 \Lambda^{s-2}$ (with $l=0$ by convention).
\begin{itemize}[label=\textbullet]
\item Let $f \in C^0([0,T],\nouveau{\dot{H}^1(S_r)/\R})$ with \nouveau{$\nabla_{t,x,r} f \in H^{s-1,k\wedge(s-1)}$}. We have 
\begin{equation}
\label{apdx:alinhac:comm1}
\diff  \gradphi[t,x,r] f = \gradphi[t,x,r] f^{\diff}  + \epsilon ({\diff} \eta \partial_r^{\varphi})\gradphi[t,x,r]f + \epsilon \Ral{1}{\nabla_{t,x,r}f},
\end{equation}
where   $\epsilon  \Ral{1}{\nabla_{t,x,r}f} := \nouveau{-}(\partial \varphi)^{-T} [{\diff} ; (\partial \varphi)^T, \gradphi[t,x,r] f]$, and we can write
\begin{equation}
\label{apdx:alinhac:R1}
\Vert \Ral{1}{\nabla_{t,x,r}f} \Vert_{L^2} \leq C \Vert \nabla_{t,x,r}\eta \Vert_{H^{s-1,k\wedge(s-1)}} \Vert \gradphi[t,x,r] f\Vert_{H^{s-1,k\wedge(s-1)}} .
\end{equation}
Here, Alinhac's good unknown is defined as
\begin{equation}
\label{apdx:alinhac:definition}
f^{\diff} := \diff f + \epsilon \frac{\diff \eta}{1+\epsilon h} \partial_r f.
\end{equation}
\item For  \nouveau{$\vec{f} \in (\dot{H}^1(S_r)/\R)^{d+1}$ with $\nabla \vec{f} \in \left(H^{s-1,k\wedge (s-1)}\right)^{(d+1)^2}$} a vector-valued function of  $d+1$ components, we can write
\begin{equation}
\label{apdx:alinhac:comm2}
\diff \gradphi \cdot \vec{f} = \gradphi \cdot \vec{f}^{\diff} + \epsilon \diff \eta \partial_r^{\varphi}(\gradphi \cdot \vec{f}) + \epsilon \Ral{2}{\nabla  \vec{f}},
\end{equation}
where
$$\Ral{2}{\nabla  \vec{f}} := \sum\limits_{i \in \{1, \dots,d+1\}} \Ral{1}{\nabla_{x,r} \vec{f}_i} \cdot \vec{e_i}.$$
Moreover,
\begin{equation}
\label{apdx:alinhac:R2}
\Vert \Ral{2}{\nabla  \vec{f}} \Vert_{L^2} \leq C \epsilon \Vert \nabla_{x,r} \eta \Vert_{H^{s-1,k\wedge(s-1) }} \Vert \gradphi \vec{f} \Vert_{H^{s-1,k\wedge(s-1)}}.
\end{equation}
Here, Alinhac's good unknown is defined as in \eqref{apdx:alinhac:definition} component by component.
\end{itemize}
\end{lemma}
\begin{proof}
\newcommand{\diff}{\mathbb{\Lambda}}
We prove the first point of Lemma \ref{apdx:alinhac:def}, and the proof of the second one can be deduced from the first one and the definition 
$$\gradphi \cdot f  = \sum\limits_{i \in \{ 1,\dots,d+1\}} \partial_i^{\varphi} f _i.$$
We write $ \diff$ an operator of the form $\Lambda^{s-k} \partial_r^k$ or $|D|^2 \Lambda^{s-2}$ (with $k=0$ by convention). In this proof we write $\nabla := \nabla_{t,x,r}$ (resp. $\gradphi$), for the sake of conciseness. We can write 
$$\gradphi f  = (\partial \varphi)^{-T} \nabla f .$$
Multiplying both sides by $(\partial \varphi)^{T}$ and applying $\diff$ yields 
$$ (\partial \varphi)^T \diff  \gradphi f + (\diff(\partial \varphi)^T) \gradphi f + [\diff;(\partial \varphi)^T,\gradphi f] = \nabla \diff f.$$
Applying $(\partial \varphi)^{-T}$ to both sides yields
$$\diff  \gradphi f  =  - (\partial \varphi)^{-T} \diff  (\partial \varphi)^{T} \gradphi  f   + \gradphi \diff  f  + \epsilon \Ral{1}{\nabla  f },$$
which we write in such a way that reveals Alinhac's good unknown 
\begin{equation}
\label{eqn:alinhac:good_gen}
\diff  \gradphi f  =  \gradphi f ^{\diff} + \gradphi (\diff  \varphi \cdot \gradphi  f ) - (\partial \varphi)^{-T} \diff  (\partial \varphi)^{T} \gradphi  f + \epsilon  \Ral{1}{\nabla  f }.
\end{equation}
Here, we wrote 
$$ f ^{\diff} := \diff  f  - (\diff  \varphi \cdot \gradphi )f .$$
Now, we need to compute the sum of the second and third terms. To this end, recall the notation $\partial^{\varphi} \varphi := (\partial_i^{\varphi} \varphi_j ) _{i,j \in \{ 1,\dots,d+1 \} }$. We then write 
$$\begin{aligned}
(\partial \varphi)^T (\partial^{\varphi} \diff   \varphi)^T &= (\partial \varphi)^T (\partial^{\varphi}_i \diff  \varphi_j)_{(i,j)\in \{1,\dots,d+1\}^2} = (\partial \varphi)^T \nabla^{\varphi} \begin{pmatrix} \diff  \varphi_1 & \dots & \diff  \varphi_{d+1} \end{pmatrix} \\
&= (\partial \varphi)^T (\partial \varphi)^{-T}\nabla \begin{pmatrix} \diff  \varphi_1 & \dots & \diff  \varphi_{d+1} \end{pmatrix} = (\partial \diff  \varphi)^T = (\diff  \partial \varphi)^T.
\end{aligned}$$
Moreover,
$$ \begin{aligned} 
(\partial^{\varphi} \diff  \varphi)^T \gradphi f  &= (\partial^{\varphi}_i \diff  \varphi_j)_{(i,j)\in \{1,\dots,d+1\}^2} \gradphi  f  = \gradphi ((\diff  \varphi \cdot \gradphi )f ) - \diff  \varphi \cdot \gradphi (\gradphi f ).
\end{aligned} $$ 
Thus we get
$$\gradphi  (\diff  \varphi \cdot \gradphi  f ) - (\partial \varphi)^{-T} \diff  (\partial \varphi)^{T} \gradphi f  = \diff  \varphi \cdot \gradphi (\gradphi f ),$$
which yields 
\begin{equation}
\label{eqn:alinhac:comm2:gen}
\diff \gradphi f  = \gradphi \left( \diff f - \diff \varphi \cdot \gradphi f  \right) + \left(\diff \varphi \cdot \gradphi \right) \gradphi f  + \epsilon \Ral{1}{\nabla  f }.
\end{equation}
Note that \eqref{eqn:alinhac:comm2:gen} is a generalization of \eqref{apdx:alinhac:comm1} as we did not use the particular form \eqref{eqn:varphi:B} of $\varphi$.\\
 
Thus, let us finally discuss the particular form of Alinhac's good unknown stated in \eqref{apdx:alinhac:definition}. In what follows, we write 
$$ \varphi_0(t,x,r) := (t,x,-r)$$
so that 
\begin{equation}
\label{eqn:varphi}
 \varphi = \varphi_0 + \epsilon (0,0,\eta).
 \end{equation}
Note that plugging \eqref{eqn:varphi} into \eqref{eqn:alinhac:good_gen} does not yield \eqref{apdx:alinhac:definition}, due to the fact that $\diff \varphi_0 \neq 0$. However, simple computations show that $\nabla_x \diff \varphi_0 = 0$ and $\partial_r \diff \varphi_0 = 0$, so that we can write
\begin{equation}
\label{eqn:alinhac:magie}
\gradphi \left(\diff \varphi_0 \cdot \gradphi f  \right) = \left(\diff \varphi_0 \cdot \gradphi \right) \gradphi f .
\end{equation}
Subtracting \eqref{eqn:alinhac:magie} from \eqref{eqn:alinhac:comm2:gen} yields
$$\diff \gradphi f  = \gradphi\left( \diff f  + \diff (\varphi - \varphi_0) \cdot \gradphi f \right) + \diff \left(\varphi - \varphi_0 \right) \cdot \gradphi(\gradphi  f ) + \Ral{1}{\nabla  f },$$
which is exactly \eqref{apdx:alinhac:comm1}, as $\varphi - \varphi_0 = \epsilon (0,0,\eta)$.\\

Finally, for the estimate on $\Ral{1}{\nabla f}$, the key point is to notice that $(\partial \varphi)^T = I_{d+1,1} + \epsilon \jac^T$, where
\begin{equation}
\label{eqn:def:J}
I_{d+1,1} := \begin{pmatrix} I_{d+1} & 0 \\ 0 & -1 \end{pmatrix}, \qquad \jac := \begin{pmatrix}
0 & 0 & \partial_t \eta \\ 0 & 0&\nabla_x \eta \\ 0 & 0 & - h
\end{pmatrix}.
\end{equation}
As the symmetric commutator $[\diff ; \cdot , \cdot \cdot]$ is bi-linear and $[\diff ;I_{d+1,1},\gradphi f] = 0$, we get 
$$ [\diff ;(\partial  \varphi)^T,\gradphi f] = \epsilon [\diff ;\jac , \gradphi f].$$
The result comes from the estimates \eqref{apdx:commutator_sym:gen} and \eqref{apdx:commutator_sym:horizontal} on symmetric commutators. Note that we used the definition of $\gradphi$, the product estimate \cite[Lemma A.3]{Duchene2022} as well as \cite[Lemma A.6]{Duchene2022} to write
 $$\Vert \nabla f  \Vert_{H^{s-1,k \wedge(s-1)}}\leq C (1+\epsilon \Vert \nabla \eta \Vert_{H^{s-1,k\wedge(s-1)}}) \Vert \gradphi f  \Vert_{H^{s-1,k\wedge(s-1)}}.$$
\end{proof}

Note that using the commutator estimate \eqref{apdx:commutator_sym:horizontal} instead of \eqref{apdx:commutator_sym:gen} in the proof for \eqref{apdx:alinhac:R1}, we have a more precise estimate of $\Ral{1}{\nabla f}$ in the case $\mathbb{\Lambda} = |D|^2\Lambda^{s-2}$. We state the result in the following lemma.

\begin{lemma}
\newcommand{\diff}{\mathbb{\Lambda}}
Let $s \in \R$ with $s \geq 1$. Assume that $\eta \in H^{s_0+\frac52,3}$ satisfies \eqref{hyp:h:bc:apdx:B} and \eqref{hyp:M:eta:apdxB} (with $k=0$). Then there exists $C > 0$ such that the following holds, for \nouveau{$f \in C^0([0,T],\dot{H}^1(S_r)/\R)$ with $\nabla_{t,x,r} f \in H^{s-1,0}\cap H^{s_0+\frac32,1}$} and $ \diff = |D|^2 \Lambda^{s-2}$
\begin{equation}
\label{apdx:alinhac:R1:un}
\Vert \Ral{1}{\nabla_{t,x,r}f} \Vert_{L^2} \leq C \Vert \nabla_{t,x,r}\eta \Vert_{H^{s-1,0}} \Vert \gradphi[t,x,r] f\Vert_{H^{s_0+\frac32,1}} + C \Vert \eta\Vert_{H^{s_0+\frac52,2}}\Vert \gradphi[t,x,r] f\Vert_{H^{s-1,0}}.
\end{equation}
\end{lemma}
Note that Lemma \ref{apdx:alinhac:def} and \eqref{apdx:alinhac:R1:un} are still valid when replacing $\gradphi$ by
$$  \gradphi[\mu] := \begin{pmatrix} \sqrt{\mu} \gradphi[x] \\ \partial_r^{\varphi} \end{pmatrix}, $$ 
by applying $J_{\mu} := \begin{pmatrix} \sqrt{\mu} I_d \\ 1 \end{pmatrix}$ to  \eqref{apdx:alinhac:comm1}, \eqref{apdx:alinhac:R1}, \eqref{apdx:alinhac:comm2}, \eqref{apdx:alinhac:R2}, \eqref{apdx:alinhac:R1:un}.\\

\nouveau{In order to treat the second term in the incompressibility equation in \eqref{eqn:euler_slag}, we write a slight variation of Lemma \ref{apdx:alinhac:def}.
\begin{lemma}
\label{apdx:lemma:alinhac:def:Vbar}
\newcommand{\diff}{\mathbb{\Lambda}}
Let $s_0> d/2$, $s \in \R$ with $s \geq s_0 + \frac52$, $k,l \in \N $ with $0 \leq l \leq k \leq s$ and $k \geq 2$. Assume that $\eta \in H^{s_0+\frac52,3}$ satisfies \eqref{hyp:h:bc:apdx:B} and \eqref{hyp:M:eta:apdxB} for some $M > 0$. Then there exists $C>0$ depending only on $M$ such that the following holds. We write $\diff$ an operator of the form $\Lambda^{s-l} \partial_r^l$ (with $l \geq 1$) or $|D|^2 \Lambda^{s-2}$ (with $l=0$ by convention).
Let $\Vb \in W^{s,\infty}([0,1])$. Then \eqref{apdx:alinhac:comm2} holds with $f=\Vb$, and 
\begin{equation}
\label{apdx:alinhac:R1:Vbar}
\Vert \Ral{2}{\nabla \Vb} \Vert_{L^2} \leq C \Vert \nabla  \eta \Vert_{H^{s-1,k\wedge(s-1)}} \vert \Vb \vert_{W^{k,\infty}}.
\end{equation}
\end{lemma}
\begin{proof}
\newcommand{\diff}{\mathbb{\Lambda}}
The proof of \eqref{apdx:alinhac:comm2} still holds. For the proof of \eqref{apdx:alinhac:R1:Vbar}, note that if $l=0$, $\Ral{2}{\nabla \Vb} = 0$, as $\Vb$ only depends on $r$. Otherwise, we use the definition of $\Ral{2}{\nabla_{t,x,r}\Vb}$ and the fact that $\Vb$ only depends on $r$ to write
\begin{equation*}
\label{apdx:alinhac:comm1:Vbar:temp1}
\epsilon \Ral{2}{\nabla \Vb } := \sum\limits_{i \in \{1, \dots,d\}} \epsilon\Ral{1}{\partial_r \Vb_i} \cdot \vec{e_i},
\end{equation*}
with $\Vb_i$ being the $i^{th}$ component of $\Vb$. We then use the definition of $\Ral{1}{\partial_r \Vb_i}$ and of the symmetric commutator to write
\begin{equation}
\label{apdx:alinhac:comm1:Vbar:temp2}
\epsilon \Ral{2}{\nabla \Vb } := - \sum\limits_{i \in \{1, \dots,d\}} (\partial \varphi)^{-T}\left( [\partial_r^l; \Lambda^{s-l}\epsilon J, - \vec{e_{d+1}} \Vb_i'] + [\diff; (\partial \varphi)^T, \frac{\epsilon}{1+\epsilon h}\begin{pmatrix} \nabla_x \eta \\ h \end{pmatrix}\Vb'_i\right) \cdot \vec{e_i},
\end{equation}
where for the first commutator we also used \eqref{eqn:def:J} and $[\partial_r^l;\Lambda^{s-l} I_{d,1},- \vec{e_{d+1}} \Vb_i'] = 0$. Thus, both terms on the right-hand side of \eqref{apdx:alinhac:comm1:Vbar:temp2} are indeed of size $\epsilon$. For the first symmetric commutator, we can use standard symmetric commutator estimates, stemming from Leibniz' rule (see for instance \cite[Lemma A.9]{Duchene2022}). For the second symmetric commutator, as $\nabla \eta \in H^{s-1,k\wedge(s-1)}$, we can use directly \eqref{apdx:commutator_sym:gen}. This yields \eqref{apdx:alinhac:R1:Vbar}.
\end{proof}
}
We show, in the following lemma, how the control on $\gradphi f^{\mathbb{\Lambda}}$ yields a control on $\nabla f$ in Sobolev norms. For this lemma we need to work in $H^{s}(S_r) := H^{s,s}(S_r)$ as we use commutator estimates in order to gain one derivative (see the estimates of Lemma \ref{apdx:commutator} and the remark just above).
\begin{lemma}
\label{lemma:alinhac:equiv}
\newcommand{\diff}{\mathbb{\Lambda}}
Let $s_0>d/2$, $s \in \N$ with $s \geq s_0 + \frac32$. Let $\eta \in H^{s_0+\frac52,3}$ such that \eqref{hyp:h:bc:apdx:B} and \eqref{hyp:M:eta:apdxB} (with $k=s$), for a constant $M > 0$, hold. Then there exists $C>0$ such that the following holds. We denote by $\mathcal{D} := \{ \Lambda^{s-l} \partial_r^l , l \in \{ 1,\dots,s \} \} \cup \{ |D|^2\Lambda^{s-2} \}$ the set of the differential operators that will be used in this lemma.\\
		Then there exists a constant $C$ depending only on $M$  such that the following holds:
		\begin{itemize}[label=\textbullet]
		\item If \nouveau{$f \in \dot{H}^{1}/\R$ }then there exists $c$ depending only on $s_0,h_*$ such that
		\begin{equation}
		\label{apdx:alinhac:equiv_1}
		\Vert \gradphi f \Vert_{L^2} \leq c ( 1 + \epsilon \Vert \eta \Vert_{H^{s_0+\frac32,2}} ) \Vert \nabla f \Vert_{L^2} .
		\end{equation}
		
		Conversely,
		\begin{equation}
		\label{apdx:alinhac:equiv_2}
		\Vert \nabla f \Vert_{L^2} \leq c ( 1 + \epsilon \Vert \eta \Vert_{H^{s_0+\frac32,2}} ) \Vert \gradphi f \Vert_{L^2} .
		\end{equation}
		\item If $f \in H^{s}$,  
		\begin{equation}
			\label{apdx:alinhac:basic}
			\sys{ \Vert f \Vert_{H^{s}} &\leq C \left( \sum_{ \diff \in \mathcal{D}} \Vert f^{\diff} \Vert_{L^2} + \Vert f \Vert_{H^{s_0+\frac32,2}} \right), \\
			 \sum_{\diff \in \mathcal{D}} \Vert f^{\diff} \Vert_{L^2} + \Vert f \Vert_{H^{s_0+\frac32,2}} &\leq C  \Vert f \Vert_{H^{s}}. \\	}
		\end{equation}
		\item If \nouveau{$f \in \dot{H}^1/\R$ with $\nabla f \in H^{s}$,}
		\begin{equation}
		\label{apdx:alinhac:grad}
		 \sys{ \Vert \nabla f \Vert_{H^{s_0+\frac32,2}} + \sum_{\diff \in \mathcal{D} } \Vert \nabla f^{\diff} \Vert_{L^2}  &\leq C \left(\Vert \gradphi f \Vert_{H^{s}} + \Vert \nabla f \Vert_{H^{s_0+\frac32,2}} \right),\\
				\Vert \gradphi f \Vert_{H^{s}} + \Vert \nabla f \Vert_{H^{s_0+\frac32,2}} & \leq C \left( \Vert \nabla f \Vert_{H^{s_0+\frac32,2}} + \sum_{\diff \in \mathcal{D} } \Vert \nabla f^{\diff} \Vert_{L^2} \right). }
		\end{equation}
		\end{itemize}
		
\end{lemma}
\begin{proof}
\newcommand{\diff}{\mathbb{\Lambda}}
For inequalities \eqref{apdx:alinhac:equiv_1} and \eqref{apdx:alinhac:equiv_2}, we use the definition of $\gradphi$ and the product estimate \eqref{apdx:pdt:tame}. \\
We now write   $ \diff$ an operator of the form $\Lambda^{s-l} \partial_r^l$ for $1 \leq l \leq s$ or $|D|^2 \Lambda^{s-2}$ (with $l=0$ by convention). \\
For the inequalities in \eqref{apdx:alinhac:basic}, we use the definition of Alinhac's good unknown \eqref{apdx:alinhac:definition}
$$f^{\diff} = \diff f + \epsilon \frac{\diff \eta}{1+\epsilon h} \partial_r f.$$
We take the $L^2-$norm of this equation, use triangular inequality together with the product estimate \eqref{apdx:pdt:tame} and $1+\epsilon h \geq h_*$
$$\sum_{\diff \in \mathcal{D}} \Vert f^{\diff} \Vert_{L^2} \leq c \Vert f \Vert_{H^{s}} + c \epsilon \Vert \eta \Vert_{H^{s}} \Vert f \Vert_{H^{s_0+\frac32,2}},$$
where $c$ only depends on $s,h_*$. This yields the result, up to a last detail : in the case $l=0$, we use $|D|^2 \Lambda^{s-2}$, so we lack the control of the $L^2$-norm. However the term $\Vert f \Vert_{H^{s_0+\frac32,2}}$ provides this control, and by interpolation we get the result. The converse inequality is proven the same way. \\ 
For the last two equations, we use equation \eqref{apdx:alinhac:comm1} of Lemma \ref{apdx:alinhac:def} to write
$$\diff  \gradphi  f = \gradphi  f^{\diff} +  \epsilon (\diff \eta \partial_r^{\varphi} )\gradphi f + \epsilon \Ral{1}{\nabla f}. $$
We take the $L^2$-norm of this expression. We use the product estimate \eqref{apdx:pdt:tame} for the second term and the estimate \eqref{apdx:alinhac:R1} for the last term:
$$\Vert \diff \gradphi  f \Vert_{L^2} \leq C \Vert \gradphi  f^{\diff} \Vert_{L^2} + \epsilon \Vert \eta \Vert_{H^{s}} \Vert \gradphi f \Vert_{H^{s_0+\frac32,2}} + C \epsilon \Vert \eta \Vert_{H^{s,2\vee (s-1) +1}} \Vert \gradphi f \Vert_{H^{s-1,2\vee (s-1) }}.$$
Note that here we use $k=s$ so that $k\wedge(s-1) = k-1$ (with the notation of Lemma \ref{apdx:commutator}). Adding $\Vert \nabla f\Vert_{H^{s_0+\frac32,2}}$ to the previous inequality  and summing the results for $\diff$ in $\mathcal{D}$ yields the inequality 
\begin{equation}
\label{eqn:apdxB:temp}
\Vert \gradphi  f \Vert_{H^{s}} + \Vert \nabla f\Vert_{H^{s_0+\frac32,2}} \leq C \left( \sum_{\diff \in \mathcal{D}} \Vert \gradphi  f^{\diff} \Vert_{L^2} + \Vert \nabla f\Vert_{H^{s_0+\frac32,2}} \right) + C \Vert \gradphi f \Vert_{H^{s-1,2\vee (s-1) }}.
\end{equation}
\nouveau{The last term can be interpolated between the high regularity norm $\Vert \gradphi  f \Vert_{H^{s}}$ and the low regularity norm $\Vert \gradphi  f \Vert_{L^2}$. Then, using Young's inequality, the high regularity norm can be absorbed by the first term on the right-hand side of \eqref{eqn:apdxB:temp} and the low regularity norm is bounded from above by $\Vert \nabla f\Vert_{H^{s_0+\frac32,2}}$.} \\
The converse inequality is proven in the same way. 
\end{proof}

We can (and need to) treat the case of no vertical derivatives (i.e. $H^{s,0}$ instead of $H^s := H^{s,s}$), which we do in the following lemma.
\begin{lemma}
\newcommand{\diff}{\mathbb{\Lambda}}
Let $s_0>d/2$, $s \in \R$ with $s > s_0 + \frac32$. Let $\eta \in H^{s_0+\frac52,3}$ such that \eqref{hyp:h:bc:apdx:B} and \eqref{hyp:M:eta:apdxB} (with $k=s$), for a constant $M > 0$, hold. Then there exists a constant $C$ depending only on $M$  such that for \nouveau{$f \in \dot{H}^1/\R$ with $\nabla f \in H^{s,0} \cap H^{s_0+\frac32,2}$},
		\begin{equation}
		\label{apdx:alinhac:grad:horizontal}
		 \sys{ \Vert \nabla f \Vert_{H^{s_0+\frac32,2}} + \Vert \nabla f^{\diff} \Vert_{L^2}  &\leq C \left(\Vert \gradphi f \Vert_{H^{s,0}} + \Vert \nabla f \Vert_{H^{s_0+\frac32,2}} \right),\\
				\Vert \gradphi f \Vert_{H^{s,0}} + \Vert \nabla f \Vert_{H^{s_0+\frac32,2}} & \leq C \left( \Vert \nabla f \Vert_{H^{s_0+\frac32,2}} + \Vert \nabla f^{\diff} \Vert_{L^2} \right). }
		\end{equation}
\end{lemma}
The proof of \eqref{apdx:alinhac:grad:horizontal} is a slight adaptation of \eqref{apdx:alinhac:grad} using \eqref{apdx:commutator_sym:horizontal} and we omit it.
\section{Some useful estimates} 
\label{appendixC}
We first show how to use the incompressibility constraint in order to estimate $\Vert w \Vert_{H^s}$ (for $s$ large enough) by terms depending only on $V$ and $\eta$, but with the loss of one derivative.
\begin{lemma}
\label{apdx:lemma:w_trick}
Let $s \geq s_0+\frac32$ and $(V,w,\eta) \in \left( H^s \right)^{d+2}$ satisfying \eqref{hyp:h}, the incompressibility constraint in \eqref{eqn:euler_slag}, as well as the boundary conditions \eqref{eqn:euler:bc}. Then, there exists \nouveau{$C > 0$ that only depends on an upper bound on $\Vert \eta \Vert_{H^s}$} such that
\begin{equation}
\label{apdx:eqn:w_trick}
\Vert w \Vert_{H^{s-1}} \leq C\Vert V \Vert_{H^s}(1 + \Vert \eta \Vert_{H^s}) + C \Vert \eta \Vert_{H^s} | \Vb |_{W^{s,\infty}}.
\end{equation}
\end{lemma}
\begin{proof}
Recall the incompressibility constraint
$$ \gradphi[x] \cdot V + \partial_r^{\varphi} w + \epsilon \frac{\nabla_x \eta}{1+\epsilon h} \partial_r \Vb=0,$$
with the definition \eqref{eqn:def:gradphi}. We thus write
$$ w(t,x,r) = \int_0^r \partial_r w(t,x,r')dr' = \int_0^r (1+\epsilon h) \left( \gradphi[x] \cdot V + \frac{\nabla_x \eta}{1+\epsilon h} \partial_r \Vb \right).$$
We thus have \eqref{apdx:eqn:w_trick}, by product estimates \cite[Lemma A.3]{Duchene2022} and \eqref{apdx:pdt:infty}.

\end{proof}
We now provide estimates on the non-linear term and buoyancy term $b$ in \eqref{eqn:euler_slag} and the elliptic equation on the pressure \eqref{eqn:elliptic}. Note that we also have, as in the Eulerian case,
	$$ \gradphi \cdot \NLun = \NLdeux,$$
where, on the right-hand side, a sum is taken over the indices $i,j$. \\
Note that the conditions on $s$ and $k$ in the following lemma are such that $H^{s,k}$ is an algebra.

	\begin{lemma}[Control of the non-linear term and the buoyancy term]
	\label{lemma:NL}
	Let $\mu,\epsilon \leq 1$, $M,\Mb > 0$, $s_0 > d/2$, $s \in \R$ with $s \geq s_0+\frac12$ and $1 \leq k \leq s$ with $k \in \N$. We make the assumption \eqref{hyp:h}, as well as $\epsilon \leq \sqrt{\mu}$ and
	$$\sys{
 | \varrho|_{W_r^{s,\infty}} + | \Vb |_{W_r^{k+1, \infty}} &\leq \Mb, \\
 \Vert \eta \Vert_{H^{s +1 ,k +1 }} + \Vert V \Vert_{H^{s +1 ,k +1}} + \sqrt{\mu}\Vert w \Vert_{H^{s +1 , k +1}}  &\leq M.}$$
	Then there exists a constant $C > 0$ depending only on $\Mb,M$ such that
	\begin{equation}
	\label{apdx:NL:1}
\sys{ \Vert  \left( \Ub + \epsilon \vec{U}\right) \cdot \nabla^{\varphi} \left( \Ub + \epsilon \vec{U}\right) \Vert_{H^{s,k}} &\leq \frac{\epsilon}{\sqrt{\mu}} C\left( \Vert \eta \Vert_{H^{s +1 ,k +1 }} + \Vert V \Vert_{H^{s +1 ,k +1}} + \sqrt{\mu}\Vert w \Vert_{H^{s +1 , k +1}} \right) ,\\
	 \Vert \NLdeux \Vert_{H^{s,k}} &\leq \frac{\epsilon}{\sqrt{\mu}} C\left( \Vert \eta \Vert_{H^{s +1 ,k +1 }} + \Vert V \Vert_{H^{s +1 ,k +1}} + \sqrt{\mu}\Vert w \Vert_{H^{s +1 , k +1}} \right). }
	\end{equation}
	We also have a similar estimate, that is uniform in $\mu$ but with a loss of derivatives:
	\begin{equation}
	\label{apdx:NL:2}
	\Vert \NLdeux \Vert_{H^{s,k}} \leq \epsilon C\left( \Vert \eta \Vert_{H^{s +2 ,k +2 }} + \Vert V \Vert_{H^{s +2 ,k +2}}\right).
	\end{equation}
	 We can also write
	\begin{equation}
	\label{apdx:NL:low_reg}	
	\Vert  \left( \Ub + \epsilon \vec{U}\right) \cdot \nabla^{\varphi} \left( \Ub + \epsilon \vec{U}\right) \Vert_{L^2} \leq \frac{\epsilon}{\sqrt{\mu}} C\left( \Vert \eta \Vert_{H^{s_0+\frac32 ,2 }} + \Vert V \Vert_{H^{s_0+\frac12 ,1}} + \sqrt{\mu}\Vert w \Vert_{H^{s_0+\frac12 ,1}} \right) .
	\end{equation}
We also write a bound on the buoyancy term
	\begin{equation}
	\label{apdx:NL:buoyancy}
	\Vert b  \Vert_{H^{s,k}} \leq C \Vert \eta \Vert_{H^{s\vee(s_0+\frac32),k\vee2}}.
	\end{equation}
\end{lemma}
	\begin{proof}
	We start with the horizontal component of the first equation in \eqref{apdx:NL:1}. We write 
		$$\begin{aligned}
		 &\left( \Ub + \epsilon \vec{U} \right) \cdot \nabla^{\varphi} \left( \Vb + \epsilon V \right) = \Vb \cdot \gradphi[x] \Vb + \epsilon \Vb \cdot \gradphi[x]V + \epsilon  \vec{U} \cdot \gradphi\left( \Vb + \epsilon V \right) \\
		&= \epsilon\frac{1}{1+\epsilon h} \left(\Vb \cdot \nabla_x \eta \right)\Vb' + \epsilon  \Vb \cdot \gradphi[x] V + \epsilon  V \cdot \gradphi[x] \left( \Vb + \epsilon V \right) + \epsilon  w  \partial_r^{\varphi} \left( \Vb + \epsilon V \right). \\
	\end{aligned} $$
Note the particular expression of the term $\gradphi[x] \Vb$ since $\Vb$ only depends on $r$. Using $\frac{1}{1+\epsilon h} = 1 - \epsilon  \frac{h}{1+\epsilon h}$, this yields
\begin{equation}
\label{eqn:NL_1}
\begin{aligned} 
\left( \Ub + \epsilon \vec{U} \right) \cdot \nabla^{\varphi} \left( \Vb + \epsilon V \right) &= \epsilon \left(\Vb \cdot \nabla_x \eta \right)\Vb' - \epsilon^2 \frac{h}{1+\epsilon h} \left(\Vb \cdot \nabla_x \eta \right)\Vb'+ \epsilon  \Vb \cdot \gradphi[x] V \\
&+ \epsilon  V \cdot \gradphi[x] \left( \Vb + \epsilon V \right) + \epsilon  w  \partial_r^{\varphi} \left( \Vb + \epsilon V \right). \end{aligned}
\end{equation}
We take the $H^{s,k}$-norm of the previous expression \eqref{eqn:NL_1}, with $s \geq s_0+\frac12$ and $k \geq 1$. We only treat the second and last terms, as the other ones can be treated the same way. Using the product estimates \cite[Lemma A.3]{Duchene2022} and \eqref{apdx:pdt:infty}, we bound the second term by
$$ \left\Vert \epsilon^2\frac{h}{1+\epsilon h} \left(\Vb \cdot \nabla_x \eta \right)\Vb' \right\Vert_{H^{s,k}} \leq \epsilon^2 c \left\Vert \frac{h}{1+\epsilon h} \right\Vert_{H^{s,k}} \Vert \nabla_x \eta \Vert_{H^{s,k}}, $$
where $c$ depends on $\Mb$. Using the composition estimate \cite[Lemma A.5]{Duchene2022} and the definition $h := - \partial_r \eta$, we can write
$$ \left\Vert \frac{h}{1+\epsilon h} \right\Vert_{H^{s,k}} \leq c \Vert \eta \Vert_{H^{s+1,k+1}}.$$
For the last term, we write, by definition of $\partial_r^{\varphi}$:
$$ \epsilon^2 w \partial_r^{\varphi} V = \nouveau{-}\epsilon^2 w \partial_r V \nouveau{+} \epsilon^3 w \frac{h}{1+\epsilon h} \partial_r V .$$
So that, by the product estimates \cite[Lemma A.3]{Duchene2022} and \eqref{apdx:pdt:infty}
$$ \begin{aligned}
\epsilon^2 \Vert w \partial_r^{\varphi} V\Vert_{H^{s,k}} &\leq c \epsilon^2 \Vert w \Vert_{H^{s,k}} \Vert \partial_r V \Vert_{H^{s,k}} \nouveau{+} \epsilon^3 \Vert w \Vert_{H^{s,k}} \left\Vert \frac{h}{1+\epsilon h} \right\Vert_{H^{s,k}} \Vert \partial_r V \Vert_{H^{s,k}} \leq \frac{\epsilon^2}{\sqrt{\mu}} C \left( \Vert V \Vert_{H^{s+1,k+1}} + \sqrt{\mu} \Vert w \Vert_{H^{s+1,k+1}} \right). \end{aligned}$$
Here, $C$ depends on $M$. We used the composition estimate \cite[Lemma A.5]{Duchene2022}. \\ Note that the factor $\frac{1}{\sqrt{\mu}}$ comes from the fact that the term $\Vert w \Vert_{H^{s+1,k+1}}$ in $M$ comes with a factor $\sqrt{\mu}$. In order to bound the last term in \eqref{eqn:NL_1}, it remains to estimate the following term:
$$ \epsilon w \partial_r^{\varphi} \Vb = \epsilon w \partial_r \Vb - \epsilon^2 w \frac{h}{1+\epsilon h} \partial_r \Vb. $$
We use product estimates \eqref{apdx:pdt:infty} and \cite[Lemma A.3]{Duchene2022} as well as the composition estimate \cite[Lemma A.5]{Duchene2022}, to get
$$ \Vert \epsilon w \partial_r^{\varphi} \Vb\Vert_{H^{s,k}} \leq c \epsilon \Vert w \Vert_{H^{s,k}}  + c \epsilon^2 \Vert w \Vert_{H^{s,k}} \Vert h \Vert_{H^{s,k}} .$$
The vertical component of the first estimate in \eqref{apdx:NL:1} reads
$$ \epsilon \left( \Ub + \epsilon \vec{U} \right) \cdot \nabla^{\varphi} w = \epsilon \Vb \cdot \gradphi[x] w + \epsilon^2 V\cdot \gradphi[x] w + \epsilon^2 w \partial_r^{\varphi} w $$
The first two terms are treated as before. For the last term, the product estimate \cite[Lemma A.3]{Duchene2022} and the composition estimate \cite[Lemma A.5]{Duchene2022} yield
$$ \Vert \epsilon^2 w \partial_r^{\varphi} w \Vert_{H^{s,k}} \leq \frac{\epsilon^2}{\mu}  C ( \sqrt{\mu} \Vert w \Vert_{H^{s+1,k+1}})^2, $$
which is bounded by $\frac{\epsilon}{\sqrt{\mu}} CM$ thanks to the assumption $\epsilon \leq \sqrt{\mu}$. \\
For the second estimate in \eqref{apdx:NL:1} and \eqref{apdx:NL:2}, we rather prove
\begin{equation}
\label{eqn:NL:temp:1}
 \Vert \NLdeux \Vert_{H^{s,k}} \leq \epsilon C\left( \Vert \eta \Vert_{H^{s +1 ,k +1 }} + \Vert V \Vert_{H^{s +1 ,k +1}} + \Vert w \Vert_{H^{s +1 , k +1}} \right),
\end{equation}
where $C$ only depends on $M,\Mb$. Then, making $\sqrt{\mu} \Vert w \Vert_{H^{s +1 , k +1}}$ appear, we get the second estimate in \eqref{apdx:NL:1}. On the other hand, from \eqref{eqn:NL:temp:1} and using \eqref{apdx:eqn:w_trick}, we get \eqref{apdx:NL:2}.\\

We write
	\begin{equation}
	\label{eqn:NL_2} \begin{aligned} \NLdeux = &\frac{\epsilon^2}{(1+\epsilon h)^2} \sum\limits_{i,j \leq d} \partial_i \eta \Vb_j' \partial_j \eta \Vb_i' + 2 \epsilon^2\frac{1}{1+\epsilon h} \sum\limits_{i,j \leq d} \partial_i \eta \Vb_j' \partial_j^{\varphi} \vec{U}_i \\
	&- 2 \epsilon  \frac{1}{1+\epsilon h} \sum\limits_{j \leq d+1}  \Vb_j' \partial_j^{\varphi} w + \epsilon^2 \sum\limits_{i,j \leq  d+1} \partial_i^{\varphi} \vec{U}_j \partial_j^{\varphi} \vec{U}_i.
	\end{aligned}
	\end{equation}
	Each term in \eqref{eqn:NL_2} is treated as before with product estimates \cite[Lemma A.3]{Duchene2022} and \eqref{apdx:pdt:infty} as well as the composition estimate \cite[Lemma A.5]{Duchene2022}. For the last term in \eqref{eqn:NL_2}, we need the assumption $\epsilon \leq \sqrt{\mu}$, as for the last term in \eqref{eqn:NL_1}.\\
	
	For the last estimate \eqref{apdx:NL:low_reg}, the proof is very similar and we only treat one term for the sake of conciseness. We write, thanks to Hölder inequality and embedding of \cite[Lemma A.1]{Duchene2022} 
	$$ \begin{aligned} \left\Vert V \cdot \left( \frac{\nabla_x \eta}{1+\epsilon h} \partial_r V \right)\right\Vert_{L^2} & \leq \Vert V \Vert_{L^{\infty}} \Vert \nabla_x \eta \Vert_{L^{\infty}} \left\Vert \frac{1}{1+\epsilon h} \right\Vert_{L^{\infty}} \Vert \partial_r V \Vert_{L^2} \\
	&\leq C \Vert V \Vert_{H^{s_0+\frac12,1}} \Vert \eta \Vert_{H^{s_0+\frac32,2}} \Vert V \Vert_{H^{1,1}} \leq C \Vert V \Vert_{H^{s_0+\frac12,1}}.
	\end{aligned} $$
For the last inequality, we used Assumption \eqref{hyp:h}, as well as $s_0+\frac12 \geq \frac{d}{2}+\frac12 \geq 1$.\\
\nouveau{For \eqref{apdx:NL:buoyancy}, we use the definition \eqref{eqn:def:b} and the composition estimate \cite[Lemma A.5]{Duchene2022}.}
	\end{proof}
\section*{Acknowledgments}
The author thanks his PhD supervisors Vincent Duchêne and David Lannes for their valuable advice, \nouveau{as well as the anonymous referees for their detailed comments, which significantly improved the quality of the present manuscript.} \\
This work was supported by the BOURGEONS project, grant ANR-23-CE40-0014-01 of the French National Research Agency (ANR).\\
This work is licensed under \href{https://creativecommons.org/licenses/by/4.0/}{ CC BY 4.0}.
\bibliographystyle{amsalpha}
\bibliography{../../../Seafile/refs/refs}
\end{document}